\documentclass[10pt]{amsart}
\usepackage{amscd}
\usepackage{amsmath}
\usepackage{amsfonts}
\usepackage{amssymb}
\usepackage{amsthm}
\usepackage{enumerate}
\usepackage[T1]{fontenc}
\usepackage{hyperref}
\usepackage{mathrsfs}
\usepackage{stackrel}
\usepackage[all]{xy}
\usepackage{yhmath}

\makeatletter
\let\yhmath@wideparen\wideparen
\DeclareRobustCommand{\wideparen}[1]{{\mathpalette\inner@wideparen{#1}}}
\newcommand\inner@wideparen[2]{%
  \sbox\z@{$#1\m@th#2$}%
  \yhmath@wideparen{\box\z@}%
}
\makeatother

\newtheorem{theorem}{Theorem}[section]
\newtheorem{lemma}[theorem]{Lemma}
\newtheorem{definition}[theorem]{Definition}
\newtheorem{proposition}[theorem]{Proposition}
\newtheorem{corollary}[theorem]{Corollary}

\DeclareMathOperator{\Sp}{Sp}
\DeclareMathOperator{\im}{Im}

\DeclareMathOperator{\coim}{Coim}
\DeclareMathOperator{\Hom}{Hom}
\DeclareMathOperator{\Tor}{Tor}

\DeclareMathOperator{\Sym}{Sym}
\DeclareMathOperator{\an}{an}
\DeclareMathOperator{\rig}{rig}

\DeclareMathOperator{\Ban}{Ban}

\DeclareMathOperator{\Der}{Der}

\DeclareMathOperator{\ac}{ac}
\DeclareMathOperator{\Loc}{Loc}
\DeclareMathOperator{\res}{res}

\begin{document}
\title{A Proper Mapping Theorem for coadmissible $\wideparen{\mathcal{D}}$-modules}
\author{Andreas Bode}
\address{Andrew Wiles Building, Mathematical Institute, Radcliffe Observatory Quarter\\
Woodstock Road, Oxford, OX2 6GG}
\email{andreas.bode@maths.ox.ac.uk}
\maketitle

\begin{abstract}
We study the behaviour of $\wideparen{\mathcal{D}}$-modules on rigid analytic varieties under pushforward along a proper morphism.\\
We prove a $\wideparen{\mathcal{D}}$-module analogue of Kiehl's Proper Mapping Theorem, considering the derived sheaf-theoretic pushforward from $\wideparen{\mathcal{D}}_X$-modules to $f_*\wideparen{\mathcal{D}}_X$-modules for proper morphisms $f: X\to Y$. Under assumptions which can be naturally interpreted as a certain properness condition on the cotangent bundle, we show that any coadmissible $\wideparen{\mathcal{D}}_X$-module has coadmissible higher direct images. This implies among other things a purely geometric justification of the fact that the global sections functor in the rigid analytic Beilinson--Bernstein correspondence preserves coadmissibility, and we are able to extend this result to twisted $\wideparen{\mathcal{D}}$-modules on analytified partial flag varieties.\\
\end{abstract}
\section{Introduction}
Let $K$ be a complete discretely valued field of mixed characteristic $(0,\ p)$, with discrete valuation ring $R$ and uniformizer $\pi$, and let $f: X\to Y$ be a proper morphism of rigid analytic $K$-varieties. \\
Recall Kiehl's Proper Mapping Theorem for coherent $\mathcal{O}$-modules.
\begin{theorem}[{see \cite{Kiehl}}]
\label{Kiehlthm}
If $\mathcal{M}$ is a coherent $\mathcal{O}_X$-module, then $\mathrm{R}^jf_*\mathcal{M}$ is a coherent $\mathcal{O}_Y$-module for each $j\geq 0$.
\end{theorem}
The main goal of this paper is to prove a noncommutative analogue of Theorem \ref{Kiehlthm}, considering coadmissible $\wideparen{\mathcal{D}}$-modules.\\
The sheaf $\wideparen{\mathcal{D}}_X$ of analytic differential operators on $X$ was introduced by Ardakov--Wadsley in \cite{Ardakov1}. This sheaf consists of differential operators, possibly of infinite order, with rapidly decreasing coefficients. It was shown in \cite{Bodecompl} that for any smooth rigid analytic $K$-variety $U$, $\wideparen{\mathcal{D}}_U$ is a full Fr\'echet--Stein sheaf (see section 3 for precise definitions), which allows us to consider the category of coadmissible modules, a natural analogue of coherent modules in this setting.\\
We present in this paper several conditions one can impose to guarantee that the derived pushforward functors $\mathrm{R}^jf_*$ preserve coadmissibility.\\
In this introduction, we say that a Lie algebroid $\mathscr{L}$ on $X$ is locally free relative to $Y$ if there exists an admissible covering $(Y_i)$ of $Y$ such that the restriction $\mathscr{L}|_{X_i}$ to each $X_i=f^{-1}Y_i$ is a free $\mathcal{O}_{X_i}$-module for each $i$.
\begin{theorem}
\label{intromainthm}
Let $f: X\to Y$ be a proper morphism of rigid analytic $K$-varieties, where $X$ is smooth. 
\begin{enumerate}[(i)]
\item Suppose the tangent sheaf $\mathcal{T}_X$ is locally free relative to $Y$. Then $f_*\wideparen{\mathcal{D}}_X$ is a full Fr\'echet--Stein sheaf on $Y$, and if $\mathcal{M}$ is a coadmissible $\wideparen{\mathcal{D}}_X$-module then $\mathrm{R}^jf_*\mathcal{M}$ is a coadmissible $f_*\wideparen{\mathcal{D}}_X$-module for each $j\geq 0$.
\item Suppose $\mathcal{T}_X$ is the quotient of a Lie algebroid $\mathscr{L}$ on $X$ which is locally free relative to $Y$. Then $f_*\wideparen{\mathcal{D}}_X$ is a Fr\'echet--Stein sheaf on $Y$, and if $\mathcal{M}$ is a coadmissible $\wideparen{\mathcal{D}}_X$-module then $\mathrm{R}^jf_*\mathcal{M}$ is a coadmissible $f_*\wideparen{\mathcal{D}}_X$-module for each $j\geq 0$.
\end{enumerate}
\end{theorem} 
We obtain as a consequence the following corollary.
\begin{corollary}
\label{intropropervar}
Let $X$ be a smooth proper rigid analytic variety over $K$. If $\mathcal{T}_X$ is generated by global sections then $\wideparen{\mathcal{D}}_X(X)$ is a Fr\'echet--Stein algebra, and if $\mathcal{M}$ is a coadmissible $\wideparen{\mathcal{D}}_X$-module then $\mathrm{H}^j(X, \mathcal{M})$ is a coadmissible $\wideparen{\mathcal{D}}_X(X)$-module for each $j\geq 0$.
\end{corollary}
We make a couple of remarks.
\begin{enumerate}[(i)]
\item Note that in Theorem \ref{intromainthm}.(ii), the sheaf $f_*\wideparen{\mathcal{D}}_X$ is not claimed to be a \emph{full} Fr\'echet--Stein sheaf, but only a Fr\'echet--Stein sheaf, meaning that the desired properties need not hold on any admissible open affinoid subspace, but only on a certain base of the topology, see section 3. This is a familiar feature already occuring in results in \cite{Ardakov1}. While we were able to overcome these issues pertaining to the results in \cite{Ardakov1} with our paper \cite{Bodecompl}, our situation here seems to be more complicated. In fact, the difficulties do not just arise from possible $\pi$-torsion in certain formal models, but rather from the limited scope of Lemma \ref{passtoFS}. 
\item The condition of being locally free relative to $Y$ imposed in Theorem \ref{intromainthm} ensures that we can lift $f$ to a proper morphism between certain vector bundles. As we can think intuitively of $\wideparen{\mathcal{D}}_X$ as a noncommutative analogue of functions on the cotangent bundle $T^*X$, our result is strictly speaking not in parallel with Kiehl's Theorem applied to $f$, but rather (in the case of Theorem \ref{intromainthm}.(i)) to the induced map $T^*X\to g_*T^*X$, where $g: X\to Z$ is the first map in the Stein factorization of $f$, and $g_*T^*X$ is the vector bundle (dually) associated to $g_*\mathcal{T}_X$, which is locally free by assumption. This map is proper as it is locally of the form $X\times \mathbb{A}^{n, \an}\to Z\times \mathbb{A}^{n, \an}$ by assumption. A similar description works for (ii), see Proposition \ref{vbinterpretation}.
\item The results given in this paper are actually more general than Theorem \ref{intromainthm}. We consider full Fr\'echet--Stein sheaves which are in a `natural' way coadmissible over $\wideparen{\mathscr{U}(\mathscr{L})}$, where $\mathscr{L}$ is some Lie algebroid on $X$ which is locally free relative to $Y$. This setup allows us to extend results to twisted $\wideparen{\mathcal{D}}$-modules, see below.
\item It is obviously necessary to work with $f_*\wideparen{\mathcal{D}}_X$ rather than the sheaf $\wideparen{\mathcal{D}}_Y$, as can be seen by just considering the point $Y=\Sp K$. This is due to the fact that we are considering the sheaf-theoretic pushforward $f_*$ rather than a $\wideparen{\mathcal{D}}$-module pushforward via transfer bimodules, which hasn't been developed yet beyond the case of a closed embedding (see \cite{Ardakov2}). We hope that our result can actually be used as a stepping stone to establishing a theory of $\wideparen{\mathcal{D}}$-module pushforwards more generally, with an analogue to the Proper Mapping Theorem (i.e.\ a $p$-adic analytic version of \cite[Theorem 2.5.1]{Hotta}) as a natural consequence.
\item The main motivation for this work (and a justification for using $f_*$) comes from the discussion of a rigid analytic Beilinson--Bernstein correspondence in \cite{Ardakovequiv}. If $X=(\mathbf{G}/\mathbf{B})^{\an}$ is the analytified flag variety of some split reductive algebraic group $\mathbf{G}$ over $K$, Corollary \ref{intropropervar} states that $\wideparen{\mathcal{D}}_X(X)$ is a Fr\'echet--Stein algebra, and that the global sections of any coadmissible $\wideparen{\mathcal{D}}_X$-module are coadmissible. Also note that in this case, the geometric picture given in remark (ii) translates to the properness of the moment map. We thus recover part of the statement of \cite[Theorem 6.4.7]{Ardakovequiv} (without equivariance) by purely geometric means, and are able to extend this straightforwardly to arbitrary coadmissible \emph{twisted} $\wideparen{\mathcal{D}}$-modules on \emph{partial} flag varieties. 
\end{enumerate}

\begin{corollary}
\label{intropartialfv}
Let $\mathbf{G}$ be a split reductive affine algebraic group scheme over $K$, $\mathbf{P}$ a parabolic subgroup scheme and $X=(\mathbf{G}/\mathbf{P})^{\an}$ the analytification of the partial flag variety. Let $\mathfrak{g}$ be the Lie algebra of $G=\mathbf{G}(K)$ and $\mathfrak{h}$ a Cartan subalgebra. If $\mathcal{M}$ is a coadmissible $\wideparen{\mathcal{D}}_X^{\lambda}$-module for some $\lambda\in \mathfrak{h}^*$, then $\mathrm{R}^j\Gamma(X, \mathcal{M})$ is a coadmissible $\wideparen{U(\mathfrak{g})}_{\lambda}$-module for each $j\geq 0$.
\end{corollary}
All relevant definitions will be given in section 6.\\
\\
We now give a brief overview of the structure of the paper.\\
In section 2, we distill those parts of the original proof of Theorem \ref{Kiehlthm} which can naturally be adapted to our situation: Schwartz' Theorem for strictly completely continuous morphisms and a finiteness result for cohomology groups due to Cartan--Serre. We will be working with a certain class of Noetherian Banach $K$-algebras which we call strictly NB, which includes both affinoid $K$-algebras and the algebras $\widehat{U(\pi^n\mathcal{L})}_K$ involved in the definition of $\wideparen{\mathcal{D}}_X$.\\
We also recall some results from \cite{Bodecompl} concerning completed tensor products.\\
\\
In section 3, we recall the basic theory of $\wideparen{\mathcal{D}}_X$-modules, using the more general language of Fr\'echet completed enveloping algebras for Lie algebroids, as in \cite{Ardakov1}. In order to capture sheaves like $\wideparen{\mathcal{D}}^{\lambda}$, we introduce the even more general notion of a Fr\'echet--Stein sheaf. We then establish enough terminology to state the main result (for a free Lie algebroid, Theorem \ref{KiehlDXfree}) and reduce all claims to statements about strictly NB $K$-algebras and finitely generated modules over them.\\
\\
Sections 4 and 5 then deal with the proof of Theorem \ref{KiehlDXfree}.\\
Taking $Y$ to be a suitable affinoid, the proof can be split into two parts: one statement about global sections (section 4) and one about localization (section 5). Having reduced to strictly NB $K$-algebras, the global sections part becomes quite straightforward, using section 2. This means that the proof of this part is almost entirely analogous to the original proof in \cite{Kiehl}. The only subtlety lies in ensuring that the finitely generated pieces we exhibit match up in the right way to produce a coadmissible module. This can be dealt with using the results on completed tensor products developed in \cite{Bodecompl}, which also provide the main tools for the arguments in section 5.\\
\\
In section 6, we formulate more general versions and variants of our theorem. We then go on to discuss a geometric interpretation of our results to provide more intuition for the conditions we impose, and describe several examples. We conclude with our main application, twisted $\wideparen{\mathcal{D}}$-modules on partial flag varieties, Corollary \ref{intropartialfv} and some generalizations of it.\\
\\
As already mentioned, we hope that it will be possible in the future to employ our results in order to study a $\wideparen{\mathcal{D}}$-module pushforward using transfer bimodules (see \cite[1.5]{Hotta}) -- for this, it seems necessary to work in a larger category than that of coadmissible modules in order to allow for a `derived' picture. We suspect that the quasi-abelian category of $\wideparen{\mathcal{D}}$-modules whose sections are complete bornological of convex type might be a suitable framework, as was indicated in \cite{ArdakovBen} and \cite{Bambozzi}. Once a $\wideparen{\mathcal{D}}$-module pushforward is in place, a corresponding Proper Mapping Theorem (at least for projective morphisms) should be a fairly straightforward consequence of our results, as we can consider in turn the cases of closed embeddings and projections, both of which are dealt with in this paper -- see Theorem \ref{derhamcoad} and the remarks following it.\\
\\
On the representation theoretic side, our discussion of pushforwards between partial flag varieties strongly suggests a theory of intertwining operators analogous to \cite{BBCass}.\\
\\
The results in this paper form part of the author's PhD thesis, which was produced under the supervision of Simon Wadsley. We would like to thank him for his encouragement and patience. 
\subsection*{Notation}
Throughout, $K$ is a complete nonarchimedean, discretely valued field of mixed characteristic $(0, \ p)$, with discrete valuation ring $R$ and uniformizer $\pi$. \\
Given a semi-normed $K$-vector space $V$, we denote by $V^\circ$ the unit ball of all elements in $V$ with semi-norm $\leq 1$. We define the value set of $V$ to be the set $|V|\setminus\{0\}$. For instance, the value set of $K$ is $|K^*|=|\pi|^{\mathbb{Z}}$.\\
\\
A normed $K$-algebra $A$ is always required to have a submultiplicative norm, so that $A^\circ$ is always a subring. Similarly, a normed $A$-module is a normed $K$-vector space $M$ with an $A$-module structure satisfying $|am|\leq |a|\cdot |m|$ for all $a\in A$, $m\in M$. In particular, $M^\circ$ is an $A^\circ$-module.\\
We denote the completion of a semi-normed $K$-vector space $V$ by $\widehat{V}$. We also write $\widehat{M}$ for the $\pi$-adic completion of an $R$-module $M$, but it should always be clear from context which completion we are using. We sometimes shorten $M\otimes_R K$ to $M_K$.\\
If $i=(i_1, \dots, i_m)\in \mathbb{N}^m$ is a multi-index, we write $|i|=i_1+i_2+\dots + i_m$, and abbreviate the expression
\begin{equation*}
X_1^{i_1}X_2^{i_2} \dots X_m^{i_m}
\end{equation*}
to $X^i$.\\
We denote by $T_m=K\langle X_1, \dots, X_m\rangle$ the $m$th Tate algebra over $K$, given by converging power series
\begin{equation*}
K\langle X\rangle=\left\{\sum_{i\in \mathbb{N}^m} a_iX^i: \ a_i\in K, |a_i|\to 0 \ \text{as} \ |i|\to \infty\right\}
\end{equation*}
Given an affinoid $K$-variety $X=\Sp A$, we let $X_w$ denote the weak Grothendieck topology (consisting of affinoid subdomains, with finite coverings by affinoid subdomains as coverings) and $X_{\rig}$ the strong Grothendieck topology (admissible open subspaces and admissible coverings, see \cite[Definition 5.1/4]{Bosch}).
  
\section{Background}
\subsection{Schwartz' Theorem and the Cartan--Serre argument}
We begin by discussing those parts of the argument in \cite{Kiehl} (see \cite[sections 6.3, 6.4]{Bosch} for an account in English) which lend themselves to generalization to the noncommutative setting. \\
\\
Throughout, $A$ will be a (not necessarily commutative) unital left Noetherian Banach $K$-algebra, whose norm is determined by an $R$-algebra $A^\circ$ as its unit ball, which we assume to be also left Noetherian. We summarize this by saying that $A$ is a \textbf{strictly Noetherian Banach (NB) $K$-algebra}.\\
By considering the gauge norm associated with $A^\circ$ (see \cite[Lemma 2.2]{SchneiderNFA}), we can assume that $|A|\setminus \{0\}=|K^*|$.\\
Note affinoid $K$-algebras equipped with a residue norm are obvious examples of strictly NB algebras: they are Noetherian Banach by \cite[Propositions 3.1/3.(i), 3.1/5.(ii)]{Bosch}, and the unit ball of any residue norm is Noetherian as long as $K$ is discretely valued by \cite[Remark 7.3/1]{Bosch}.
\begin{lemma}
\label{complsNB}
Let $\mathcal{A}$ be a left Noetherian $R$-algebra containing $R$. \\
Then $\widehat{\mathcal{A}}_K=\widehat{\mathcal{A}}\otimes_R K$ carries a natural structure of a strictly NB $K$-algebra.
\end{lemma}
\begin{proof}
Note that the $K$-algebra $\mathcal{A}_K$ is naturally equipped with a gauge semi-norm (see \cite[Lemma 2.2]{SchneiderNFA}) with unit ball $\overline{\mathcal{A}}=\mathcal{A}/\pi\mathrm{-tor}(\mathcal{A})$, and its completion is isomorphic to $\widehat{\overline{\mathcal{A}}}_K=\widehat{\overline{\mathcal{A}}}\otimes K$, a Banach $K$-algebra. By \cite[3.2.3.(iv)]{Berthelot}, $\widehat{\overline{\mathcal{A}}}$ injects into $\widehat{\overline{\mathcal{A}}}_K$ and can thus be naturally identified with the unit ball. But by \cite[3.2.3.(vi)]{Berthelot}, $\widehat{\overline{\mathcal{A}}}$ is left Noetherian, so $\widehat{\overline{\mathcal{A}}}_K$ is a left Noetherian Banach $K$-algebra with left Noetherian unit ball, and hence a strictly NB $K$-algebra.\\
Finally, the natural morphism $\widehat{\mathcal{A}}_K\to \widehat{\overline{\mathcal{A}}}_K$ is an isometric isomorphism of Banach $K$-algebras, by the same argument as in \cite[Lemma 2.5]{Ardakov1}.
\end{proof}
In this section, we will verify that Schwartz' theorem as given in \cite[Satz 1.2]{Kiehl} as well as the Cartan--Serre argument about finite cohomology groups (\cite[Proof of Satz 2.5]{Kiehl}, see \cite[Lemma 1.10]{Kedlaya} for another generalization) hold in the more general context of strictly NB $K$-algebras. All proofs will be essentially as in \cite{Kiehl}, except that some of our arguments become easier to formulate due to our assumptions on the field $K$ (note in particular that $R$ is always Noetherian in our setting). \\
\\
The module category we will be working with consists of all (left) Banach $A$-modules, together with continuous $A$-module morphisms. We call this category $\Ban_A$. \\
We recall the following facts.
\begin{enumerate}[(i)]
\item Since every $A$-module is also a $K$-vector space, an $A$-module morphism between normed $A$-modules is continuous if and only if it is bounded (see \cite[Proposition 3.1]{SchneiderNFA}).
\item Any surjection in $\Ban_A$ is open (Open Mapping Theorem, \cite[Proposition 8.6]{SchneiderNFA}).
\item Any finitely generated $A$-module is in $\Ban_A$, equipped with a canonical topology (see \cite[Proposition 3.7.3/3]{BGR}). Any $A$-module morphism between finitely generated $A$-modules is continuous with respect to the canonical topologies (see \cite[Proposition 3.7.3/2]{BGR}).
\item Given two objects $M$, $N$ of $\Ban_A$, their direct sum $M\oplus N$ carries the structure of a Banach $A$-module with respect to the max norm (see \cite[Definition 2.1.5/1, Proposition 2.5.1/6]{BGR}). 
\end{enumerate}
Note that for any $M, N\in \Ban_A$, the space of morphisms 
\begin{equation*}
\Ban_A(M, N)=\Hom_A^{\text{cts}}(M, N)
\end{equation*}
may be equipped with the supremum norm
\begin{equation*}
|f|_{\sup}:=\sup_{x\neq 0} \frac{|f(x)|}{|x|}.
\end{equation*}
This turns $\Hom_A^{\text{cts}}(M, N)$ into a Banach $K$-vector space by the same argument as in \cite[Proposition 3.3]{SchneiderNFA}.\\
When we speak of a sequence of morphisms $f_i$ converging to some $f \in \Ban_A(M, N)$, we mean uniform convergence, i.e. convergence with respect to the supremum norm. \\
\\
We also need to define topologically free modules. Given an indexing set $S$, consider the $A$-module
\begin{equation*}
\oplus_{s\in S} Ae_s,
\end{equation*}
equipped with the direct sum (maximum) norm, where $|ae_s|=|a|$ for each $s\in S$, $a\in A$. Its completion 
\begin{equation*}
F_S:=\widehat{\oplus}_{s\in S} Ae_s
\end{equation*}
lies in $\Ban_A$ and satisfies the following universal property.

\begin{proposition}
\label{univfree}
Given $M$ in $\Ban_A$ and any map $f: S\to M$ such that the set $\{|f(s)|: \ s\in S\}$ is bounded in $\mathbb{R}$, there exists a unique morphism $\phi: F_S\to M$ in $\Ban_A$ satisfying $\phi(e_s)=f(s)$ for each $s\in S$.\\
Moreover, the operator norm of $\phi$ is $|\phi|=\sup_{s\in S} |f(s)|$. 
\end{proposition}
\begin{proof}
By the universal property of (abstract) free modules, there exists a unique $A$-module morphism extending $f$, given by
\begin{align*}
\theta: &\oplus_{s\in S} Ae_i\to M\\
& \sum a_se_s\mapsto \sum a_sf(s).
\end{align*}
Moreover, $|\theta(\sum a_se_s)|=|\sum a_s f(s)|\leq \max |a_s| |f(s)|$ for any finite sum $\sum a_se_s$, so that $\theta$ is continuous by the boundedness assumption, with operator norm $|\theta|=\sup |f(s)|$. By continuity, $\theta$ extends uniquely to a continuous map $\phi$ between the completions $F_S\to M$, and $|\phi|=|\theta|$.
\end{proof}

We call $F_S$ the \textbf{topologically free module} over $S$ or the topologically free module (topologically) generated by $S$.\\
An obvious example is the $K$-algebra $A\langle X_1, \dots, X_n\rangle=A\widehat{\otimes}_K T_n$ for any $n\in \mathbb{N}$, which has a natural structure of a topologically free $A$-module over $\mathbb{N}^n$ as it is the completion of the polynomial algebra $A[X_1, \dots, X_n]$ with respect to the natural norm.\\
\\
The following corollary is a direct consequence of the proposition above.
\begin{corollary}
\label{freesurj}
For any $M \in \Ban_A$, there exists a topologically free module $F\in \Ban_A$ and a continuous surjection
\begin{equation*}
p: F\to M.
\end{equation*}
\end{corollary}

We now introduce a special kind of morphism in the category $\Ban_A$.
\begin{definition}
\label{scc}
A morphism $f: M\to N$ in $\Ban_A$ is called \textbf{strictly completely continuous} if $f$ is the limit of morphisms $f_i: M\to N$ in $\Ban_A$ such that $f_i(M^\circ)$ is a finitely generated $A^\circ$-module for each $i$. 
\end{definition}
It follows from Noetherianity of $A^\circ$ that this notion does not depend on a particular choice of norm on $M$, but only on its equivalence class. \\
We mention here that Kiehl phrases this definition differently in \cite[Definition 1.1]{Kiehl}, since he does not assume $K$ to be discretely valued (in particular, an affinoid $K$-algebra might have a unit ball which is not Noetherian). It is easy to check that the two definitions are equivalent in $\Ban_A$, where $A$ is some strictly NB algebra. \\
\\
We discuss one example which will feature in our proofs later.
\begin{lemma}
\label{sccexample}
Let $F=\widehat{\oplus}_S Ae_s$ be a topologically free $A$-module over $S$. \\
If $f: F\to M$ is a morphism in $\Ban_A$ such that for any $\epsilon>0$ there are only finitely many $s\in S$ with $|f(e_s)|\geq \epsilon$, then $f$ is strictly completely continous.
\end{lemma}
\begin{proof}
For any $\epsilon>0$, denote by $S_{\epsilon}$ the finite set of $s\in S$ such that $|f(e_s)|\geq \epsilon$.\\
Given $s\in S$, consider the continuous $A$-module morphism
\begin{align*}
g_s: & \hspace{0.5cm}F \longrightarrow \hspace{0.5cm} M\\
 &\sum a_je_j\mapsto a_sf(e_s), 
\end{align*}
i.e. we only consider the $e_s$ part of $f$. Now we set for any $n\in \mathbb{N}$
\begin{equation*}
f_n=\sum_{s\in S_{1/n}} g_s,
\end{equation*}
a continuous $A$-module morphism such that
\begin{equation*}
f_n(F^\circ)\subseteq \sum_{s\in S_{1/n}} A^\circ f(e_s)
\end{equation*}
is a finitely generated $A^\circ$-module. \\
It thus remains to show that the $f_n$ tend to $f$. By Proposition \ref{univfree}, $|f-f_n|= \sup_{s\in S} |f(e_s)-f_n(e_s)|$. Now if $s\in S_{1/n}$, then $f(e_s)=f_n(e_s)$, and if $s$ is not in $S_{1/n}$, then $f_n(e_s)=0$ and $|f(e_s)-f_n(e_s)|<1/n$ by construction. Thus $|f-f_n|< 1/n$, proving the result.
\end{proof}

\begin{corollary}
\label{sccpowerseries}
Let $f: A\langle x_1, \dots, x_n\rangle \to M$ be a morphism in $\Ban_A$ such that $f(x^i)$ tends to zero as $|i|\to \infty$. Then $f$ is strictly completely continuous.
\end{corollary}
We briefly record the following properties.
\begin{lemma}
\label{sccprop}
Let $f: M\to N$ be a strictly completely continuous morphism in $\Ban_A$, and let $L$, $G$ be in $\Ban_A$. Then the following holds:
\begin{enumerate}[(i)]
\item For any morphism $g: N\to G$ in $\Ban_A$, the composition $gf$ is strictly completely continuous.
\item For any morphism $h: L\to M$ in $\Ban_A$, the composition $fh$ is strictly completely continuous.
\end{enumerate}
\end{lemma}
\begin{proof}
Let ($f_i: M\to N$) be a sequence of morphisms in $\Ban_A$ as in Definition \ref{scc}. 
\begin{enumerate}[(i)]
\item For any continuous morphism $g$, the compositions $gf_i$ converge to $gf$, and since $f_i(M^\circ)$ is finitely generated, so is $gf_i(M^\circ)$: if $f_i(M^\circ)$ is generated by $n_1, \dots, n_r$, then $gf_i(M^\circ)$ is generated by $g(n_1), \dots, g(n_r)$.
\item Since $|(f-f_i)h|\leq |f-f_i|\cdot |h|$, we know that $f_ih$ converges to $fh$. Since $h$ is continuous, boundedness implies that there exists some integer $a$ such that
\begin{equation*}
h(L^\circ)\subseteq \pi^aM^\circ,
\end{equation*}
and thus $f_ih(L^\circ)$ is contained in $\pi^af_i(M^\circ)$, a finitely generated $A^\circ$-module by definition of the $f_i$ (multiplication by $\pi^a$ establishes an isomorphism $f_i(M^\circ)\cong \pi^af_i(M^\circ)$). By Noetherianity of $A^\circ$, $f_ih(L^\circ)$ is thus a finitely generated $A^\circ$-module. 
\end{enumerate}
\end{proof}
\begin{lemma}
\label{directsumscc}
Let $f_1:M_1\to N_1, \ \dots, \ f_r: M_r\to N_r$ be a finite set of strictly completely continuous morphisms in $\Ban_A$. Then the finite direct sum
\begin{equation*}
\oplus_{i=1}^r f_i: \oplus M_i\to \oplus N_i 
\end{equation*}
is also a strictly completely continuous morphism in $\Ban_A$.
\end{lemma}
\begin{proof}
As mentioned earlier, the modules $\oplus M_i$ and $\oplus N_i$ are in $\Ban_A$ and $\oplus f_i$ is a morphism in $\Ban_A$, as each $f_i$ is bounded.\\
For each $i$, let $f_i$ be the limit of $A$-module morphisms $g_{ij}$ such that $g_{ij}(M_i^\circ)$ is finitely generated for each $j\in \mathbb{N}$. Then clearly $\oplus f_i$ is the uniform limit of $(\oplus_i g_{ij})_j$, and moreover
\begin{equation*}
(\oplus_i g_{ij})(\oplus M_i^\circ)=\oplus_i g_{ij}(M_i^\circ)
\end{equation*}
is  a finitely generated $A^\circ$-module for any $j$, as required. 
\end{proof}

The class of strictly completely continuous morphisms is used in the proof of Kiehl's Proper Mapping Theorem by applying Theorem \ref{Schwartz}, which is known as Schwartz' Theorem. First, we need a definition.
\begin{definition}
\label{defcfi}
Let $N$ be an object of $\Ban_A$, and let $M$ be a submodule of $N$. We say $M$ is \textbf{closed and of finite index} in $N$ if $M$ is a closed submodule such that the quotient module $N/M$ is a finitely generated $A$-module.
\end{definition}
\begin{theorem}[{\cite[Satz 1.2]{Kiehl}}]
\label{Schwartz}
Let $f: M\to N$ be a surjection in $\Ban_A$, and let $g: M\to N$ be a strictly completely continuous homomorphism of $A$-modules. Then $\mathrm{Im}(f+g)$ is closed and of finite index in $N$. 
\end{theorem}

Before turning to the proof of Theorem \ref{Schwartz}, note that we have the following easy properties concerning submodules which are closed and of finite index.

\begin{lemma}
\label{cfiprop}
Let $N$ be in $\Ban_A$ and let $M$ be some $A$-submodule of $N$. Suppose there exists some morphism 
\begin{equation*}
f: N\to G
\end{equation*}
in $\Ban_A$ such that $f(M)$ is closed and of finite index in $G$, and $M$ contains the kernel of $f$. Then $M$ is closed and of finite index in $N$. 
\end{lemma}
\begin{proof}
By continuity of $f$, we know that $f^{-1}(f(M))$ is closed in $N$. But $f^{-1}(f(M))=M$, because $M$ contains the kernel of $f$. Moreover, as abstract $A$-modules we have isomorphisms
\begin{equation*}
N/M\cong (N/\ker f)/(M/\ker f)\cong f(N)/f(M)\leq G/f(M),
\end{equation*}
which is finitely generated by Noetherianity of $A$.
\end{proof}

\begin{lemma}
\label{towerofcfi}
Let $N$ be in $\Ban_A$ and let $M$ be some $A$-submodule of $N$. Suppose $M$ contains some $A$-module $M'$ which is closed and of finite index in $N$. Then $M$ is closed and of finite index in $N$.
\end{lemma}
\begin{proof}
Since $M'$ is closed in $N$, the quotient semi-norm on $N/M'$ is actually a complete norm, i.e. $N/M'$ equipped with the quotient norm is in $\Ban_A$. It follows from \cite[Proposition 3.7.3/3]{BGR} that this gives rise to the canonical topology on the finitely generated $A$-module $N/M'$.\\
Now apply the above lemma to the natural projection $\mathrm{pr}:N\to N/M'$, noting that $\mathrm{pr}(M)$ is closed in $N/M'$, as every $A$-submodule of a finitely generated $A$-module (with the canonical topology) is closed by \cite[Proposition 3.7.2/2]{BGR}, while finite generation of the quotient $(N/M')/\mathrm{pr}(M)$ follows directly from finite generation of $N/M'$. 
\end{proof}

\begin{lemma}
\label{projcfi}
Let $M$ and $N$ be modules in $\Ban_A$ such that $M$ is closed and of finite index in $N$. Let $f: N\to G$ be a surjection in $\Ban_A$. Then $f(M)$ is closed and of finite index in $G$.
\end{lemma}
\begin{proof}
By Lemma \ref{towerofcfi}, the submodule $M+\ker f$ is closed and of finite index in $N$. Since $f(M)=f(M+\ker f)$, we can assume without loss of generality that $\ker f\subseteq M$.\\
By the Open Mapping Theorem, $f$ is open. By assumption, the set complement $N\setminus M$ is open in $N$, so $f(N\setminus M)$ is open in $G$. But since $f$ is surjective and $\ker f\subseteq M$, we have
\begin{equation*}
f\left(N\setminus M\right)=G\setminus f(M),
\end{equation*} 
so that $f(M)$ is a closed submodule of $G$.\\
Moreover, we have the following isomorphisms as abstract $A$-modules
\begin{equation*}
G/f(M)\cong (N/\ker f)/(M/\ker f)\cong N/M,
\end{equation*}
which is finitely generated by assumption.
\end{proof}

The content of the following lemma can be summarized as: small continuous displacements of surjections are still surjective.
\begin{lemma}
\label{surjdisplace}
Let $f: M\to N$ be a surjection in $\Ban_A$. Then there exists a real number $c>0$ such that for any $\epsilon\in \Ban_A(M, N)$ with $|\epsilon|<c$ (again with respect to the supremum norm), the map $f-\epsilon$ is still surjective. 
\end{lemma}
\begin{proof}
This is exactly \cite[Lemma 1.3]{Kiehl}. The proof given there works for any strictly NB $K$-algebra. 
\end{proof}
\begin{proof}[Proof of Theorem \ref{Schwartz}]
We follow the argument in \cite[Satz 1.2]{Kiehl}.\\
Since $g$ is strictly completely continuous, we have a sequence of homomorphisms $g_i: M\to N$ converging to $g$ such that each $g_i(M^\circ)$ is a finitely generated $A^\circ$-module. Note in particular that for each $i$, the image $g_i(M)$ is a finitely generated $A$-module. \\
By Lemma \ref{surjdisplace}, we can choose $i$ large enough such that $f-(g_i-g)$ is surjective. We set $h=f-(g_i-g)$, and note that $f+g=h+g_i$.\\
Let $K=\ker g_i$, which is closed (by continuity of $g_i$) and of finite index in $M$, since $M/K\cong g_i(M)$ as abstract $A$-modules. Thus by surjectivity of $h$, Lemma \ref{projcfi} implies that $h(K)$ is closed and of finite index in $N$. But now $h(K)=(h+g_i)(K)$ by definition of $K$, and $(h+g_i)(K)$ is contained in $(h+g_i)(M)$. Thus by Lemma \ref{towerofcfi}, $(h+g_i)(M)$ is closed and of finite index in $N$, as required. 
\end{proof}
We can now straightforwardly generalize two results from \cite{Kiehl} regarding affinoid $K$-algebras $A$ to arbitrary strictly NB algebras.
\begin{theorem}[{see \cite[Satz 1.4]{Kiehl}}]
\label{freeprojscc}
Let $f: M\to N$ be a morphism in $\Ban_A$. Suppose that $N$ is a closed submodule of some $G\in \Ban_A$ via the injection $j: N\to G$ such that the composition $jf$ is strictly completely continuous. Then there exists a topologically free $A$-module $F$ and a surjection $p: F\to M$ in $\Ban_A$ such that $fp$ is strictly completely continuous. 
\end{theorem}
\begin{proof}
We only sketch the argument, as it is entirely analogous to \cite[Satz 1.4]{Kiehl}. By Corollary \ref{freesurj}, there exists a topologically free $A$-module $F$ and a continuous surjection $p:F\to M$ in $\Ban_A$. By Lemma \ref{sccprop}, the composition $jfp$ is strictly completely continuous, so replacing $M$ by $F$, we can assume that $M$ is topologically free, and we only need to show that $f$ is strictly completely continuous in that case.\\
\\
Write $M=\widehat{\oplus}_{s\in S} Ae_s$, $M^\circ=\widehat{\oplus} A^\circ e_s$, and let $0<\epsilon<1$. Since $jf$ is strictly completely continuous, there exists some continuous morphism $h: M\to G$ such that $|jf-h|\leq \epsilon$, with $h(M^\circ)$ a finitely generated $A^\circ$-module, generated by $y_1, \dots, y_r$, say.\\
For any $s\in S$, let $a^s_t\in A^\circ$, $t=1, \dots, r$, such that 
\begin{equation*}
h(e_s)=\sum_{t=1}^r a^s_t y_t.
\end{equation*}
Since $y_t\in h(M^\circ)$ for any $t$, we can choose $x_t\in M^\circ$ such that $h(x_t)=y_t$, and set $z_t=f(x_t)\in N$.\\
\\
Now define elements 
\begin{equation*}
f_s=\sum_{t=1}^r a^s_tz_t\in N
\end{equation*}
for any $s\in S$, and consider the continous morphism $\phi: M=\widehat{\oplus}Ae_s\to N$ obtained by applying Proposition \ref{univfree} to the function
\begin{align*}
S\to N\\
s\mapsto f_s,
\end{align*}
which is bounded as $|f_s|\leq \max_t |z_t|$ for any $s\in S$. Moreover, $\phi(M^\circ)\subseteq \sum A^\circ z_t$ is a finitely generated $A^\circ$-module by Noetherianity.\\
The same calculation as in \cite{Kiehl} verifies that $|f-\phi|\leq \epsilon$, showing that $f$ is strictly completely continuous.
\end{proof}
\begin{theorem}[{see \cite[Korollar 1.5]{Kiehl}}]
\label{fginject}
Let $f: M\to N$ be a surjection in $\Ban_A$, and let $g: M\to N$ be another morphism in $\Ban_A$. Suppose $N$ is a closed submodule of some $G\in \Ban_A$ via the injection $j: N\to G$, and suppose that the composition $jg$ is strictly completely continuous. Then $\im(f+g)$ is closed and of finite index in $N$.
\end{theorem}
\begin{proof}
By Theorem \ref{freeprojscc}, there exists a topologically free module $F$ and a surjection $p: F\to M$ in $\Ban_A$ such that $gp$ is strictly completely continuous. By surjectivity of $p$, we have that $fp$ is still surjective, so Theorem \ref{Schwartz} implies that $\im(fp+gp)=\im((f+g)\circ p)$ is closed and of finite index. But since $p$ is surjective, this is the same as $\im(f+g)$, and the result follows.
\end{proof}
Note that the same result holds for $\im(f-g)$. If $jg$ is strictly completely continuous, written as the limit of some $(h_i)_i$, then $j\circ (-g)$ is strictly completely continuous, as it is the limit of $(-h_i)_i$.\\
\\
These results are applied in the proof of Theorem \ref{Kiehlthm} using the following observation, which in \cite{Kedlaya} is attributed to Cartan--Serre.
\begin{proposition} 
\label{cficomplex}
Let $C^{\bullet}$, $D^{\bullet}$ be two cochain complexes in $\Ban_A$, and let $\alpha=(\alpha_i\in \Ban_A(C^i, D^i))$ be a quasi-isomorphism. Assume further that for each $i$ there exists $F^i\in \Ban_A$ together with a continuous surjection $\beta_i: F^i\to C^i$ such that $\alpha_i\beta_i$ is a stricly completely continuous morphism of $A$-modules. Then $\mathrm{H}^i(D^{\bullet})$ is a finitely generated $A$-module for each $i$.
\end{proposition}
\begin{proof}
This proof can be found in \cite{Kiehl} as part of the proof of Satz 2.5 and Satz 2.6. In a slight abuse of notation, all differentials will be denoted by the same letter $d$.\\
Let $G^i$ be the preimage of $Z^i(C^{\bullet})=\ker d\subseteq C^i$ in $F^i$. Note that $Z^i(C^{\bullet})$ is closed in $C^i$, so it is complete when equipped with the subspace norm. Similarly it follows from continuity that $G^i$ is closed in $F^i$, and hence an object in $\Ban_A$.\\
\\
We wish to apply Theorem \ref{fginject} to
\begin{align*}
G^i\oplus D^{i-1}\to & Z^i(D^{\bullet})\\
(a, b)\ \mapsto & \ d(b) =(\alpha_i\beta_i(a)+d(b))-\alpha_i\beta_i(a).
\end{align*}
Firstly, we claim that the map
\begin{align*}
f: \ &G^i\oplus D^{i-1}\to  Z^i(D^{\bullet})\\
&(a, b) \mapsto \alpha_i\beta_i(a)+d(b)
\end{align*}
is a surjection in $\Ban_A$.\\
We have already shown that each of the modules appearing is an object in $\Ban_A$ (recall that $\Ban_A$ is closed under taking finite direct sums with the corresponding max norm), and since $\alpha_i$, $\beta_i$ and $d$ are all bounded, $f$ is clearly also bounded. For surjectivity, note that we assume that $\alpha_i$ induces an isomorphism of cohomology groups, and hence the map
\begin{align*}
Z^i(C^{\bullet})\oplus D^{i-1}\to & \ Z^i(D^{\bullet})\\
(a, b)\ \ \ \ \ \mapsto &\ \alpha_i(a)+d(b)
\end{align*}
is surjective. \\
Since $\beta_i$ is surjective, it follows that the restriction $\beta_i|_{G^i}: G^i\to Z^i(C^{\bullet})$ is surjective by definition of $G^i$. Therefore the composition 
\begin{equation*}
f: G^i\oplus D^{i-1}\to Z^i(C^{\bullet})\oplus D^{i-1}\to Z^i(D^{\bullet})
\end{equation*}
is also surjective, as required.\\
\\
Secondly, we need to show that the map
\begin{align*}
g: \ &G^i\oplus D^{i-1}\to Z^i(D^{\bullet})\\
&\ \ (a, b)\ \  \mapsto \ \ \alpha_i\beta_i(a)
\end{align*} 
is strictly completely continuous after composition with the injection $j: Z^i(D^{\bullet})\to D^i$. \\
Again, it is straightforward to see that $g$ is a morphism in $\Ban_A$. Note that it fits into the commutative diagram
\begin{equation*}
\begin{xy}
\xymatrix{
G^i\oplus D^{i-1} \ar[d]^{\text{pr}} \ar[rr]^g& & Z^i(D^{\bullet})\ar[dd]^j\\
G^i\ar[d]^{\iota}\\
F^i\ar[r]^{\beta_i} & C^i\ar[r]^{\alpha_i} & D^i
}
\end{xy}
\end{equation*} 
where the bottom row is strictly completely continuous by assumption, the map $\text{pr}$ is the projection onto the first factor, and $\iota$ is the natural inclusion. \\
By Lemma \ref{sccprop}, the composition $\alpha_i\beta_i\iota \text{pr}$ is strictly completely continuous.\\
It follows by commutativity of the diagram that the composition $jg$ is a strictly completely comtinuous morphism of $A$-modules, as required. \\
\\
We can therefore apply Theorem \ref{fginject} (and the remark after its proof) to conclude that $\im d=\im(f-g)$ is closed and of finite index in $Z^i(D^{\bullet})$, i.e.
\begin{equation*}
\mathrm{H}^i(D^{\bullet})=Z^i(D^{\bullet})/d(D^{i-1})
\end{equation*}
is a finitely generated $A$-module.
\end{proof}

Recall that a morphism of semi-normed $K$-vector spaces $\phi: M\to N$ is \textbf{strict} if the induced morphism $\coim \phi\to \im \phi$ is a linear homeomorphism (i.e. the quotient semi-norm on $M/\ker \phi$ is equivalent to the subspace semi-norm on $\im \phi$). Due to the Open Mapping Theorem, a morphism in $\Ban_A$ is strict if and only if it has closed image (see \cite[Proposition 3.7.3/4]{BGR}, \cite[Lemma 2.6]{Bodecompl}). 
\begin{corollary}
\label{cficomplexgivesstrict}
In the situation of Proposition \ref{cficomplex}, $D^\bullet$ is a cochain complex with strict morphisms.
\end{corollary}
\begin{proof}
By the above, $\im d^{j-1}$ is a closed subspace of $Z^j(D^\bullet)$, which is in turn a closed subspace of $D^j$ by continuity. Thus we can apply \cite[Proposition 3.7.3/4]{BGR} to show that $d^{j-1}$ is strict for each $j$.
\end{proof}

\subsection{Completed tensor products}
Recall the definition of the completed tensor product $M\widehat{\otimes}_A N$ from \cite[Appendix B]{Bosch}, where $A$ is a normed $K$-algebra, $M$ a normed right $A$-module, $N$ a normed left $A$-module.\\
We note the following straightforward properties.

\begin{lemma}
\label{removehat}
Let $A$ and $B$ be Noetherian Banach $K$-algebras. Let $N$ be a finitely generated left Banach $A$-module, and let $M$ be a Banach $(B, A)$-bimodule which is finitely generated as a left $B$-module. Then the natural morphism
\begin{equation*}
M\otimes_A N\to M\widehat{\otimes}_A N,
\end{equation*} 
is a $B$-linear homeomorphism, i.e. the tensor semi-norm is a norm with respect to which $M\otimes N$ is already complete.
\end{lemma}
\begin{proof}
This is a straightforward generalization of \cite[Proposition 3.7.3/6]{BGR}.\\
By \cite[Proposition 3.7.3/3]{BGR}, the norm on $N$ is equivalent to one induced by a surjection $\rho: A^{\oplus r}\to N$ for some integer $r$. The map $\rho$ is then strict by definition.\\
Now \cite[Proposition 2.1.8/6]{BGR} implies that the map $M\otimes \rho: M\otimes_A A^{\oplus r}\to M\otimes_A N$ is a strict surjection of semi-normed left $B$-modules, i.e. the tensor semi-norm on $M\otimes N$ is equivalent to the quotient semi-norm induced by $M\otimes \rho$. \\
Note that $M\otimes A^{\oplus r}\cong M^{\oplus r}$ as semi-normed left $B$-modules. But $M^{\oplus r}$ is a finitely generated left Banach $B$-module. Therefore the kernel of $M\otimes \rho$ is closed by \cite[Proposition 3.7.2/2]{BGR}, making $M\otimes N$ a Banach $B$-module by \cite[Propositions 2.1.2/1, 3]{BGR}.
\end{proof}

\begin{lemma}
\label{fgunitball}
Let $A$ be a strictly NB $K$-algebra, and let $M$ be a finitely generated left Banach $A$-module. Then $M^\circ$ is a finitely generated $A^\circ$-module. 
\end{lemma}
\begin{proof}
Since $A^\circ$ is Noetherian, the property of having a finitely generated unit ball is preserved under replacing an $A$-module norm by an equivalent one. So by \cite[Proposition 3.7.3/3]{BGR}, we can assume that the norm on $M$ is induced by a surjection $\rho: A^{\oplus r}\to M$ for some positive integer $r$, giving the finitely generated unit ball $M^\circ\subseteq \pi^{-1}\rho((A^\circ)^{\oplus r})$. 
\end{proof}

We now recall a result from \cite{Bodecompl} concerning exactness properties of $\widehat{\otimes}$.\\
Let $A$ be any normed $K$-algebra, and let $(C^\bullet, \partial)$ be a cochain complex of left Banach $A$-modules, with strict differentials. \\
Let $U$ be normed right $A$-module that is flat as an abstract $A$-module. \\
\\
We equip the cohomology groups $\mathrm{H}^j(C^\bullet)$ with the quotient norm induced from the subspace norm on $\ker \partial^j$. For any left $A^\circ$-module $M$, we will abbreviate the $R$-module $\Tor^{A^\circ}_s(U^\circ, M)$ to $T_s(M)$.
\begin{theorem}[{\cite[Theorem 2.16]{Bodecompl}}]
\label{fullcomplexstrict}
Suppose that for large enough $j$, $T_s((\coim \partial^j)^\circ)$ and $T_s((\ker \partial^j)^\circ)$ have bounded $\pi$-torsion for all $s\geq 0$. Suppose further that for all $j$, the following is satisfied:
\begin{enumerate}[(i)]
\item $T_s\left(\mathrm{H}^j(C^\bullet)^\circ\right)$ has bounded $\pi$-torsion for all $s\geq 0$.
\item $T_s\left ((C^j)^\circ\right)$ has bounded $\pi$-torsion for all $s\geq 0$.
\end{enumerate}
Then the complex $U\otimes_A C^{\bullet}$, with each term being equipped with the tensor product semi-norm, consists of strict morphisms, and the canonical morphism
\begin{equation*}
U\widehat{\otimes}_A \mathrm{H}^j(C^{\bullet})\to \mathrm{H}^j(U\widehat{\otimes}_A C^{\bullet})
\end{equation*}
is an isomorphism for each $j$.
\end{theorem}
Note that this applies in particular whenever $U^\circ$ is a flat $A^\circ$-module (see \cite[Corollary 2.15]{Bodecompl}).\\
We now give some further applications. We assume that $C^j=0$ for sufficiently large $j$, so that the condition on $T_s((\coim \partial^j)^\circ)$ and $T_s((\ker \partial^j)^\circ)$ is automatically satisfied. Also note that since $\mathrm{H}^j(C^{\bullet})$ is Banach with respect to the quotient norm, the $A$-module structure extends to an $\widehat{A}$-module structure.
\begin{corollary}
\label{noethapplication}
Suppose that $U^\circ$ and $A^\circ$ are left Notherian rings such that the module structure on $U^\circ$ is given by a ring morphism $A^\circ\to U^\circ$, and that the multiplication extends to endow $U^\circ \otimes_{A^\circ} \widehat{A^\circ}$ with the structure of a left Noetherian ring. Assume further that for each $j$, the following is satisfied:
\begin{enumerate}[(i)]
\item $\mathrm{H}^j(C^\bullet)$ is a finitely generated $\widehat{A}$-module.
\item $T_s((C^j)^\circ)$ has bounded $\pi$-torsion for each $s\geq 0$.
\end{enumerate}
Then $U\otimes_A C^\bullet$ consists of strict morphisms and the natural morphism
\begin{equation*}
\widehat{U} \otimes_{\widehat{A}} \mathrm{H}^j(C^\bullet)\to \mathrm{H}^j(U\widehat{\otimes}_A C^\bullet)
\end{equation*}
is an isomorphism of $\widehat{U}$-modules.
\end{corollary}
\begin{proof}
By Lemma \ref{complsNB}, $\widehat{A}$ is a strictly NB $K$-algebra whose unit ball is $\widehat{A^\circ}$. Therefore by Lemma \ref{fgunitball}, $\mathrm{H}^j(C^\bullet)^\circ$ is a finitely generated $\widehat{A^\circ}$-module. Now
\begin{equation*}
T_s(\mathrm{H}^j(C^\bullet)^\circ)=\Tor^{A^\circ}_s(U^\circ, \mathrm{H}^j(C^\bullet)^\circ)\cong \Tor^{\widehat{A^\circ}}_s(U^\circ\otimes_{A^\circ} \widehat{A^\circ}, \mathrm{H}^j(C^\bullet)^\circ)
\end{equation*} 
by \cite[Proposition 3.2.9]{Weibel}, as $\widehat{A^\circ}$ is flat over $A^\circ$ by \cite[3.2.3.(iv)]{Berthelot}.\\
By Noetherianity of $\widehat{A^\circ}$, $\mathrm{H}^j(C^\bullet)^\circ$ now admits a free resolution of finitely generated $\widehat{A^\circ}$-modules, so that each 
\begin{equation*}
\Tor^{\widehat{A}^\circ}_s(U^\circ \otimes \widehat{A^\circ}, \mathrm{H}^j(C^\bullet)^\circ)
\end{equation*}
is a finitely generated left $U^\circ\otimes \widehat{A^\circ}$-module, as we assume this ring to be Noetherian.\\
So by Noetherianity, the $\pi$-torsion submodule is also finitely generated, and thus $T_s(\mathrm{H}^j(C^\bullet)^\circ)$ has in fact bounded $\pi$-torsion for each $s\geq 0$. Now apply Theorem \ref{fullcomplexstrict}.\\
\\
For the last isomorphism, note that we have
\begin{equation*}
\widehat{U}\widehat{\otimes}_{\widehat{A}} \mathrm{H}^j(C^\bullet)\cong U\widehat{\otimes}_A\mathrm{H}^j(C^\bullet) \cong \mathrm{H}^j(U\widehat{\otimes}_A C^\bullet)
\end{equation*}
by \cite[Proposition 2.1.7/4]{BGR} and the above, and we can remove the completion symbol over the first tensor product by Lemma \ref{removehat}.
\end{proof}

\begin{corollary}
\label{banachapplication}
Suppose that $U^\circ$ and $A^\circ$ are left Noetherian rings such that the module structure on $U^\circ$ is given by a ring morphism $A^\circ\to U^\circ$, and that both $A$ and $U$ are Banach algebras. Assume further that for each $j$, the following is satisfied:
\begin{enumerate}[(i)]
\item $\mathrm{H}^j(C^\bullet)$ is a finitely generated $A$-module.
\item $T_s((C^j)^\circ)$ has bounded $\pi$-torsion for each $s\geq 0$.
\end{enumerate} 
Then $U\otimes_A C^\bullet$ consists of strict morphisms and $U\otimes_A \mathrm{H}^j(C^\bullet)\cong \mathrm{H}^j(U\widehat{\otimes} C^\bullet)$.
\end{corollary}
\begin{proof}
This is just a special case of the above: $\widehat{A^\circ}=A^\circ$, as we assume $A$ to be complete, so $U^\circ\otimes_{A^\circ} \widehat{A^\circ}=U^\circ$, which is a left Noetherian ring by assumption.
\end{proof}

\section{Statement of the Theorem}
\subsection{Fr\'echet--Stein algebras and coadmissible modules}
We recall some important notions from \cite{Ardakov1}.
\begin{definition}[{see \cite[Definition 9.1]{Ardakov1}}]
\label{defLiealgebroid}
A \textbf{Lie algebroid} on a rigid analytic $K$-variety $X$ is a pair $(\rho, \mathscr{L})$, where $\mathscr{L}$ is a locally free $\mathcal{O}_X$-module of finite rank on $X_{\rig}$ which is also a sheaf of $K$-Lie algebras, and the \textbf{anchor map} $\rho: \mathscr{L}\to \mathcal{T}_X$ is an $\mathcal{O}$-linear map of sheaves of Lie algebras, satisfying
\begin{equation*}
[x, ay]=a[x, y]+\rho(x)(a)y
\end{equation*}
for any $x, y\in \mathscr{L}(U)$, $a\in \mathcal{O}_X(U)$, $U$ any admissible open subset of $X$.
\end{definition} 
This is the natural sheaf analogue of Rinehart's notion of a $(K, A)$-Lie algebra, see \cite{Rinehart}.\\
\\
Let $X$ be a rigid analytic $K$-variety. To any Lie algebroid $(\rho, \mathscr{L})$ on $X$ we can associate the sheaf of Fr\'echet completed enveloping algebras $\wideparen{\mathscr{U}(\mathscr{L})}$, whose sections can be described as follows.\\
Let $U=\Sp A$ be an admissible open affinoid subspace of $X$, and let $\mathcal{A}$ be the unit ball of $A$ with respect to some residue norm (an `\textbf{affine formal model}'). Inside $\mathscr{L}(U)$, choose an $(R, \mathcal{A})$-Lie lattice, i.e. a finitely generated $\mathcal{A}$-submodule $\mathcal{L}$ such that
\begin{enumerate}[(i)]
\item $\mathcal{L}$ generates $\mathscr{L}(U)$ as an $A$-module;
\item $\mathcal{L}$ is closed under the Lie bracket of $\mathscr{L}(U)$;
\item $\mathcal{A}$ is preserved under the induced action of $\mathcal{L}$ on $A$ via $\rho$.
\end{enumerate}
Then $\wideparen{\mathscr{U}(\mathscr{L})}(U)$ is given by
\begin{equation*}
\wideparen{\mathscr{U}(\mathscr{L})}(U)=\wideparen{U_A(\mathscr{L}(U))}:=\varprojlim_n \widehat{U_{\mathcal{A}}(\pi^n\mathcal{L})}_K,
\end{equation*}
where $U_{\mathcal{A}}(\mathcal{L})$ is the enveloping algebra introduced in \cite{Rinehart}. A standard argument shows that this expression is independent of the choices made, see \cite[section 6.2]{Ardakov1}. By construction, $\wideparen{\mathscr{U}(\mathscr{L})}(U)$ carries the structure of a Fr\'echet $K$-algebra. \\
\\
If $X$ is a smooth rigid analytic $K$-variety, then the tangent sheaf $\mathcal{T}_X$ is a Lie algebroid, and we denote the resulting sheaf $\wideparen{\mathscr{U}(\mathcal{T}_X)}$ by $\wideparen{\mathcal{D}}_X$.
\begin{definition}[{see \cite[section 3]{Schneider03}}]
\label{defineFS}
A topological $K$-algebra $U$ is called a (left, two-sided) \textbf{Fr\'echet--Stein algebra} if $U\cong\varprojlim U_n$ is an inverse limit of countably many (left, two-sided) Noetherian Banach $K$-algebras $U_n$, such that for every $n$ the following is satisfied:
\begin{enumerate}[(i)]
\item The morphism $U_{n+1}\to U_n$ makes $U_n$ a flat $U_{n+1}$-module (on the right, on both sides).
\item The morphism $U_{n+1}\to U_n$ has dense image.
\end{enumerate}
\end{definition}
Given a Lie algebroid $\mathscr{L}$ on a rigid analytic $K$-variety $X$, it was shown in \cite[Theorem 3.5]{Bodecompl} that $\wideparen{\mathscr{U}(\mathscr{L})}(U)$ is a two-sided Fr\'echet--Stein algebra for any admissible open affinoid subspace $U$ of $X$.\\
For Fr\'echet--Stein algebras, the natural analogue of a coherent module is a coadmissible module, whose definition we recall below.
\begin{definition}[{see \cite[section 3]{Schneider03}}]
\label{definecoad}
A left module $M$ of a left Fr\'echet--Stein algebra $U=\varprojlim U_n$ is called \textbf{coadmissible} if $M=\varprojlim M_n$, such that the following is satisfied for every $n$:
\begin{enumerate}[(i)]
\item $M_n$ is a finitely generated left $U_n$-module.
\item The natural morphism $U_n\otimes_{U_{n+1}}M_{n+1}\to M_n$ is an isomorphism.
\end{enumerate}
\end{definition}
By \cite[Lemma 3.8]{Schneider03}, the notion of coadmissibility is independent of the chosen presentation $U\cong \varprojlim U_n$.\\
\\
We record the following basic results from \cite[section 3]{Schneider03}.
\begin{proposition}
\label{coadproperties}
Let $M=\varprojlim M_n$ be a coadmissible $U=\varprojlim U_n$-module. Then the following hold:
\begin{enumerate}[(i)]
\item The natural morphism $U_n\otimes_U M\to M_n$ is an isomorphism for each $n$ (see \cite[Corollary 3.1]{Schneider03}).
\item The system $(M_n)_n$ has the Mittag-Leffler property as described in \cite[0.13.2.4]{EGA}, so that $\varprojlim^{(j)}M_n=0$ for any $j\geq 1$ (see \cite[section 3, Theorem B]{Schneider03}).
\item The category of coadmissible $U$-modules is an abelian category, containing all finitely presented $U$-modules (see \cite[Corollaries 3.4, 3.5]{Schneider03}).
\end{enumerate}
\end{proposition}

Given a $(K, A)$-Lie algebra $L$ which is finitely generated projective over $A$, finitely generated $U_A(L)$-modules give rise to coadmissible $\wideparen{U_A(L)}$-modules in a natural way as follows.\\
As $U_A(L)$ is Noetherian, any finitely generated $U_A(L)$-module $M$ is finitely presented, so the module
\begin{equation*}
\wideparen{M}:=\wideparen{U_A(L)}\otimes_{U(L)} M
\end{equation*}
is a finitely presented $\wideparen{U(L)}$-module and thus coadmissible by property (iii) of Proposition \ref{coadproperties}.\\
Choose an $(R, \mathcal{A})$-Lie lattice $\mathcal{L}$ in $L$ and write $U_n=U_{\mathcal{A}}(\pi^n\mathcal{L})$. It now follows from property (i) in Proposition \ref{coadproperties} that
\begin{equation*}
\wideparen{M}\cong \varprojlim (\widehat{U_n}_K\otimes_{U(L)} M).
\end{equation*}
We call $\wideparen{M}$ the coadmissible completion of $M$, as in \cite[Definition 7.1]{Ardakov1}.  
\begin{lemma}[{\cite[Lemma 4.14]{Bodecompl}}]
\label{fgfrcompletion}
The functor $M\mapsto \wideparen{M}$ is exact on finitely generated $U_A(L)$-modules.
\end{lemma}
\begin{theorem}
Let $L$ be a $(K, A)$-Lie algebra which is finitely generated projective as an $A$-module. Then $\wideparen{U_A(L)}$ is flat as a (left and right) $U_A(L)$-module.
\end{theorem}
\begin{proof}
Let $I$ be a left ideal of $U_A(L)$, which is automatically finitely generated by Noetherianity of $U_A(L)$. By \cite[Proposition 3.2.4]{Weibel}, it is sufficient to show that the map $\wideparen{U(L)}\otimes I\to \wideparen{U(L)}$ is injective. But this follows immediately from the previous lemma. Thus $\wideparen{U_A(L)}$ is flat as a right $U(L)$-module. The proof concerning the left module structure is entirely analogous.
\end{proof}
\subsection{Fr\'echet--Stein sheaves}
We now introduce a class of sheaves of algebras which is large enough to include not only those of the form $\wideparen{\mathscr{U}(\mathscr{L})}$, but also various quotients, for example the twisted sheaves $\wideparen{\mathcal{D}}^{\lambda}$ which we are going to define in section 6.

\begin{definition}
\label{globalFSsheaf}
A sheaf of topological $K$-algebras $\mathcal{F}$ on a rigid analytic $K$-variety $X$ is called a (left) \textbf{global Fr\'echet--Stein sheaf} if there exists 
\begin{enumerate}[(i)]
\item a collection of sites $(X_n)_{n\in \mathbb{N}}$ on $X$ such that $X_n$ is contained in $X_{n+1}$ for each $n$, and any $U\in X_w$ is in $X_n$ for sufficiently large $n$, likewise for $X_w$-coverings; and
\item for each $n$, a sheaf of $K$-algebras $\mathcal{F}_n$ on $X_n$, together with morphisms $\mathcal{F}_{n+1}|_{X_n}\to \mathcal{F}_n$,
\end{enumerate}
such that the following holds:
\begin{enumerate}[(i)]
\item there is an isomorphism $\mathcal{F}\cong\varprojlim \mathcal{F}_n$ (where we write $\varprojlim \mathcal{F}_n$ for the sheaf on $X$ obtained from $U\mapsto \varprojlim \mathcal{F}_n(U)$ for $U\in X_w$) which exhibits $\mathcal{F}(U)\cong \varprojlim \mathcal{F}_n(U)$ as a (left) Fr\'echet--Stein algebra for every admissible open affinoid subspace $U$ of $X$;
\item if $V$ is an affinoid subdomain of an admissible open affinoid subspace $U\subseteq X$ such that both $U$ and $V$ are open in $X_n$, then the restriction map $\mathcal{F}_n(U)\to \mathcal{F}_n(V)$ is flat (on the right); and
\item if $U\subseteq X$ is an admissible open affinoid subspace which is open in $X_n$ and $\mathfrak{U}$ is a finite covering of $U$ in $X_n$ by affinoid subdomains, then $\check{\mathrm{H}}^j(\mathfrak{U}, \mathcal{F}_n)=0$ for each $j>0$.
\end{enumerate}
\end{definition}
Given a left global Fr\'echet--Stein sheaf on an affinoid $K$-variety $X$, we can repeat all the arguments in \cite[sections 5, 8]{Ardakov1} to produce a localization functor $\Loc$, sending a coadmissible left $\mathcal{F}(X)$-module $M$ to the sheaf of $\mathcal{F}$-modules given by
\begin{equation*}
U\mapsto \mathcal{F}(U)\wideparen{\otimes}_{\mathcal{F}(X)} M
\end{equation*}
for each $U\in X_w$. By the same argument as in \cite[Theorem 4.16]{Bodecompl}, $\Loc M$ is a sheaf with vanishing higher \v{C}ech cohomology for every finite affinoid covering.
\begin{definition}
\label{defcoadforglFS}
Let $\mathcal{F}$ be a left global Fr\'echet--Stein sheaf on a rigid analytic $K$-variety $X$. A left $\mathcal{F}$-module $\mathcal{M}$ is then called coadmissible if there exists an admissible covering $\mathfrak{U}$ of $X$ by affinoid subspaces such that for every $U\in \mathfrak{U}$, the following holds:
\begin{enumerate}[(i)]
\item $\mathcal{M}(U)$ is a coadmissible $\mathcal{F}(U)$-module;
\item the natural morphism $\Loc \mathcal{M}(U)\to \mathcal{M}|_U$ is an isomorphism.
\end{enumerate}
\end{definition}

\begin{definition}
\label{defFSsheaf}
A sheaf of $K$-algebras $\mathcal{F}$ on a rigid analytic $K$-variety $X$ is a \textbf{full Fr\'echet--Stein sheaf} if for any admissible open affinoid subspace $U$ of $X$, the restriction $\mathcal{F}|_U$ is a global Fr\'echet--Stein sheaf on $U$.\\
An $\mathcal{F}$-module $\mathcal{M}$ is called coadmissible if there exists an admissible covering $\mathfrak{U}$ by affinoid subspaces such that for all $U\in \mathfrak{U}$, $\mathcal{M}|_U$ is a coadmissible $\mathcal{F}|_U$-module.
\end{definition}

The analogue of Kiehl's Theorem (\cite[Theorem 8.4]{Ardakov1}, \cite[Theorem 4.17]{Bodecompl}) still holds in this generalized context, so that if $\mathcal{M}$ is coadmissible with respect to one covering, then it is coadmissible with respect to any affinoid covering.\\
\\
If there exists a collection $\mathcal{S}$ of admissible open affinoid subspaces of $X$ forming a basis of the topology with the property that $\mathcal{F}|_U$ is a global Fr\'echet--Stein sheaf for each $U\in \mathcal{S}$, we call $\mathcal{F}$ simply a \textbf{Fr\'echet--Stein sheaf}. For this it is evidently sufficient to give one admissible covering $(U_i)$ of $X$ by affinoid subspaces such that $\mathcal{F}|_{U_i}$ is a global Fr\'echet--Stein sheaf for each $i$. \\
There is a natural analogue of the theory above for Fr\'echet--Stein sheaves, see \cite{Ardakov1}.\\
\\
We can now restate a result from \cite{Bodecompl} as follows.
\begin{proposition}[{see \cite[Theorems 3.5, 4.9, 4.10]{Bodecompl}}]
\label{recallglobal}
If $\mathscr{L}$ is a Lie algebroid on some rigid analytic $K$-variety $X$, then $\wideparen{\mathscr{U}(\mathscr{L})}$ is a full Fr\'echet--Stein sheaf. 
\end{proposition}
Let $\mathscr{L}$ be a Lie algebroid on an affinoid $K$-variety $X$. For future reference, we recall the corresponding sites $X_n$ and sheaves $\mathcal{F}_n$ explicitly in this case.\\
\\
Let $X=\Sp A$, and let $\mathcal{A}$ be an affine formal model in $A$. Choose an $(R, \mathcal{A})$-Lie lattice $\mathcal{L}$ in $\mathscr{L}(X)$. Recall from \cite[Definition 3.1]{Ardakov1} that an affinoid subdomain $Y=\Sp B$ of $X$ is called $\mathcal{L}$-\textbf{admissible} if $B$ contains an $\mathcal{L}$-stable affine formal model, i.e. an affine formal model which contains the image of $\mathcal{A}$ under restriction and is preserved by the $\mathcal{L}$-action.\\
We now consider the site $X_n=X_{\ac}(\pi^n\mathcal{L})$ of $\pi^n\mathcal{L}$-accessible subdomains as defined in \cite{Ardakov1}.
\begin{definition}[{\cite[Definitions 4.6, 4.8]{Ardakov1}}]
\label{rationalac}
Let $Y$ be a rational subdomain of $X$. If $Y=X$, we say that it is $\mathcal{L}$-accessible in 0 steps. Inductively, if $n\geq 1$ then we say that it is $\mathcal{L}$-accessible in $n$ steps if there exists a chain $Y\subseteq Z\subseteq X$ such that the following is satisfied:
\begin{enumerate}[(i)]
\item $Z\subseteq X$ is $\mathcal{L}$-accessible in $(n-1)$ steps,
\item $Y=Z(f)$ or $Z(f^{-1})$ for some non-zero $f\in \mathcal{O}(Z)$,
\item there is an $\mathcal{L}$-stable affine formal model $\mathcal{C}\subset \mathcal{O}(Z)$ such that $\mathcal{L}\cdot f\subseteq \mathcal{C}$.
\end{enumerate}
A rational subdomain $Y\subseteq X$ is said to be $\mathcal{L}$-accessible if it is $\mathcal{L}$-accessible in $n$ steps for some $n\in \mathbb{N}$.\\
An affinoid subdomain $Y$ of $X$ is said to be $\mathcal{L}$-\textbf{accessible} if it is $\mathcal{L}$-admissible and there exists a finite covering $Y=\cup_{j=1}^r Y_j$ where each $Y_j$ is an $\mathcal{L}$-accessible rational subdomain of $X$.\\
A finite covering $\{Y_j\}$ of $X$ by affinoid subdomains is said to be $\mathcal{L}$-accessible if each $Y_j$ is an $\mathcal{L}$-accessible affinoid subdomain of $X$.
\end{definition}
We then define the sheaf $\mathscr{U}_n(\mathscr{L})$ on $X_{\ac}(\pi^n\mathcal{L})$ by setting
\begin{equation*}
\mathscr{U}_n(\mathscr{L})(Y)=\widehat{U_{\mathcal{B}}(\mathcal{B}\otimes_{\mathcal{A}}\pi^n\mathcal{L})}_K,
\end{equation*}
for any $\pi^n\mathcal{L}$-accessible affinoid subdomain $Y=\Sp B$ of $X$, where $\mathcal{B}\subset B$ is some $\pi^n\mathcal{L}$-stable affine formal model.\\
Also note that by \cite[Proposition 2.3]{Ardakov1}, $\mathscr{U}_n(\mathscr{L})(Y)$ is isomorphic as a $B$-module to $B\widehat{\otimes}_A U_A(\mathscr{L}(X))$, where $B$ is equipped with some residue norm (without loss of generality with unit ball $\mathcal{B}$ as above) and $U_A(\mathscr{L}(X))$ is equipped with the gauge semi-norm associated to $U_{\mathcal{A}}(\pi^n\mathcal{L})$.\\
It was shown in \cite{Bodecompl} that the sheaves $\mathscr{U}_n(\mathscr{L})$ have the desired properties.\\
\\
Recall the following result from \cite{Schneider03}.
\begin{lemma}[{see \cite[Proposition 3.7, Lemma 3.8]{Schneider03}}]
\label{quotiscoadenl}
Let $I$ be a closed two-sided ideal in a Fr\'echet--Stein algebra $U$. Then $U/I$ is a Fr\'echet--Stein algebra, and a $U/I$-module is coadmissible if and only if it is coadmissible as a $U$-module.
\end{lemma}
We will slightly extend this result and its sheaf analogue in order to produce further examples of global Fr\'echet--Stein sheaves.\\
Just as a Fr\'echet--Stein algebra $B=\varprojlim B_n$ carries a natural Fr\'echet topology as the inverse limit topology of the Banach norms on $B_n$, so any coadmissible $B$-module $M=\varprojlim M_n$ carries a canonical Fr\'echet topology induced by the canonical Banach module structures on each $M_n$.\\
An algebra structure on $M$ is said to have \textbf{continuous multiplication} (with respect to $(B_n)$) if we can choose for each sufficiently large $n$ a Banach $B_n$-module norm $|-|$ on $M_n$ such that the natural morphism $\iota_n: M\to M_n$ endows $M$ with a semi-norm $|-|_n$ which is submultiplicative, i.e.
\begin{equation*}
|xy|_n:=|\iota_n(xy)| \leq |\iota_n(x)|\cdot |\iota_n(y)|
\end{equation*}   
for any $x, y \in M$. If this holds, multiplication on $M$ extends continously to make $M_n=B_n\otimes_B M$ a Banach $K$-algebra.\\
It is important to note that this definition depends on the presentation $B\cong \varprojlim B_n$, and is stronger than requiring that $M$ be a Fr\'echet algebra with respect to the canonical Fr\'echet topology.

\begin{lemma}
\label{passtoFS}
Let $B=\varprojlim B_n$ be a left Fr\'echet--Stein $K$-algebra, and let $C$ be a $K$-algebra which is also a left coadmissible $B$-module with continuous multiplication via an algebra morphism $B\to C$. Then $C$ is a left Fr\'echet--Stein algebra.\\
If $M$ is a left $C$-module, then it is coadmissible if and only if it is coadmissible as a $B$-module.
\end{lemma}
\begin{proof}
By assumption, $C_n:=B_n\otimes_B C$ is a finitely generated left $B_n$-module and the canonical Banach norm gives rise to a submultiplicative semi-norm $|-|_n$ on $C$. Note that the image of $\iota_n: C\to C_n$ is dense by \cite[Theorem A]{Schneider03}. But then the completion of $C$ with respect to $|-|_n$ is a Banach $K$-algebra, which as a Banach space is naturally isomorphic to $C_n$ by construction.\\
Thus $C= \varprojlim C_n$ is the limit of Banach $K$-algebras, and each $C_n$ is left Noetherian, as it is finitely generated over $B_n$ via the natural morphism $B_n\to C_n$.\\
\\
Since the functor $C_n\otimes_{C_{n+1}}-$ can be written as 
\begin{equation*}
(B_n\otimes_{B_{n+1}} C_{n+1})\otimes_{C_{n+1}}-,
\end{equation*}
flatness follows from flatness of the maps $B_{n+1}\to B_n$.\\
Since the map $B_{n+1}\to B_n$ has dense image, it follows that $C_{n+1}\to C_n=B_n\otimes_{B_{n+1}} C_{n+1}$ also has dense image for each $n$. Thus $C$ is a left Fr\'echet--Stein algebra.\\
\\
The second part of the statement is now a simplified version of \cite[Lemma 3.8]{Schneider03}.\\
If $M$ is a left $C$-module, then 
\begin{equation*}
M_n:=C_n\otimes_C M\cong (B_n\otimes_B C)\otimes_C M=B_n\otimes_B M
\end{equation*} 
as a $B_n$-module, and $M_n$ is finitely generated as a $C_n$-module if and only if it is finitely generated as a $B_n$-module, because $C_n$ is finitely generated over $B_n$. Moreover,
\begin{align*}
B_n\otimes_{B_{n+1}} M_{n+1}&\cong (B_n\otimes_{B_{n+1}} C_{n+1})\otimes_{C_{n+1}} M_{n+1} \\
& \cong C_n\otimes_{C_{n+1}} M_{n+1},
\end{align*}
finishing the proof.
\end{proof}

\begin{proposition}
\label{FSbasechange}
Let $X$ be a rigid analytic $K$-variety and let $\mathcal{F}'=\varprojlim \mathcal{F}'_n$ be a left global Fr\'echet--Stein sheaf on $X$.\\
Let $\mathcal{F}$ be a sheaf of $K$-algebras on $X$ which is also a left coadmissible $\mathcal{F}'$-module via a morphism $\theta: \mathcal{F}'\to \mathcal{F}$. Assume that $\mathcal{F}(U)$ has continuous multiplication with respect to $(\mathcal{F}'_n(U))$ for each admissible open affinoid subspace $U$.\\
Then $\mathcal{F}$ is itself a left global Fr\'echet--Stein sheaf on $X$, and an $\mathcal{F}$-module $\mathcal{M}$ is coadmissible if and only if it is coadmissible as an $\mathcal{F}'$-module.
\end{proposition} 
\begin{proof}
Consider the sheaves $\mathcal{F}_n:=\mathcal{F}'_n\otimes_{\mathcal{F}'}\mathcal{F}$ on $X_n$. By assumption, this is in fact a sheaf of $K$-algebras, extending the multiplication on $\mathcal{F}$ by continuity. Now let $U$ be an admissible open affinoid subspace of $X$.\\
We have shown in \cite[Theorem 4.16]{Bodecompl} that the sheaf $\mathcal{F}_n|_U$ on $X_n$ has vanishing higher \v{C}ech cohomology. The flatness of restriction maps is again inherited from $\mathcal{F}'_n$, and $\mathcal{F}(U)$ is a Fr\'echet--Stein $K$-algebra by Lemma \ref{passtoFS}. This makes $\mathcal{F}$ a global Fr\'echet--Stein sheaf on $X$.\\
The statement on coadmissible modules is now just Lemma \ref{passtoFS} combined with the natural isomorphism
\begin{equation*}
\mathcal{F}(V)\wideparen{\otimes}_{\mathcal{F}(U)} M\cong \left(\mathcal{F}'(V)\wideparen{\otimes}_{\mathcal{F}'(U)} \mathcal{F}(U)\right)\wideparen{\otimes}_{\mathcal{F}(U)} M \cong \mathcal{F}'(V)\wideparen{\otimes}_{\mathcal{F}'(U)}M
\end{equation*}
for any coadmissible $\mathcal{F}(U)$-module $M$ and any affinoid subdomain $V$ of $U$ (see \cite[Proposition 7.4]{Ardakov1}).
\end{proof}
In this case, we call $\mathcal{F}$ a \textbf{coadmissible enlargement} of $\mathcal{F}'$.\\
A standard example of a coadmissible enlargement is given by the following. If $\mathscr{L}'\to \mathscr{L}$ is an epimorphism of Lie algebroids on an affinoid $K$-variety $X=\Sp A$, we can choose an affine formal model $\mathcal{A}$ in $A$ and an $(R, \mathcal{A})$-Lie lattice $\mathcal{L}'$ in $L':=\mathscr{L}'(X)$, and let $\mathcal{L}$ denote the image of $\mathcal{L}'$ in $L:=\mathscr{L}(X)$. Then $\mathcal{L}$ is an $(R, \mathcal{A})$-Lie lattice in $L$, and the induced map $U_{\mathcal{A}}(\pi^n\mathcal{L}')\to U_{\mathcal{A}}(\pi^n\mathcal{L})$ is surjective for each $n$.\\
By \cite[3.2.3.(iii)]{Berthelot}, $\widehat{U_{\mathcal{A}}(\pi^n\mathcal{L})}$ is isomorphic to $\widehat{U(\pi^n\mathcal{L}')}\otimes_{U(\pi^n\mathcal{L}')}U(\pi^n\mathcal{L})$ as a $\widehat{U(\pi^n\mathcal{L}')}$-module, and hence tensoring with $K$ yields
\begin{equation*}
\widehat{U(\pi^n\mathcal{L})}_K\cong \widehat{U(\pi^n\mathcal{L}')}_K\otimes_{U(L')} U(L).
\end{equation*}
Thus $\wideparen{\mathscr{U}(\mathscr{L})}(X)=\wideparen{U_A(L)}$ is the coadmissible completion of  the finitely generated $U_A(L')$-module $U_A(L)$ as discussed in Lemma \ref{fgfrcompletion}. This shows that $\wideparen{\mathscr{U}(\mathscr{L})}(X)$ is a coadmissible $\wideparen{\mathscr{U}(\mathscr{L}')}(X)$-module with continuous multiplication, and repeating the argument for arbitrary affinoid subdomains shows that the natural epimorphism $\wideparen{\mathscr{U}(\mathscr{L}')}\to \wideparen{\mathscr{U}(\mathscr{L})}$ turns $\wideparen{\mathscr{U}(\mathscr{L})}$ into a coadmissible enlargement of $\wideparen{\mathscr{U}(\mathscr{L}')}$. The proposition above can then be viewed as the sheaf analogue of Lemma \ref{quotiscoadenl}.\\
\\
Note that another example of coadmissible enlargement was already given in \cite[Theorem 5.1]{Schneider03}: there, the distribution algebra $D(G_0, K)$ was described as a free $D(H_0, K)$-module of finite rank, and a crucial step in the proof that $D(G_0, K)$ was Fr\'echet--Stein consisted in checking that the multiplication on $D(G_0, K)$ satisfied the necessary continuity condition.\\ 
We will see other examples arising as quotients of $\wideparen{\mathscr{U}(\mathscr{L})}$ in section 6.

\begin{lemma}
\label{pushfaffinoidglobal}
Let $h: X\to Y$ be an affinoid morphism of rigid analytic $K$-varieties, and let $\mathcal{F}$ be a global Fr\'echet--Stein sheaf on $X$. Then $h_*\mathcal{F}$ is a global Fr\'echet--Stein sheaf on $Y$.\\
If $\mathcal{M}$ is a coadmissible $\mathcal{F}$-module, then $h_*\mathcal{M}$ is a coadmissible $h_*\mathcal{F}$-module.\\
If $\mathcal{G}$ is a coadmissible enlargement of $\mathcal{F}$, then $h_*\mathcal{G}$ is a coadmissible enlargement of $h_*\mathcal{F}$. 
\end{lemma}
\begin{proof}
Let $\mathcal{F}_n$ be sheaves on sites $X_n$, satisfying the conditions given in Definition \ref{globalFSsheaf}. Define $Y_n$ to be the Grothendieck topology on $Y$ induced by $X_n$, i.e. an affinoid subspace $U$ in $Y_w$ is open in $Y_n$ if $h^{-1}U$ is open in $X_n$, analogously for coverings. Then the sheaf $h_*\mathcal{F}_n$ is defined on $Y_n$, and $h_*\mathcal{F}\cong \varprojlim h_*\mathcal{F}_n$.\\
The claims now follow immediately from the definitions, as for any admissible open affinoid subspace $U$ of $Y$, its preimage $h^{-1}U$ is an admissible open affinoid subspace of $X$ by assumption.
\end{proof}

\subsection{Proper morphisms}
We now describe the geometric situation we will be concerned with in this paper.
\begin{definition}[{see \cite[Definition 6.3/6]{Bosch}}]
\label{relcpct}
Let $f: X\to Y$ be a morphism of rigid analytic varieties with $Y$ being affinoid, and let $U\subseteq U'\subseteq X$ be admissible open affinoid subspaces. We say $U$ is \textbf{relatively compact} in $U'$ (with respect to $Y$), or $U$ lies in the interior of $U'$ with respect to $Y$, if the map $\mathcal{O}_Y(Y)\to \mathcal{O}_X(U')$ gives rise to a surjection 
\begin{equation*}
\theta: \mathcal{O}_Y(Y)\langle x_1, \dots x_l\rangle \to \mathcal{O}_X(U')
\end{equation*}
for some integer $l$, such that
\begin{equation*}
U\subseteq \{x\in U': |f_i(x)|<1\},
\end{equation*}
where $f_i$ is the image of $x_i$ under $\theta$.
\end{definition}
Recall that $|f_i(x)|$ is the norm of the residue class $\overline{f_i}$ in $\mathcal{O}_X(U')/\mathfrak{m}_x$, a finite field extension of $K$, with $\mathfrak{m}_x$ the maximal ideal of $\mathcal{O}_X(U')$ corresponding to $x\in U'$.
\begin{definition}[{see \cite[Definition 6.3/8]{Bosch}}]
\label{proper}
A morphism $f: X\to Y$ between rigid analytic varieties is \textbf{proper} if it is separated and there exists an admissible affinoid covering $(\Sp A_i)_{i\in I}$ of $Y$ such that for all $i\in I$, $X_i=f^{-1}(\Sp A_i)$ has two finite admissible affinoid coverings $(U_{ij})$, $(V_{ij})$ with $V_{ij}$ being relatively compact in $U_{ij}$ with respect to $\Sp A_i$ for each $j$.
\end{definition}
As properness is local on the base (see \cite[Proposition 9.6.2/3]{BGR}), we will often restrict our attention to the case when $Y=\Sp A$ is itself affinoid and satisfies the condition in Definition \ref{proper}, i.e. we have two finite admissible affinoid coverings $\mathfrak{U}=(U_i)$, $\mathfrak{V}=(V_i)$ of $X$ such that $V_i$ is relatively compact in $U_i$ with respect to $Y$ for each $i$. Thus there exists a commutative diagram
\begin{equation*}
\begin{xy}
\xymatrix{
A\langle x_1, \dots, x_l\rangle \ar[d]^{\theta_i} \ar[rd]^{h_i}\\
\mathcal{O}_X(U_i)\ar[r]^{\text{res}} & \mathcal{O}_X(V_i)}
\end{xy}
\end{equation*}
such that the map $\theta_i$ is surjective and
\begin{equation*}
|h_i(x_j)|_{\text{sup}}<1
\end{equation*}
for any $j=1, \dots, l$ by the maximum principle \cite[Theorem 3.1/15]{Bosch}.\\
In particular, $h_i(x_j)$ is topologically nilpotent in $\mathcal{O}_X(V_i)$ for each $j$ (it follows from \cite[Corollary 3.1/18]{Bosch} that this notion is independent of the choice of norm on $\mathcal{O}_X(V_i)$).\\
Moreover, writing $U_{i_1\dots i_j}$ for the finite intersection $U_{i_1}\cap \dots \cap U_{i_j}$, it follows from separatedness that all $U_{i_1\dots i_j}$ and $V_{i_1\dots i_j}$ are admissible open affinoid subspaces of $X$, and that $V_{i_1\dots i_j}$ is relatively compact in $U_{i_1\dots i_j}$ with respect to $Y$ (see \cite[Lemma 6.3/7.(iii)]{Bosch}).\\
\\
In this situation (i.e. when the covering $(\Sp A_i)$ in Definition \ref{proper} consists of a single affinoid) we say that $f: X\to Y$ is \textbf{elementary proper}.\\
\\
Note that if $f: X\to Y=\Sp A$ is elementary proper, then $\mathcal{O}_X(X)$ is a finitely generated $A$-module by Theorem \ref{Kiehlthm}. In particular, it is of topologically finite type, and hence an affinoid $K$-algebra. The morphism $f$ thus admits a factorization $X\to \Sp \mathcal{O}_X(X)\to Y$.\\
More generally, we have the following version of Stein factorization.
\begin{proposition}[{see \cite[Proposition 9.6.3/5]{BGR}}]
\label{steinfac}
Let $f: X\to Y$ be a proper morphism of rigid analytic $K$-varieties. Then there exists a rigid analytic $K$-variety $Z$ and a factorization
\begin{equation*}
\begin{xy}
\xymatrix{X\ar[rd]_f \ar[r]^g & Z\ar[d]^h\\
& Y}
\end{xy}
\end{equation*}
where $g$ is a surjective proper morphism with connected fibres and $g_*\mathcal{O}_X\cong \mathcal{O}_Z$, and where $h$ is finite.
\end{proposition} 

Let $f: X\to Y$ be a proper morphism of rigid analytic $K$-varieties, and let $(\rho, \mathscr{L})$ be a Lie algebroid on $X$.\\
Most of this paper will be devoted to the special case when $\mathscr{L}$ is a \emph{free} $\mathcal{O}_X$-module. The more general result will be a relatively straightforward corollary.

\begin{theorem}
\label{KiehlDXfree}
Let $f: X\to Y$ be an elementary proper morphism of rigid analytic $K$-varieties, and let $\mathscr{L}$ be a Lie algebroid on $X$ which is free as an $\mathcal{O}_X$-module. Then $\wideparen{\mathscr{U}(\mathscr{L})}$ is a global Fr\'echet--Stein sheaf on $X$, $f_*\wideparen{\mathscr{U}(\mathscr{L})}$ is a global Fr\'echet--Stein sheaf on $Y$, and if $\mathcal{M}$ is a coadmissible left $\wideparen{\mathscr{U}(\mathscr{L})}$-module, then $\mathrm{R}^jf_*\mathcal{M}$ is a coadmissible left $f_*\wideparen{\mathscr{U}(\mathscr{L})}$-module for each $j\geq 0$.
\end{theorem}

While we will prove the above theorem for coadmissible left modules, all arguments can be easily adapted to right modules. From now on, all coadmissible modules will be understood to be \emph{left} modules.\\
\\
By Lemma \ref{pushfaffinoidglobal}, we can assume without loss of generality that $f$ is equal to the first map in its Stein factorization, i.e. $Y=\Sp A$, where $A=\mathcal{O}_X(X)$. We will work in this specific setting until the end of section 5. \\
Note that if $U\subseteq Y$ is an affinoid subdomain of $Y$, then all our assumptions are still satisfied after restricting to $f|_{f^{-1}U}: f^{-1}U\to U$. If $U=\Sp B$, then $\mathcal{O}_X(f^{-1}U)=B$ by Kiehl's Proper Mapping Theorem, $\mathscr{L}|_{f^{-1}U}$ is a free Lie algebroid and $f|_{f^{-1}U}$ is an elementary proper morphism $f^{-1}U\to U$ by the behaviour of relative compactness under direct products (see \cite[Lemma 6.3/7.(i)]{Bosch}).

\subsection{The sheaves $\mathscr{U}_n$ and $\mathcal{M}_n$}
Let us abbreviate $\wideparen{\mathscr{U}(\mathscr{L})}$ to $\wideparen{\mathscr{U}}$, and let $\mathcal{M}$ be a coadmissible $\wideparen{\mathscr{U}}$-module.\\
We begin by showing that both $\wideparen{\mathscr{U}}$ and $f_*\wideparen{\mathscr{U}}$ are global Fr\'echet--Stein sheaves. For this, we will construct sheaves $\mathscr{U}_n$ such that $\varprojlim \mathscr{U}_n\cong\wideparen{\mathscr{U}}$, similarly to the discussion of $\wideparen{\mathscr{U}(\mathscr{L})}$ on an \emph{affinoid} $K$-variety in \cite{Bodecompl}. We will then proceed by describing a similar construction for $\mathcal{M}$, which will allow us to reduce to the Noetherian case.
\begin{lemma}
\label{pushforwalgebroid}
The pushforward $f_*\mathscr{L}$ is a Lie algebroid on $Y$.
\end{lemma}
\begin{proof}
By Kiehl's Proper Mapping Theorem, $f_*\mathscr{L}$ is a coherent $\mathcal{O}_Y$-module, and it is free by assumption.\\
The anchor map $\rho: \mathscr{L}\to \mathcal{T}_X$ gives rise to a Lie algebra action of $L$ on $\mathcal{O}_X(X)=A$. Restricting to an admissible affinoid covering of $X$, it follows from the definition of a Lie algebroid that $L$ acts via derivations (see \cite[Proposition 9.1]{Ardakov1}), and that the corresponding map $\rho'(Y): L\to \Der_K(A)$ satisfies the Leibniz property of an anchor map. \\
By the remark at the end of the previous subsection, we obtain corresponding morphisms $\rho'(U): f_*\mathscr{L}(U)\to \mathcal{T}_Y(U)$ for any affinoid subdomain $U\subseteq Y$, which naturally give rise to an anchor map $\rho': f_*\mathscr{L}\to \mathcal{T}_Y$, finishing the proof.
\end{proof}
We now fix an affine formal model $\mathcal{A}$ inside $A=\mathcal{O}_X(X)$, and let $\mathcal{L}$ be an $(R, \mathcal{A})$-Lie lattice inside $L=\mathscr{L}(X)$. Since $\mathscr{L}$ is assumed to be free, $L$ is a free $A$-module, so that we can (and will) take $\mathcal{L}$ to be a free $\mathcal{A}$-module.
\begin{proposition}
\label{UisglFS}
The sheaf $\wideparen{\mathscr{U}}$ is a global Fr\'echet--Stein sheaf on $X$.
\end{proposition}
\begin{proof}
For each $n\geq 0$, we define a Grothendieck topology on $X$, whose site we denote by $X_n$, and a sheaf $\mathscr{U}_n$ on $X_n$ satisfying the conditions in Definition \ref{globalFSsheaf}.\\
\\
Let $\mathfrak{U}=(U_i)$, $\mathfrak{V}=(V_i)$ be affinoid coverings of $X$ as described in Definition \ref{proper}, i.e. for each $i$, $V_i$ is relatively compact in $U_i$ with respect to $Y$.\\
\\
Note that by \cite[Lemma 3.1]{Ardakov1}, each $\mathcal{O}_X(U_i)$ admits an affine formal model containing the image of $\mathcal{A}$ under the restriction map
\begin{equation*}
\mathcal{O}_X(X)\to \mathcal{O}_X(U_i).
\end{equation*}
Replacing $\mathcal{L}$ by $\pi^n\mathcal{L}$ for suitable $n$, we can assume that $\mathcal{O}_X(U_i)$ admits an affine formal model $\mathcal{B}_i$ that 
\begin{enumerate}[(i)]
\item contains the image of $\mathcal{A}$,  and
\item is preserved under the action of $\mathcal{L}$ induced via the map $\mathscr{L}(X)\to \mathscr{L}(U_i)$ (as any affine formal model is topologically of finite type).
\end{enumerate}
We adopt the same terminology as in the case of affinoid subdomains and call such a $\mathcal{B}_i$ an $\mathcal{L}$-stable affine formal model.\\
\\
Thus 
\begin{equation*}
\mathcal{L}_i:=\mathcal{B}_i\otimes_{\mathcal{A}}\mathcal{L}\subseteq \mathcal{O}_X(U_i)\otimes_A L=\mathscr{L}(U_i)
\end{equation*}
 is an $(R, \mathcal{B}_i)$-Lie lattice inside $L_i=\mathscr{L}(U_i)$ for each $i$.\\
\\
Recall that for each $i$, we have defined the G-topology $U_{i, \ac}(\mathcal{L}_i)$ of $\mathcal{L}_i$-accessible subdomains of $U_i$ in Definition \ref{rationalac}. Again, replacing $\mathcal{L}$ by $\pi^n\mathcal{L}$ for suitable $n$ and invoking \cite[Proposition 7.6]{Ardakov1}, we can assume that each $U_{i_1\dots i_j}$ and each $V_{i_1\dots i_j}$ is $\mathcal{L}_i$-accessible whenever it is a subspace of $U_i$ (here we are using the fact that both coverings are finite).\\
Furthermore, using \cite[Lemma 3.1]{Ardakov1}, we can find affine formal models $\mathcal{B}_{ij}$ in $\mathcal{O}_X(U_{ij})$ such that $\mathcal{B}_{ij}$ contains the image of $\mathcal{B}_i$ and $\mathcal{B}_j$ under restriction. Replacing $\mathcal{L}$ by $\pi^n\mathcal{L}$, we can assume that $\mathcal{B}_{ij}$ is $\mathcal{L}$-stable for each $i, j$, and thus both $\mathcal{L}_i$-stable and $\mathcal{L}_j$-stable by construction.\\
In particular, 
\begin{equation*}
\mathcal{B}_{ij}\otimes_{\mathcal{B}_i} \mathcal{L}_i\cong\mathcal{B}_{ij}\otimes_{\mathcal{A}}\mathcal{L}\cong\mathcal{B}_{ij}\otimes_{\mathcal{B}_j} \mathcal{L}_j
\end{equation*}
is an $(R, \mathcal{B}_{ij})$-Lie lattice inside $\mathscr{L}(U_{ij})$.\\
\\
We now define the site $X_n$ to be the G-topology on $X$ generated by the $U_{i, \ac}(\mathcal{L}_i)$, i.e. the finest G-topology on $X$ inducing on $U_i$ the topology $U_{i, \ac}(\mathcal{L}_i)$ -- see \cite[9.1.3]{BGR}.\\
\\
Recall from the discussion after Definition \ref{rationalac} that we have for each non-negative integer $n$ a sheaf of $K$-algebras $\mathscr{U}_{n, i}$ on $U_{i,\ac}(\pi^n\mathcal{L}_i)$ given by
\begin{equation*}
U\mapsto \mathcal{O}_X(U)\widehat{\otimes}_{B_i} \widehat{U(\pi^n\mathcal{L}_i)}_K,
\end{equation*}
satisfying $\varprojlim \mathscr{U}_{n, i}(U)=\wideparen{\mathscr{U}(\mathscr{L})}(U)$ for each affinoid subdomain $U\subseteq U_i$.\\
On each overlap $U_{ij}=U_i\cap U_j$, we have
\begin{align*}
\mathscr{U}_{n, i}|_{U_{ij}}& =(\mathcal{O}_X(U_{ij})\widehat{\otimes}_{B_i}U(L_i))\ \widetilde{}\\
& =\widehat{U(\mathcal{B}_{ij}\otimes_{\mathcal{A}} \mathcal{L})}_K\ \widetilde{}\\
&=(\mathcal{O}_X(U_{ij})\widehat{\otimes}_{B_j}U(L_j))\ \widetilde{}\\
&=\mathscr{U}_{n, j}|_{U_{ij}},
\end{align*}
where we write $M \ \widetilde{}$ \ for the presheaf $V\mapsto \mathcal{O}_X(V)\widehat{\otimes} M$. Thus the sheaves $\mathscr{U}_{n, i}$ agree on all overlaps and glue to give a sheaf $\mathscr{U}_n$ on $X_n$.\\
Since $\varprojlim \mathscr{U}_{n, i}=\wideparen{\mathscr{U}}|_{U_i}$ on each $U_i$, this implies the equality $\varprojlim \mathscr{U}_n(U)= \wideparen{\mathscr{U}}(U)$ for any admissible open subspace $U$ of $X$.\\
\\
Restricting to any admissible open affinoid subspace $U$ of $X$, the construction above coindices with the one given after Definition \ref{rationalac}, so that the conditions on flat restriction and vanishing cohomology on $U$ follow directly from \cite[Theorems 4.9, 4.10]{Bodecompl}.\\
Thus $\wideparen{\mathscr{U}}\cong \varprojlim \mathscr{U}_n$ is a global Fr\'echet--Stein sheaf.
\end{proof}

\begin{lemma}
\label{pushffsalg}
The natural morphism $\wideparen{\mathscr{U}(f_*\mathscr{L})}\to f_*\wideparen{\mathscr{U}}$ is an isomorphism. In particular, $f_*\wideparen{\mathscr{U}}$ is a global Fr\'echet--Stein sheaf on $Y$.
\end{lemma}
\begin{proof}
Consider the \v{C}ech complex $\check{C}^{\bullet}(\mathfrak{V}, \mathcal{O}_X)$, where $\mathfrak{V}=(V_i)$. It follows from Corollary \ref{cficomplexgivesstrict} that this is a finite cochain complex of Banach $A$-modules with strict morphisms, where each cohomology group is a finitely generated $A$-module by Kiehl's Proper Mapping Theorem.\\
Since $\pi^n\mathcal{L}$ is a free $\mathcal{A}$-module, so is $U_{\mathcal{A}}(\pi^n\mathcal{L})$ by Rinehart's Theorem, \cite[Theorem 3.1]{Rinehart}. In particular, applying \cite[Corollary 2.15]{Bodecompl}, the complex
\begin{equation*}
\widehat{U(\pi^n\mathcal{L})}_K\widehat{\otimes}_A \check{C}^{\bullet}(\mathfrak{V}, \mathcal{O}_X)=\check{C}^\bullet(\mathfrak{V}, \mathscr{U}_n)
\end{equation*}
has cohomology
\begin{equation*}
\check{\mathrm{H}}^j(\mathfrak{V}, \mathscr{U}_n)\cong\widehat{U(\pi^n\mathcal{L})}_K\widehat{\otimes}_A \check{\mathrm{H}}^j(\mathfrak{V}, \mathcal{O}_X)=\widehat{U(\pi^n\mathcal{L})}_K\otimes_A \check{\mathrm{H}}^j(\mathfrak{V}, \mathcal{O}_X)
\end{equation*}
by Lemma \ref{removehat}.\\
This naturally identifies $\mathscr{U}_n(X)$ with $\widehat{U(\pi^n\mathcal{L})}_K$, and $f_*\wideparen{\mathscr{U}}(Y)=\wideparen{\mathscr{U}}(X)=\wideparen{U_A(L)}$.\\
\\
Restricting $f$ to $f|_{f^{-1}U}: f^{-1}U\to U$ preserves all assumed properties of the morphism, so that the same argument applies to arbitrary affinoid subdomains $U$ of $Y$. Thus the natural morphism $\wideparen{\mathscr{U}(f_*\mathscr{L})}\to f_*\wideparen{\mathscr{U}}$ is bijective on each affinoid subdomain of $Y$, and is hence an isomorphism.\\
Since the sheaf of Fr\'echet completed enveloping algebras on an affinoid $K$-variety is a global Fr\'echet--Stein sheaf by Proposition \ref{recallglobal}, the last statement follows immediately.
\end{proof}

We can also read off from the above discussion that $\check{\mathrm{H}}^j(\mathfrak{V}, \mathscr{U}_n)$ is a finitely generated $\mathscr{U}_n(X)$-module for any $j\geq 0$, which can be seen as a first partial result in the direction of Theorem \ref{KiehlDXfree}.\\
\\
We conclude this section by constructing $\mathscr{U}_n$-modules $\mathcal{M}_n$ such that $\varprojlim \mathcal{M}_n\cong \mathcal{M}$.\\
\\
In \cite[subsection 4.3]{Bodecompl}, we showed that $\mathcal{M}|_{U_i}$ can be written as the inverse limit of sheaves $\mathcal{M}_{n, i}$ on $U_{i, \ac}(\pi^n\mathcal{L}_i)$, given by 
\begin{equation*}
U\mapsto \mathscr{U}_n(U)\otimes_{\wideparen{\mathscr{U}}(U_i)}\mathcal{M}(U_i).
\end{equation*}
Note that then by definition of $\wideparen{\otimes}$,
\begin{align*}
\mathcal{M}_{n, i}(U)& = \mathscr{U}_n(U)\otimes_{\wideparen{\mathscr{U}}(U)} (\wideparen{\mathscr{U}}(U)\wideparen{\otimes}_{\wideparen{\mathscr{U}}(U_i)} \mathcal{M}(U_i))\\
& = \mathscr{U}_n(U)\otimes_{\wideparen{\mathscr{U}}(U)} \mathcal{M}(U)
\end{align*}
for any $U\in U_{i, \ac}(\pi^n\mathcal{L})$.\\
Thus $\mathcal{M}_{n, i}$ agrees with $\mathcal{M}_{n, j}$ on $U_{ij}$, giving a sheaf $\mathcal{M}_n$ on the site $X_n$ defined in the proof of Proposition \ref{UisglFS}. Since $\varprojlim \mathcal{M}_{n, i}\cong \mathcal{M}|_{U_i}$, we see that $\varprojlim \mathcal{M}_n(U)\cong \mathcal{M}(U)$ for any admissibe open subspace $U$ of $X$.\\
\\
It follows from \cite[Theorem 4.16]{Bodecompl} and \cite[Tag 03F9]{Stacksproject} that $\mathfrak{U}$ resp. $\mathfrak{V}$ are admissible coverings such that if $U\in X_n$ is a finite intersection of sets in $\mathfrak{U}$ resp. $\mathfrak{V}$, then $\mathrm{H}^j(U, \mathcal{M}_n)=0$ for any $j>0$.\\
Thus applying \cite[Tag 03F7]{Stacksproject} gives
\begin{equation*}
\check{\mathrm{H}}^j(\mathfrak{U}, \mathcal{M}_n) \cong \mathrm{H}^j(X_n, \mathcal{M}_n)\cong \check{\mathrm{H}}^j(\mathfrak{V}, \mathcal{M}_n)
\end{equation*}
for any $j\geq 0$.\\
\\
Finally, we will see later that
\begin{equation*}
\mathrm{R}^jf_*\mathcal{M}(Y)=\mathrm{H}^j(X, \mathcal{M})\cong \varprojlim \check{\mathrm{H}}^j(\mathfrak{V}, \mathcal{M}_n),
\end{equation*}
so we have found natural candidates exhibiting the coadmissibility of $\mathrm{H}^j(X, \mathcal{M})$.

\section{Global sections}
In particular, we can reduce our problem to a `Noetherian' setup. For the global sections, we wish to show the following.
\begin{enumerate}[(i)]
\item For each $j\geq 0$ and each $n$, $\check{\mathrm{H}}^j(\mathfrak{V}, \mathcal{M}_n)$ is a finitely generated $\mathscr{U}_n(X)$-module.
\item The natural morphism
\begin{equation*}
\mathscr{U}_n(X)\otimes_{\mathscr{U}_{n+1}(X)} \check{\mathrm{H}}^j(\mathfrak{V}, \mathcal{M}_{n+1})\to \check{\mathrm{H}}^j(\mathfrak{V}, \mathcal{M}_n)
\end{equation*}
is an isomorphism of $\mathscr{U}_n(X)$-modules.
\item The natural morphism
\begin{equation*}
\check{\mathrm{H}}^j(\mathfrak{V}, \mathcal{M})\to \varprojlim \check{\mathrm{H}}^j(\mathfrak{V}, \mathcal{M}_n)
\end{equation*}
is an isomorphism of $\wideparen{\mathscr{U}}(X)$-modules.
\end{enumerate}

The argument for (i) will rest on the discussion in section 2 and be analogous to the argument in \cite{Kiehl}, while (ii) will be established through an application of Theorem \ref{fullcomplexstrict}. The last statement (iii) will then follow easily from property (ii) in Proposition \ref{coadproperties}. \\
\\
Recall the commutative diagram of $A$-modules
\begin{equation*}
\begin{xy}
\xymatrix{
A\langle x_1, \dots, x_l\rangle \ar[rd]^{h_{i_1\dots i_j}} \ar[d]\\
\mathcal{O}_X(U_{i_1\dots i_j})\ar[r] & \mathcal{O}_X(V_{i_1\dots i_j})
}
\end{xy}
\end{equation*}
induced from the definition of properness, where $h_{i_1\dots i_j}(x_m)$ is topologically nilpotent in $\mathcal{O}_X(V_{i_1\dots i_j})$ for each $m=1, \dots, l$. \\
\\
Equip $A\langle x_1, \dots x_l\rangle$ with the natural residue norm (i.e. with unit ball $\mathcal{A}\langle x\rangle$), and recall that in the proof of Proposition \ref{UisglFS} we have already chosen residue norms for the other terms given by $\mathcal{L}$-stable affine formal models as unit balls, which turns the above into a diagram in $\Ban_A$.\\
Now apply the functor $\mathscr{U}_n(X)\widehat{\otimes}_A-$ to the diagram to obtain
\begin{equation*}
\begin{xy}
\xymatrix{
\mathscr{U}_n(X)\widehat{\otimes}_A A\langle x_1, \dots, x_l\rangle \ar[d]_{\theta'} \ar[rd]^{h'}\\
\mathscr{U}_n(X)\widehat{\otimes}_A \mathcal{O}(U_{i_1\dots i_j})\ar[r] & \mathscr{U}_n(X)\widehat{\otimes}_A \mathcal{O}_X(V_{i_1\dots i_j})
}
\end{xy}
\end{equation*}
which is a commutative diagram in $\Ban_{\mathscr{U}_n(X)}$.\\
Note that $h'$ is no longer a homomorphism of algebras, but only of left Banach $\mathscr{U}_n(X)$-modules. It inherits from $h_{i_1\dots i_j}$ the property that
\begin{equation*}
h'(x^r)=(\mathscr{U}_n(X)\widehat{\otimes}_A h_{i_1\dots i_j})(x^r)
\end{equation*}
tends to zero as $|r|\to \infty$, so Corollary \ref{sccpowerseries} implies that $h'$ is strictly completely continuous in $\Ban_{\mathscr{U}_n(X)}$.\\
\\
Now note that by \cite[Proposition 2.1.7/4]{BGR}
\begin{equation*}
\mathscr{U}_n(X)\widehat{\otimes}_A \mathcal{O}_X(U_{i_1\dots i_j})\cong U_A(L)\widehat{\otimes}_A \mathcal{O}_X(U_{i_1\dots i_j}),
\end{equation*}
where $U_A(L)$ is equipped with the norm with unit ball $U(\pi^n\mathcal{L})$. \\
Thus $\mathscr{U}_n(X)\widehat{\otimes}_A\mathcal{O}_X(U_{i_1\dots i_j})\cong\mathscr{U}_n(U_{i_1\dots i_j})$.\\
\\
The corresponding statement holds for $V_{i_1\dots i_j}$, and the horizontal map between the two terms is simply the restriction map.\\
\\
Thus we can read the above diagram as
\begin{equation*}
\begin{xy}
\xymatrix{
\mathscr{U}_n(X)\widehat{\otimes}_A A\langle x_1, \dots, x_l\rangle \ar[d]_{\theta'} \ar[rd]^{h'}\\
\mathscr{U}_n(U_{i_1\dots i_j})\ar[r]^{\res} & \mathscr{U}_n(V_{i_1\dots i_j})
}
\end{xy}
\end{equation*}
where $h'$ is strictly completely continuous.

\begin{lemma}
\label{fgcohomDn}
If $\mathcal{M}$ is a left coadmissible $\wideparen{\mathscr{U}}$-module, then $\check{\mathrm{H}}^j(\mathfrak{V}, \mathcal{M}_n)$ is a finitely generated $\mathscr{U}_n(X)$-module for all $j\geq 0$.
\end{lemma}
\begin{proof}
By functoriality, both $\theta'$ and $h'$ are maps in $\Ban_{\mathscr{U}_n(X)}$. Likewise, the restriction maps are naturally morphisms in $\Ban_{\mathscr{U}_n(X)}$. \\
\\
By \cite[Proposition 2.1.8/6]{BGR}, the map $\theta'$ is a strict surjection in $\Ban_{\mathscr{U}_n(X)}$. \\
We have thus shown that all the maps in the diagram are in $\Ban_{\mathcal{U}_n(X)}$, the arrow on the left is surjective, and $h'$ is strictly completely continuous.\\
\\
We now verify the conditions of Proposition \ref{cficomplex} by following the corresponding steps from the proof of Theorem \ref{Kiehlthm} as in \cite{Kiehl}.\\
Since $\mathcal{M}_n(U_{i_1\dots i_j})$ is finitely generated over $\mathscr{U}_n(U_{i_1\dots i_j})$, it is equipped with a canonical topology, making it an object in $\Ban_{\mathscr{U}_n(U_{i_1\dots i_j})}$ and hence a fortiori in $\Ban_{\mathscr{U}_n(X)}$. All the restriction maps are naturally continuous, so the \v{C}ech complexes $\check{C}^{\bullet}(\mathfrak{U}, \mathcal{M}_n)$ and $\check{C}^{\bullet}(\mathfrak{V}, \mathcal{M}_n)$ are cochain complexes in $\Ban_{\mathscr{U}_n(X)}$.\\
By construction, we have 
\begin{equation*}
\mathcal{M}_n(V_{i_1\dots i_j})\cong \mathscr{U}_n(V_{i_1\dots i_j})\otimes_{\mathscr{U}_n(U_{i_1\dots i_j})} \mathcal{M}_n(U_{i_1\dots i_j}) ,
\end{equation*}
so that finite generation induces a commutative diagram in $\Ban_{\mathscr{U}_n(X)}$
\begin{equation*}
\begin{xy}
\xymatrix{
\mathscr{U}_n(U_{i_1\dots i_j})^{\oplus r} \ar[r] \ar[d] & \mathscr{U}_n(V_{i_1\dots i_j})^{\oplus r}\ar[d] \\
\mathcal{M}_n(U_{i_1\dots i_j})\ar[r]^{\text{res}} & \mathcal{M}_n(V_{i_1\dots i_j})
}
\end{xy}
\end{equation*}
where both vertical maps are surjections and $r$ is the size of some finite generating set.\\
\\
Attaching this to $r$ copies of the previous diagram, we obtain 
\begin{equation*}
\begin{xy}
\xymatrix{
\mathscr{U}_n(X)\widehat{\otimes}_A A\langle x_1, \dots, x_l\rangle^{\oplus r} \ar[d] \ar[rd]\\
\mathscr{U}_n(U_{i_1\dots i_j})^{\oplus r} \ar[r] \ar[d] & \mathscr{U}_n(V_{i_1\dots i_j})^{\oplus r}\ar[d]\\
\mathcal{M}_n(U_{i_1\dots i_j})\ar[r] & \mathcal{M}_n(V_{i_1\dots i_j})
}
\end{xy}
\end{equation*}
Writing $G_{i_1\dots i_j}:=(\mathscr{U}_n(X)\widehat{\otimes}_A A\langle x_1, \dots x_l\rangle)^{\oplus r}$ and $\beta'_{i_1\dots i_j}: G_{i_1\dots i_j}\to \mathcal{M}_n(U_{i_1\dots i_j})$ for the surjective morphism on the left-hand side of the diagram, we can invoke Lemma \ref{directsumscc} and Lemma \ref{sccprop} to see that 
\begin{equation*}
\text{res}\circ \beta'_{i_1\dots i_j}: G_{i_1\dots i_j}\to \mathcal{M}_n(V_{i_1\dots i_j})
\end{equation*}
is strictly completely continuous in $\Ban_{\mathscr{U}_n(X)}$, by commutativity of the diagram.\\
\\
Fixing $j$ and summing over all different $U_{i_1\dots i_j}$, Lemma \ref{directsumscc} thus implies that 
\begin{equation*}
\beta_j: F^j:=\oplus G_{i_1\dots i_j}\to \oplus \mathcal{M}_n(U_{i_1\dots i_j})=\check{C}^j(\mathfrak{U}, \mathcal{M}_n)
\end{equation*}
is a surjection in $\Ban_{\mathscr{U}_n(X)}$ with the property that $\text{res}\circ \beta_j$ is strictly completely continuous.\\
But $\res: \check{C}^j(\mathfrak{U}, \mathcal{M}_n)\to \check{C}^j(\mathfrak{V}, \mathcal{M}_n)$ induces an isomorphism on the level of cohomology groups, as seen in the previous section. Thus we have verified that Proposition \ref{cficomplex} applies, proving the result.
\end{proof}

We note that it now follows from Corollary \ref{cficomplexgivesstrict} that $\check{C}^{\bullet}(\mathfrak{V}, \mathcal{M}_n)$ consists of strict morphisms.\\
\\
In general, we see that the part of Theorem \ref{KiehlDXfree} which is concerned with certain finiteness properties is still very close to the proof of Theorem \ref{Kiehlthm}. The only additional difficulties here lie in passing to sheaves $\mathscr{U}_n$ and $\mathcal{M}_n$ whose structure is more `finite' than that of the original sheaves, and analyzing some easy completed tensor products.\\
\\
Note however that there remains an additional property to be checked which has no counterpart in Theorem \ref{Kiehlthm}. We need to show that the finite components which we have exhibited match up in the right way, that is to say
\begin{equation*}
\mathscr{U}_n(X)\otimes_{\mathscr{U}_{n+1}(X)} \check{\mathrm{H}}^j(\mathfrak{V}, \mathcal{M}_{n+1})\cong \check{\mathrm{H}}^j(\mathfrak{V}, \mathcal{M}_n).
\end{equation*}
Replacing $\mathcal{L}$ by $\pi^n\mathcal{L}$, it is enough to consider the case $n=0$.\\
\\
Recall that $\mathscr{U}_0(X)=\widehat{U(\mathcal{L})}_K$ is flat over $\mathscr{U}_1(X)$ by \cite[Theorem 6.6]{Ardakov1}, so we know that
\begin{equation*}
\mathscr{U}_0(X)\otimes_{\mathscr{U}_1(X)}\check{\mathrm{H}}^j(\mathfrak{V}, \mathcal{M}_1)\cong \mathrm{H}^j(\mathscr{U}_0(X)\otimes \check{C}^\bullet(\mathfrak{V}, \mathcal{M}_1)).
\end{equation*}
Our first goal will be to show that the isomorphism claimed above can be viewed as a $\widehat{\otimes}$-version of this statement.

\begin{lemma}
\label{sectionsbycompleteloc}
If $V$ is an admissible open affinoid subspace of $X$ and $V\in X_0$, then the natural map
\begin{equation*}
\mathscr{U}_0(X)\widehat{\otimes}_{\mathscr{U}_1(X)} \mathcal{M}_1(V)\to \mathcal{M}_0(V)
\end{equation*}
is an isomorphism of $\mathscr{U}_0(X)$-modules.
\end{lemma}
\begin{proof}
Using Lemma \ref{removehat} and associativity of the completed tensor product (see \cite[Proposition 2.1.7/6]{BGR}), we have the following chain of isomorphisms
\begin{align*}
\mathcal{M}_0(V)& \cong \mathscr{U}_0(V)\otimes_{\mathscr{U}_1(V)} \mathcal{M}_1(V)\\
&\cong\mathscr{U}_0(V)\widehat{\otimes}_{\mathscr{U}_1(V)} \mathcal{M}_1(V)\\
& \cong \left( \widehat{U(\mathcal{L})}_K\widehat{\otimes}_A \mathcal{O}_X(V)\right) \widehat{\otimes}_{\mathscr{U}_1(V)} \mathcal{M}_1(V)\\
&\cong \left(\widehat{U(\mathcal{L})}_K\widehat{\otimes}_{\widehat{U(\pi\mathcal{L})}_K} \widehat{U(\pi\mathcal{L})}_K \widehat{\otimes}_A \mathcal{O}_X(V)\right) \widehat{\otimes}_{\mathscr{U}_1(V)} \mathcal{M}_1(V)\\
& \cong \widehat{U(\mathcal{L})}_K\widehat{\otimes}_{\widehat{U(\pi\mathcal{L})}_K} \left(\left(\widehat{U(\pi\mathcal{L})}_K\widehat{\otimes}_A \mathcal{O}_X(V)\right) \widehat{\otimes}_{\mathscr{U}_1(V)} \mathcal{M}_1(V)\right) \\
&\cong\mathscr{U}_0(X)\widehat{\otimes}_{\mathscr{U}_1(X)}\left(\mathscr{U}_1(V)\widehat{\otimes}_{\mathscr{U}_1(V)} \mathcal{M}_1(V)\right)\\
&\cong \mathscr{U}_0(X)\widehat{\otimes}_{\mathscr{U}_1(X)} \mathcal{M}_1(V),
\end{align*}
as required.
\end{proof}

To continue in our proof of Theorem \ref{KiehlDXfree}, we therefore wish to show that
\begin{equation*}
\mathscr{U}_0(X)\widehat{\otimes}_{\mathscr{U}_1(X)} \check{\mathrm{H}}^j(\mathfrak{V}, \mathcal{M}_1) \cong \mathrm{H}^j(\mathscr{U}_0(X)\widehat{\otimes}_{\mathscr{U}_1(X)} \check{C}^{\bullet}(\mathfrak{V}, \mathcal{M}_1)).
\end{equation*}
This will be achieved by checking all the conditions in Corollary \ref{noethapplication}, where the role of $A^\circ$ is played by $U(\pi \mathcal{L})$, and that of $U^\circ$ by $U(\mathcal{L})$.

\begin{lemma}
\label{changeringnoeth}
$U(\mathcal{L})\otimes_{U(\pi\mathcal{L})} \widehat{U(\pi\mathcal{L})}$ carries a natural ring structure, making it a left and right Noetherian ring.
\end{lemma}
\begin{proof}
By freeness of $\mathcal{L}$, we have a natural injection $U_{\mathcal{A}}(\mathcal{L})\to U_A(L)$ (by Rinehart's Theorem, \cite[Theorem 3.1]{Rinehart}). Since $\widehat{U(\pi\mathcal{L})}$ is flat over $U(\pi\mathcal{L})$ by \cite[3.2.3.(iv)]{Berthelot}, we thus can view $U(\mathcal{L})\otimes_{U(\pi\mathcal{L})} \widehat{U(\pi\mathcal{L})}$ as a subset of
\begin{equation*}
U_A(L)\otimes_{U(\pi\mathcal{L})} \widehat{U(\pi\mathcal{L})}=U_A(L)\otimes_{U_A(L)} \widehat{U(\pi\mathcal{L})}_K=\widehat{U(\pi\mathcal{L})}_K,
\end{equation*}
identifying it with $U(\mathcal{L})\cdot \widehat{U(\pi\mathcal{L})}$. We will show that this is a Noetherian subring of $\widehat{U(\pi\mathcal{L})}_K$.\\
\\
Since $[\mathcal{L}, \pi\mathcal{L}]\subseteq \pi\mathcal{L}$, an easy inductive argument shows that $[\mathcal{L}, U(\pi\mathcal{L})]\subseteq U(\pi\mathcal{L})$, where the commutator is understood in $U_A(L)$.\\
Hence we have that for each $\partial\in \mathcal{L}$, the commutator map
\begin{equation*}
[\partial, -]: U_A(L)\to U_A(L)
\end{equation*}
preserves $U(\pi\mathcal{L})$, i.e. is a bounded linear map on $U_A(L)$ with unit ball $U(\pi\mathcal{L})$, where we can take $1$ as a bound.\\
Thus passing to the completion, $[\mathcal{L}, \widehat{U(\pi\mathcal{L})}]\subseteq \widehat{U(\pi\mathcal{L})}$, where the commutator is understood in $\widehat{U(\pi\mathcal{L})}_K$. Therefore
\begin{equation*}
U(\mathcal{L})\cdot \widehat{U(\pi\mathcal{L})}
\end{equation*}
is a subring of $\widehat{U(\pi\mathcal{L})}_K$ by another easy induction argument, as required.\\
Denote this ring by $\mathcal{E}$.\\
\\
Let $F_\bullet U(\mathcal{L})$ be the usual degree filtration on $U_{\mathcal{A}}(\mathcal{L})$. Then $\mathcal{E}$ is filtered by
\begin{equation*}
F'_i\mathcal{E}=F_iU(\mathcal{L})\cdot \widehat{U(\pi\mathcal{L})},
\end{equation*}
such that the following is satisfied:
\begin{enumerate}[(i)]
\item $F'_0\mathcal{E}=\widehat{U(\pi\mathcal{L})}$.
\item $F'_i\mathcal{E}\cdot F'_j\mathcal{E}\subseteq F'_{i+j}\mathcal{E}$, by reiterating the above commutator expression.
\end{enumerate}
Just as in the proof of \cite[Theorem 3.5]{Bodecompl}, the associated graded ring $\text{gr}'\mathcal{E}$ is generated by finitely many central elements over the zeroth piece $\widehat{U(\pi\mathcal{L})}$, which is Noetherian by Rinehart's Theorem and \cite[3.2.3.(vi)]{Berthelot}.\\
Thus $\mathcal{E}$ is a Noetherian ring by \cite[Corollary D.IV.5]{Oyst}.
\end{proof}

We thus have confirmed that the first condition in Corollary \ref{noethapplication} is satisfied. It remains to show that the relevant Tor groups have bounded $\pi$-torsion.\\
\\
Write $\mathcal{O}_X(V_{i_1\dots i_j})=B$, and let $\mathcal{B}=\mathcal{B}_{V_{i_1\dots i_j}}$ be an $\mathcal{L}$-stable affine formal model, as discussed in the previous section.\\
Denote by $U_{\mathcal{B}}$ the Noetherian ring $U_{\mathcal{B}}(\mathcal{B}\otimes_{\mathcal{A}}\pi \mathcal{L})$ and by $\widehat{U_{\mathcal{B}}}$ its $\pi$-adic completion. Note that this is the unit ball of $\mathscr{U}_1(V_{i_1\dots i_j})$.\\
Similarly to the above, we have the following lemma.
\begin{lemma}
\label{locringnoeth}
$U(\mathcal{L})\otimes_{U(\pi\mathcal{L})} \widehat{U_{\mathcal{B}}}$ carries a natural ring structure, making it a left and right Noetherian ring.
\end{lemma}
\begin{proof}
As before, we identify the tensor product with a certain subset of a $K$-algebra.\\
Since $U(\mathcal{L})$ is flat over $\mathcal{A}$, we have an injection 
\begin{equation*}
U(\mathcal{L})\otimes_{U(\pi\mathcal{L})} U(\pi\mathcal{L}) \otimes_{\mathcal{A}} \mathcal{B}=U(\mathcal{L})\otimes_{\mathcal{A}} \mathcal{B}\to U(\mathcal{L})\otimes_{\mathcal{A}} B=U_A(L)\otimes_A B.
\end{equation*}
As $U(\pi\mathcal{L})\otimes \mathcal{B}\cong U(\mathcal{B}\otimes \pi\mathcal{L})$, and $\widehat{U(\mathcal{B}\otimes\pi\mathcal{L})}$ is flat over $U(\mathcal{B}\otimes \pi\mathcal{L})$ by \cite[3.2.3.(iv)]{Berthelot}, this induces an injective map
\begin{equation*}
U(\mathcal{L})\otimes_{U(\pi\mathcal{L})} \widehat{U(\mathcal{B}\otimes \pi\mathcal{L})}\to U(B\otimes L)\otimes_{U(B\otimes L)} \widehat{U(\mathcal{B}\otimes \pi\mathcal{L})}_K,
\end{equation*}
and the right-hand side is clearly just $ \widehat{U(\mathcal{B}\otimes \pi\mathcal{L})}_K= \widehat{U_{\mathcal{B}}}\otimes K$. The map above identifies $U(\mathcal{L})\otimes \widehat{U_{\mathcal{B}}}$ with $U(\mathcal{L})\cdot \widehat{U_{\mathcal{B}}}$ in this algebra. \\
Since $\mathcal{B}$ is $\mathcal{L}$-stable, we can repeat the argument in Lemma \ref{changeringnoeth} to show this is a Noetherian subring of $\widehat{U(\mathcal{B}\otimes \pi\mathcal{L})}_K$.
\end{proof}

\begin{lemma}
\label{torgpsvanishn}
Let $N$ be a finitely generated $\widehat{U_{\mathcal{B}}}$-module.\\
Then the module $\Tor^{U(\pi\mathcal{L})}_s(U(\mathcal{L}), N)$ has bounded $\pi$-torsion for each $s\geq 0$.\\
Thus $\Tor^{U(\pi\mathcal{L})}_s(U(\mathcal{L}), \check{C}^j(\mathfrak{V}, \mathcal{M}_1)^\circ)$ has bounded $\pi$-torsion for each $s\geq 0$ and each $j$.
\end{lemma}
\begin{proof}
We abbreviate the functor $\Tor^{U(\pi\mathcal{L})}_s(U(\mathcal{L}), -)$ to $T_s(-)$.\\
By Noetherianity, we have a short exact sequence
\begin{equation*}
0 \to N' \to \widehat{U_{\mathcal{B}}}^{\oplus r} \to N \to 0
\end{equation*}
for some integer $r$ and some finitely generated $\widehat{U_{\mathcal{B}}}$-module $N'$.\\
We will use this to prove the lemma inductively via the corresponding long exact sequence.\\
\\
For $s=0$, we have that 
\begin{equation*}
U(\mathcal{L})\otimes_{U(\pi\mathcal{L})} N=U(\mathcal{L})\otimes_{U(\pi\mathcal{L})} \widehat{U_\mathcal{B}} \otimes_{\widehat{U_\mathcal{B}}}N
\end{equation*}
is a finitely generated $U(\mathcal{L})\otimes \widehat{U_\mathcal{B}}$-module and hence has bounded $\pi$-torsion by Noetherianity (see Lemma \ref{locringnoeth}).\\
\\
Next, we show that 
\begin{equation*}
\Tor^{U(\pi\mathcal{L})}_s(U(\mathcal{L}), \widehat{U_{\mathcal{B}}})=0
\end{equation*}
for $s\geq 1$.\\
For this note that by flatness of $U(\mathcal{L})$ and $U(\pi\mathcal{L})$ over $\mathcal{A}$, we have
\begin{equation*}
0=\Tor^{\mathcal{A}}_s(U(\mathcal{L}), \mathcal{B})=\Tor^{U(\pi\mathcal{L})}_s(U(\mathcal{L}), U(\pi\mathcal{L})\otimes_{\mathcal{A}} \mathcal{B}),
\end{equation*}
using \cite[Proposition 3.2.9]{Weibel}. Therefore, $\Tor^{U(\pi\mathcal{L})}_s(U(\mathcal{L}), U_{\mathcal{B}})=0$.\\
As moreover $\widehat{U_{\mathcal{B}}}=\widehat{U(\mathcal{B}\otimes \pi\mathcal{L})}$ is flat over $U_{\mathcal{B}}$ by \cite[3.2.3.(iv)]{Berthelot}, we obtain
\begin{equation*}
\text{Tor}^{U(\pi\mathcal{L})}_s(U(\mathcal{L}), \widehat{U_{\mathcal{B}}})=0
\end{equation*}
for $s\geq 1$ by \cite[Corollary 3.2.10]{Weibel}.\\
\\
Thus, the long exact sequence
\begin{equation*}
\dots \to T_s(N')\to T_s\left(\widehat{U_{\mathcal{B}}}\right)^{\oplus r} \to T_s(N)\to T_{s-1}(N')\to \dots
\end{equation*}
shows that $T_s(N)\to T_{s-1}(N')$ is an injection for all $s$, and an isomorphism for $s\geq 2$.\\
So if we suppose that $T_{s-1}(N)$ has bounded $\pi$-torsion for any finitely generated $\widehat{U_{\mathcal{B}}}$-module $N$, this holds in particular for $N'$, proving that $T_s(N)$ has bounded $\pi$-torsion as well.\\
By induction, this finishes the proof of the first statement.\\
\\
Now $\mathcal{M}_1(V_{i_1\dots i_j})^\circ$ is a finitely generated $\widehat{U_{\mathcal{B}}}=\widehat{U(\mathcal{B}\otimes \pi\mathcal{L})}$-module by Lemma \ref{fgunitball}, so taking the corresponding finite direct sum to form the $j$th term of the \v{C}ech complex proves the result.
\end{proof}

\begin{theorem}
\label{matchupn}
The natural morphism
\begin{equation*}
\mathscr{U}_n(X)\otimes_{\mathscr{U}_{n+1}(X)} \check{\mathrm{H}}^j(\mathfrak{V}, \mathcal{M}_{n+1})\to \mathrm{H}^j(\mathscr{U}_n(X)\widehat{\otimes}_{\mathscr{U}_{n+1}(X)} \check{C}^\bullet(\mathfrak{V}, \mathcal{M}_{n+1}))
\end{equation*}
is an isomorphism of $\mathscr{U}_n(X)$-modules for each $n\geq 0$, $j\geq 0$.
\end{theorem}
\begin{proof}
Without loss of generality, we can assume $n=0$.\\
Then the theorem is precisely Corollary \ref{noethapplication} applied to the \v{C}ech complex $\check{C}^\bullet(\mathfrak{V}, \mathcal{M}_1)$. This is a finite cochain complex in $\text{Ban}_{\mathscr{U}_1(X)}$ with strict morphisms by Corollary \ref{cficomplexgivesstrict}.\\
By Lemma \ref{changeringnoeth}, $U(\mathcal{L})\otimes_{U(\pi\mathcal{L})} \widehat{U(\pi\mathcal{L})}$ is a Noetherian ring, $\check{\mathrm{H}}^j(\mathfrak{V}, \mathcal{M}_1)$ is a finitely generated $\mathscr{U}_1(X)=\widehat{U(\pi\mathcal{L})}_K$-module by Theorem \ref{fgcohomDn}, and Lemma \ref{torgpsvanishn} ensures bounded $\pi$-torsion for each of the Tor groups. Thus Corollary \ref{noethapplication} states that
\begin{equation*}
U_A(L)'\otimes_{U_A(L)} \check{C}^\bullet(\mathfrak{V}, \mathcal{M}_1)
\end{equation*}
is a strict complex, where $U(L)$ is equipped with the norm with unit ball $U(\pi\mathcal{L})$ and we write $U(L)'$ for $U(L)$ with unit ball $U(\mathcal{L})$. \\
Thus Corollary \ref{noethapplication} together with \cite[Proposition 2.1.7/4]{BGR} implies that 
\begin{align*}
\mathscr{U}_0(X)\otimes_{\mathscr{U}_1(X)} \check{\mathrm{H}}^j(\mathfrak{V}, \mathcal{M}_1)&= \widehat{U(L)'}\otimes_{\widehat{U(L)}} \check{\mathrm{H}}^j(\mathfrak{V}, \mathcal{M}_1)\\
& \cong \mathrm{H}^j(U(L)'\widehat{\otimes}_{U(L)} \check{C}^\bullet(\mathfrak{V}, \mathcal{M}_1))\\
&\cong \mathrm{H}^j(\mathscr{U}_0(X)\widehat{\otimes}_{\mathscr{U}_1(X)} \check{C}^{\bullet}(\mathfrak{V}, \mathcal{M}_1)),
\end{align*}
proving the result.
\end{proof}

\begin{corollary}
\label{cohomcoad}
The $\wideparen{\mathscr{U}}(X)$-module $\varprojlim \check{\mathrm{H}}^j(\mathfrak{V}, \mathcal{M}_n)$ is coadmissible for each $j\geq 0$.
\end{corollary}
\begin{proof}
Each module $\check{\mathrm{H}}^j(\mathfrak{V}, \mathcal{M}_n)$ is a finitely generated $\mathscr{U}_n(X)$-module by Theorem \ref{fgcohomDn}, and
\begin{equation*}
\mathscr{U}_n(X)\otimes_{\mathscr{U}_{n+1}(X)} \check{\mathrm{H}}^j(\mathfrak{V}, \mathcal{M}_{n+1})\cong \check{\mathrm{H}}^j(\mathfrak{V}, \mathcal{M}_n)
\end{equation*}
by the theorem above combined with the observation that
\begin{equation*}
\mathscr{U}_n(X)\widehat{\otimes}_{\mathscr{U}_{n+1}(X)} \check{C}^\bullet(\mathfrak{V}, \mathcal{M}_{n+1})= \check{C}^\bullet(\mathfrak{V}, \mathcal{M}_n)
\end{equation*}
by Lemma \ref{sectionsbycompleteloc}.
\end{proof}

Finally, fixing an integer $j$, we show that $\varprojlim \check{\mathrm{H}}^j(\mathfrak{V}, \mathcal{M}_n)$ gives indeed the corresponding higher direct image of $\mathcal{M}$.

\begin{proposition}
\label{projlimiso}
For each $j\geq 0$, the canonical morphism of $\wideparen{\mathscr{U}}(X)$-modules
\begin{equation*}
\mathrm{H}^j(X, \mathcal{M})\cong \varprojlim \check{\mathrm{H}}^j(\mathfrak{V}, \mathcal{M}_n)
\end{equation*}
is an isomorphism.
\end{proposition}
\begin{proof}
By Proposition \ref{coadproperties}.(ii), each system $(\check{C}^j(\mathfrak{V}, \mathcal{M}_n))_n$ satisfies the Mittag-Leffler property as described in \cite[0.13.2.4]{EGA}, and by Corollary \ref{cohomcoad}, so does the inverse system $(\check{\mathrm{H}}^j(\mathfrak{V}, \mathcal{M}_n))_n$. Hence we can apply \cite[Proposition 0.13.2.3]{EGA} to deduce that
\begin{equation*}
\check{\mathrm{H}}^j(\mathfrak{V}, \mathcal{M})=\mathrm{H}^j(\varprojlim \check{C}^\bullet(\mathfrak{V}, \mathcal{M}_n))\cong \varprojlim \check{\mathrm{H}}^j(\mathfrak{V}, \mathcal{M}_n).
\end{equation*}
Since $\mathcal{M}$ also has vanishing higher \v{C}ech cohomology on affinoids by the comment following \cite[Theorem 4.16]{Bodecompl}, we have $\mathrm{H}^j(X, \mathcal{M})\cong \check{\mathrm{H}}^j(\mathfrak{V}, \mathcal{M})$ by \cite[Tag 03F7]{Stacksproject}, and the result follows.
\end{proof}

This concludes the proof that $\mathrm{R}^jf_*\mathcal{M}(Y)=\mathrm{H}^j(X, \mathcal{M})$ is a coadmissible $\wideparen{\mathscr{U}}(X)$-module.

\section{Localization}
It remains to show that other sections of the sheaf are obtained by localization, that is if $U=\Sp B\subseteq Y$ is an affinoid subdomain, we want to show that
\begin{equation*}
\wideparen{\mathscr{U}}(f^{-1}U)\wideparen{\otimes}_{\wideparen{\mathscr{U}}(X)} \mathrm{H}^j(X, \mathcal{M}):=\varprojlim \left(\mathscr{U}_n(f^{-1}U)\otimes_{\mathscr{U}_n(X)} \mathrm{H}^j(X, \mathcal{M}_n)\right)\cong \mathrm{H}^j(f^{-1}U, \mathcal{M})
\end{equation*}
via the natural morphism.\\
Similarly to the above, our strategy will consist in a reduction to the Noetherian components of the coadmissible module and an argument involving completed tensor products similar to Theorem \ref{matchupn}.\\
\\
Recall from Lemma \ref{pushforwalgebroid} that $f_*\mathscr{L}$ is a Lie algebroid on $Y$ with $f_*\mathscr{L}(Y)=L$. In particular, we can talk about the $\pi^n\mathcal{L}$-\textbf{accessible} affinoid subdomains of $Y$, as introduced in Definition \ref{rationalac}.\\
\\
Our plan looks as follows.
\begin{itemize}
\item Step A: Consider the case where $U$ is a rational subdomain which is $\pi^n\mathcal{L}$-accessible in one step. As $B=\mathcal{O}_Y(U)$ can be described as a quotient of $A\langle t\rangle$, we will establish some properties relating to $A\langle t\rangle$ before passing to the quotient via some homological algebra.
\item Step B: An easy inductive argument extends the result to any $\pi^n\mathcal{L}$-accessible rational subdomain.
\item Step C: Passing to suitable coverings and arguing locally, we can generalize to arbitrary $\pi^n\mathcal{L}$-accessible affinoid subdomains. Since any affinoid subdomain is $\pi^n\mathcal{L}$-accessible for sufficiently large $n$ (\cite[Proposition 7.6]{Ardakov1}), this finishes the proof. 
\end{itemize}

\subsection*{Step A}
Let $x\in A$ be non-zero such that $\pi^n\mathcal{L}\cdot x\subseteq \mathcal{A}$, and consider 
\begin{equation*}
Y_1=Y(x)=\Sp B_1, \ Y_2=Y(x^{-1})=\Sp B_2.
\end{equation*} 
We adopt the notation from \cite{Ardakov1}. Let $a$ be a positive integer satisfying $\pi^ax\in \mathcal{A}$, and define 
\begin{equation*}
u_1=\pi^ax-\pi^at\in \mathcal{A}\langle t\rangle, \ u_2=\pi^axt-\pi^a\in \mathcal{A}\langle t\rangle.
\end{equation*}
Then we have short exact sequences
\begin{equation*}
\begin{xy}
\xymatrix{
0 \ar[r]& A\langle t\rangle \ar[r]^{u_i\cdot} &  A\langle t \rangle \ar[r]^{\rho} &  B_i \ar[r] & 0
}
\end{xy}
\end{equation*}
for $i=1, 2$ by \cite[Lemma 4.1]{Ardakov1}.\\
We define $\mathcal{B}_i=\mathcal{A}\langle t\rangle/u_i\mathcal{A}\langle t \rangle$. Then $\overline{\mathcal{B}_i}=\mathcal{B}_i/\pi\mathrm{-tor}(\mathcal{B}_i)=\rho(\mathcal{A}\langle t\rangle)$ is a $\pi^n\mathcal{L}$-stable affine formal model in $B_i$, see \cite[Lemma 4.3]{Ardakov1}.\\
\\
By the definition of $\wideparen{\otimes}$ and Proposition \ref{projlimiso}, it will be enough to show that the natural morphism
\begin{equation*}
\mathscr{U}_n(f^{-1}Y_i)\otimes_{\mathscr{U}_n(X)} \check{\mathrm{H}}^j(\mathfrak{V}, \mathcal{M}_n)\to \check{\mathrm{H}}^j(\mathfrak{V}\cap f^{-1}Y_i, \mathcal{M}_n)
\end{equation*}
is an isomorphism for $i=1, 2$.\\
\\
First recall that by Kiehl's Proper Mapping Theorem, $\mathcal{O}_X(f^{-1}Y_i)=B_i$, and
\begin{equation*}
\mathscr{U}_n(f^{-1}Y_i)=\widehat{U(\overline{\mathcal{B}_i}\otimes \pi^n\mathcal{L})}_K\cong B_i\widehat{\otimes}_A \mathscr{U}_n(X)
\end{equation*}
by Corollary \ref{pushffsalg}. \\
\\
Note that for any admissible open affinoid subspace $V\subseteq X$, we have 
\begin{equation*}
f^{-1}Y_i\cap V=\Sp B_i\times_{\Sp A} \Sp \mathcal{O}_X(V)=\Sp \left(B_i\widehat{\otimes}_A \mathcal{O}_X(V)\right)
\end{equation*}
by \cite[Proposition 7.1.4/4]{BGR}, and hence
\begin{align*}
\mathcal{M}_n(f^{-1}Y_i\cap V)&=\mathscr{U}_n(f^{-1}Y_i\cap V)\otimes_{\mathscr{U}_n(V)} \mathcal{M}_n(V)\\
&= \mathscr{U}_n(f^{-1}Y_i\cap V)\widehat{\otimes}_{\mathscr{U}_n(V)} \mathcal{M}_n(V)\\
&=\left((B_i\widehat{\otimes}_A \mathcal{O}_X(V))\widehat{\otimes}_{\mathcal{O}_X(V)} \mathscr{U}_n(V)\right)\widehat{\otimes}_{\mathscr{U}_n(V)} \mathcal{M}_n(V)\\
&=(B_i\widehat{\otimes}_A\mathcal{O}_X(V))\widehat{\otimes}_{\mathcal{O}_X(V)} \mathcal{M}_n(V)\\
&=B_i\widehat{\otimes}_A \mathcal{M}_n(V).
\end{align*} 
Thus
\begin{equation*}
\check{C}^\bullet(\mathfrak{V}\cap f^{-1}Y_i, \mathcal{M}_n)\cong B_i\widehat{\otimes}_A \check{C}^\bullet(\mathfrak{V}, \mathcal{M}_n),
\end{equation*}
which in turn can be written as $\mathscr{U}_n(f^{-1}Y_i)\widehat{\otimes}_{\mathscr{U}_n(X)} \check{C}^\bullet(\mathfrak{V}, \mathcal{M}_n)$ by the above.\\
\\
We thus wish to show that
\begin{equation*}
\mathscr{U}_n(f^{-1}Y_i)\widehat{\otimes}_{\mathscr{U}_n(X)} \check{\mathrm{H}}^j(\mathfrak{V}, \mathcal{M}_n)\cong \mathrm{H}^j(\mathscr{U}_n(f^{-1}Y_i)\widehat{\otimes}_{\mathscr{U}_n(X)} \check{C}^\bullet(\mathfrak{V}, \mathcal{M}_n)).
\end{equation*}
We will prove this isomorphism by a number of lemmas, mainly exploiting the short exact sequence
\begin{equation*}
0\to A\langle t\rangle \to A\langle t\rangle \to B_i\to 0
\end{equation*}
and our study of completed tensor products.\\
\\
Recall from \cite[Proposition 4.2]{Ardakov1} that the $\pi^n\mathcal{L}$-action on $\mathcal{A}$ lifts to an action $\sigma_i: \pi^n\mathcal{L}\to \Der_R(\mathcal{A}\langle t\rangle)$, turning $\mathcal{A}\langle t\rangle\otimes_{\mathcal{A}} \pi^n\mathcal{L}$ into an $(R, \mathcal{A}\langle t\rangle)$-Lie algebra.
We will write
\begin{equation*}
\mathscr{U}_n(X)\langle t\rangle_i=\widehat{U(\mathcal{A}\langle t\rangle \otimes_{\mathcal{A}} \pi^n\mathcal{L})}_K
\end{equation*}
for the corresponding completed enveloping algebra. 
\begin{lemma}
\label{tversionloc}
The natural map
\begin{equation*}
\mathscr{U}_n(X)\langle t\rangle_i \otimes_{\mathscr{U}_n(X)} \check{\mathrm{H}}^j(\mathfrak{V}, \mathcal{M}_n)\to \mathrm{H}^j(\mathscr{U}_n(X)\langle t\rangle_i \widehat{\otimes}_{\mathscr{U}_n(X)} \check{C}^\bullet(\mathfrak{V}, \mathcal{M}_n))
\end{equation*}
is an isomorphism of left $\mathscr{U}_n(X)\langle t\rangle_i$-modules for each $j\geq 0$.
\end{lemma}
\begin{proof}
We know that $\check{C}^\bullet(\mathfrak{V}, \mathcal{M}_n)$ is a finite cochain complex of Banach $\mathscr{U}_n(X)$-modules, a fortiori of Banach $A$-modules, with strict morphisms.\\
\\
As as right $\mathscr{U}_n(X)$-module, $\mathscr{U}_n(X)\langle t\rangle_i$ is isomorphic to 
\begin{equation*}
A\langle t\rangle \widehat{\otimes}_A \mathscr{U}_n(X)
\end{equation*}
by \cite[Proposition 2.3]{Ardakov1}, which is the completion of $A\langle t\rangle \otimes_A \mathscr{U}_n(X)$ with respect to the tensor product semi-norm with unit ball given by
\begin{equation*}
\mathcal{A}\langle t\rangle \otimes_{\mathcal{A}} \widehat{U(\pi\mathcal{L})}.
\end{equation*}
In particular, viewing the morphism in the statement of the lemma as a morphism of $A\langle t\rangle$-modules, it can be written as
\begin{equation*}
A\langle t\rangle \widehat{\otimes}_A \check{\mathrm{H}}^j(\mathfrak{V}, \mathcal{M}_n)\to \mathrm{H}^j(A\langle t\rangle \widehat{\otimes}_A \check{C}^\bullet(\mathfrak{V}, \mathcal{M}_n)).
\end{equation*}
Since $\mathcal{A}\langle t\rangle$ is flat over $\mathcal{A}$ by \cite[Remark 7.3/2]{Bosch}, this is an isomorphism of $A\langle t\rangle$-modules by \cite[Corollary 2.15]{Bodecompl} and hence a bijection. Thus the natural morphism
\begin{equation*}
\mathscr{U}_n(X)\langle t\rangle_i\widehat{\otimes}_{\mathscr{U}_n(X)} \check{\mathrm{H}}^j(\mathfrak{V}, \mathcal{M}_n)\to \mathrm{H}^j(\mathscr{U}_n(X)\langle t\rangle_i\widehat{\otimes} \check{C}^{\bullet}(\mathfrak{V}, \mathcal{M}_n))
\end{equation*}
is an isomorphism of $\mathscr{U}_n(X)\langle t\rangle_i$-modules.
\end{proof}

We now fix a finite set of indices $i_1, \dots, i_j$ and write $V=V_{i_1\dots i_j}$, $C=\mathcal{O}_X(V)$.

\begin{lemma}
\label{partialtensorbddpitor}
Let $\mathcal{C}$ be a $\pi^n\mathcal{L}$-stable affine formal model in $C$. Then
\begin{equation*}
\mathcal{B}_i\otimes_{\mathcal{A}} \widehat{U_{\mathcal{C}}(\mathcal{C}\otimes_{\mathcal{A}} \pi^n\mathcal{L})}
\end{equation*}
has bounded $\pi$-torsion.
\end{lemma}
\begin{proof}
Define $\mathcal{B}'_i=\mathcal{A}[t]/u_i\mathcal{A}[t]$. Note that the $\pi$-adic completion of the short exact sequence
\begin{equation*}
\begin{xy}
\xymatrix{
0\ar[r]& \mathcal{A}[t]\ar[r]^{u_i\cdot}& \mathcal{A}[t] \ar[r]& \mathcal{B}'_i\ar[r]& 0}
\end{xy}
\end{equation*}
is obtained by applying the exact functor $\mathcal{A}\langle t\rangle \otimes_{\mathcal{A}[t]}-$ by \cite[Theorem 7.2]{Eisenbud}. In particular, $\widehat{\mathcal{B}'_i}=\mathcal{B}_i$.\\
\\
Since $\mathcal{B}'_i$ is of finite type over $\mathcal{A}$, the ring $\mathcal{B}'_i\otimes_{\mathcal{A}} \mathcal{C}$ is of finite type over $\mathcal{C}$ and is hence Noetherian. In particular, it has bounded $\pi$-torsion.\\
\\
Tensoring with a flat module preserves the property of bounded $\pi$-torsion, as the $\pi$-torsion of an $R$-module $M$ is given by the kernel of the natural map $M\to M\otimes_R K$. Since $\mathcal{B}_i=\widehat{\mathcal{B}'_i}$, it is flat over $\mathcal{B}'_i$ by \cite[Theorem 7.2]{Eisenbud}. Moreover,
\begin{equation*}
U(\mathcal{C}\otimes_{\mathcal{A}} \pi\mathcal{L})\cong \mathcal{C}\otimes_{\mathcal{A}}U(\pi^n\mathcal{L})
\end{equation*}
is flat over $\mathcal{C}$, and $\widehat{U(\mathcal{C}\otimes_{\mathcal{A}} \pi^n\mathcal{L})}$ is flat over $U(\mathcal{C}\otimes_{\mathcal{A}} \pi^n\mathcal{L})$, again by \cite[3.2.3.(iv)]{Berthelot}.\\
Thus 
\begin{equation*}
\mathcal{B}_i\otimes_{\mathcal{A}} \widehat{U(\mathcal{C}\otimes_{\mathcal{A}} \pi^n\mathcal{L})}\cong\mathcal{B}_i\otimes_{\mathcal{B}'_i} (\mathcal{B}'_i\otimes_{\mathcal{A}} \mathcal{C}) \otimes_{\mathcal{C}} \widehat{U(\mathcal{C}\otimes \pi^n\mathcal{L})}
\end{equation*}
has bounded $\pi$-torsion, as required.
\end{proof}

\begin{corollary}
\label{bddpitorcorollary}
Let $\mathcal{C}$ be a $\pi^n\mathcal{L}$-stable affine formal model in $C$. Then 
\begin{equation*}
\overline{\mathcal{B}_i}\otimes_{\mathcal{A}}\widehat{U(\mathcal{C}\otimes_{\mathcal{A}} \pi^n\mathcal{L})}
\end{equation*}
has bounded $\pi$-torsion.
\end{corollary}
\begin{proof}
Writing $T=\pi \mathrm{-tor}(\mathcal{B}_i)$, we have a short exact sequence
\begin{equation*}
0\to T\to \mathcal{B}_i\to \overline{\mathcal{B}_i}\to 0.
\end{equation*}
Since $\mathcal{B}_i$ is Noetherian, $T$ is annihilated by some power of $\pi$, i.e. for some $r$ we have $\pi^rT=0$.\\
\\
Now consider the exact sequence
\begin{equation*}
T\otimes\widehat{U(\mathcal{C}\otimes_{\mathcal{A}} \pi^n\mathcal{L})}\to \mathcal{B}_i\otimes \widehat{U(\mathcal{C}\otimes_{\mathcal{A}} \pi^n\mathcal{L})}\to \overline{\mathcal{B}_i}\otimes \widehat{U(\mathcal{C}\otimes_{\mathcal{A}} \pi^n\mathcal{L})}\to 0.
\end{equation*}
If $x\in \overline{\mathcal{B}_i}\otimes \widehat{U(\mathcal{C}\otimes_{\mathcal{A}} \pi^n\mathcal{L})}$ is killed by $\pi^s$, then any preimage $y\in \mathcal{B}_i\otimes \widehat{U(\mathcal{C}\otimes_{\mathcal{A}} \pi^n\mathcal{L})}$ has the property that $\pi^sy$ gets sent to $0$. In particular, $\pi^{r+s}y=0$, as $\pi^rT=0$. So if $y$ is a preimage of $x$, then $y$ is $\pi$-torsion. Since $\mathcal{B}_i\otimes \widehat{U(\mathcal{C}\otimes_{\mathcal{A}} \pi^n\mathcal{L})}$ has bounded $\pi$-torsion by the previous lemma, this proves the result.
\end{proof}

\begin{lemma}
\label{sesofalg}
There is a short exact sequence
\begin{equation*}
\begin{xy}
\xymatrix{
0 \ar[r] & A\langle t\rangle \widehat{\otimes}_A \mathscr{U}_n(V)\ar[r]^{u_i\cdot} & A\langle t\rangle\widehat{\otimes}_A \mathscr{U}_n(V)\ar[r] & B_i\widehat{\otimes}_A \mathscr{U}_n(V)\ar[r] & 0
}
\end{xy}
\end{equation*}
of left $A\langle t\rangle$-modules.
\end{lemma}
\begin{proof}
This is an easy variant of \cite[Lemma 4.11]{Bodecompl} and \cite[Proposition 4.3.c)]{Ardakov1}.\\
First note that the short exact sequence
\begin{equation*}
0\to A\langle t\rangle \to A\langle t\rangle \to B_i\to 0
\end{equation*}
consists of strict morphisms by \cite[Proposition 3.1/20]{Bosch} and \cite[Lemma 2.6]{Bodecompl}.\\
\\
Since $B_i$ is flat over $A$ by \cite[Corollary 4.1/5]{Bosch}, we have 
\begin{equation*}
\Tor^A_1(B_i, \mathscr{U}_n(V))=0,
\end{equation*}
so tensoring with $\mathscr{U}_n(V)$ yields a short exact sequence
\begin{equation*}
0\to A\langle t\rangle \otimes_A \mathscr{U}_n(V)\to A\langle t\rangle \otimes_A \mathscr{U}_n(V)\to B_i\otimes_A \mathscr{U}_n(V)\to 0.
\end{equation*}
Finally, $\overline{\mathcal{B}_i}\otimes_{\mathcal{A}} \widehat{U(\mathcal{C}\otimes_{\mathcal{A}} \pi^n\mathcal{L})}$ has bounded $\pi$-torsion by Corollary \ref{bddpitorcorollary}, so that the short exact sequence above consists of strict morphisms and stays exact after completion by \cite[Lemma 2.13]{Bodecompl}.
\end{proof}

\begin{lemma}
\label{sesUnVmod}
Let $N$ be a finitely generated left Banach $\mathscr{U}_n(V)$-module. Then we have a short exact sequence
\begin{equation*}
\begin{xy}
\xymatrix{
0 \ar[r]& A\langle t\rangle \widehat{\otimes}_A N\ar[r]^{u_i\cdot} & A\langle t\rangle \widehat{\otimes}_A N\ar[r]& B_i\widehat{\otimes}_A N\ar[r]& 0
}
\end{xy}
\end{equation*}
of left $A\langle t\rangle$-modules, analogously for right modules.
\end{lemma}
\begin{proof}
Since $N$ is finitely generated over the Noetherian algebra $\mathscr{U}_n(V)$, we have a short exact sequence
\begin{equation*}
0\to N'\to \mathscr{U}_n(V)^{\oplus r} \to N\to 0,
\end{equation*}
where $N'$ is another finitely generated Banach module over $\mathscr{U}_n(V)$. By \cite[Lemma 2.6]{Bodecompl}, this consists of strict morphisms.\\
\\
Since $\mathcal{A}\langle t\rangle $ is flat over $\mathcal{A}$ by \cite[Remark 7.3/2]{Bosch}, we know by \cite[Lemma 2.13]{Bodecompl} that
\begin{equation*}
0\to A\langle t\rangle \widehat{\otimes}_A N'\to A\langle t\rangle \widehat{\otimes}_A \mathscr{U}_n(V)^{\oplus r} \to A\langle t\rangle \widehat{\otimes}_A N\to 0
\end{equation*}
is exact.\\
\\
Moreover, $B_i\widehat{\otimes}_A N'\cong \mathscr{U}_n(f^{-1}Y_i\cap V)\otimes_{\mathscr{U}_n(V)} N'$ as left $B_i$-modules, where we could omit the completion symbol on the right-hand side by Lemma \ref{removehat}. Likewise for the other terms.\\
Now $f^{-1}Y_i\cap V$ is a rational subdomain of $V$ by \cite[Proposition 3.3/13]{Bosch}, and is actually $\mathcal{C}\otimes \pi^n\mathcal{L}$-accessible - it is $V(x)$ if $i=1$, $V(x^{-1})$ if $i=2$, again by \cite[Proposition 3.3/13]{Bosch}. So by \cite[Theorem 4.10]{Bodecompl}, we know that
\begin{equation*}
0\to B_i\widehat{\otimes}_A N'\to B_i\widehat{\otimes}_A \mathscr{U}_n(V)^{\oplus r} \to B_i\widehat{\otimes}_A N\to 0
\end{equation*} 
is exact.\\
\\
We thus obtain the following commutative diagram of left $A\langle t\rangle$-modules
\begin{equation*}
\begin{xy}
\xymatrix{
0\ar[r] & A\langle t\rangle \widehat{\otimes}_A N'\ar[r] \ar[d]^{f_1} & A\langle t\rangle \widehat{\otimes}_A \mathscr{U}_n(V)^{\oplus r} \ar[r] \ar[d]^{g_1} & A\langle t\rangle \widehat{\otimes}_A N\ar[r] \ar[d]^{h_1}& 0\\
0\ar[r] & A\langle t\rangle \widehat{\otimes}_A N'\ar[r] \ar[d]^{f_2} & A\langle t\rangle \widehat{\otimes}_A \mathscr{U}_n(V)^{\oplus r} \ar[r] \ar[d]^{g_2} & A\langle t\rangle \widehat{\otimes}_A N\ar[r] \ar[d]^{h_2}& 0\\
0\ar[r] & B_i \widehat{\otimes}_A N'\ar[r] & B_i \widehat{\otimes}_A \mathscr{U}_n(V)^{\oplus r} \ar[r] & B_i \widehat{\otimes}_A N\ar[r]& 0
}
\end{xy}
\end{equation*}
where each row is exact.\\
We know from \cite[Proposition 2.1.8/6]{BGR} that $f_2$, $g_2$ and $h_2$ are surjections, so we have a long exact sequence
\begin{equation*}
0\to \ker f_1\to \ker g_1\to \ker h_1\to \ker f_2/\im f_1 \to \ker g_2/\im g_1\to \ker h_2/\im h_1\to 0.
\end{equation*}
By Lemma \ref{sesofalg}, this becomes
\begin{equation*}
0 \to \ker f_1\to 0 \to \ker h_1\to \ker f_2/\im f_1\to 0 \to\ker h_2/\im h_1\to 0,
\end{equation*}
so we immediately get that $\ker h_2=\im h_1$. But this argument holds for any finitely generated $\mathscr{U}_n(V)$-module, so in particular for $N'$. Thus $\ker f_2=\im f_1$, and by exactness $\ker h_1=0$.\\
\\
Thus
\begin{equation*}
0\to A\langle t\rangle \widehat{\otimes}_A N\to A\langle t\rangle \widehat{\otimes}_A N\to B_i\widehat{\otimes}_A N\to 0
\end{equation*}
is a short exact sequence.
\end{proof}

\begin{theorem}
\label{stepalocalize}
The natural morphism
\begin{equation*}
\mathscr{U}_n(f^{-1}Y_i)\otimes_{\mathscr{U}_n(X)} \check{\mathrm{H}}^j(\mathfrak{V}, \mathcal{M}_n)\to \mathrm{H}^j(\mathscr{U}_n(f^{-1}Y_i)\widehat{\otimes}_{\mathscr{U}_n(X)} \check{C}^\bullet(\mathfrak{V}, \mathcal{M}_n))
\end{equation*}
is an isomorphism of $\mathscr{U}_n(f^{-1}Y_i)$-modules for each $j\geq 0$.
\end{theorem}
\begin{proof}
We abbreviate $\check{\mathrm{H}}^j(\mathfrak{V}, \mathcal{M}_n)$ to $H^j$ and $\check{C}^\bullet(\mathfrak{V}, \mathcal{M}_n)$ to $C^\bullet$.\\
\\
Since 
\begin{equation*}
\mathscr{U}_n(f^{-1}Y_i)\cong B_i\widehat{\otimes}_A \mathscr{U}_n(X)
\end{equation*}
as $B_i$-modules, it is enough to show that the natural morphism
\begin{equation*}
B_i\widehat{\otimes}_A H^j\to \mathrm{H}^j(B_i\widehat{\otimes}_A C^\bullet)
\end{equation*}
is an isomorphism of left $B_i$-modules, or equivalently of $A\langle t\rangle$-modules.\\ 
\\
Since $\mathscr{U}_n(f^{-1}Y_i)=\mathscr{U}_n(f_*\mathscr{L})(Y_i)$ is flat over $\mathscr{U}_n(X)=\mathscr{U}_n(f_*\mathscr{L})(Y)$ on the right by \cite[Theorem 4.10]{Bodecompl}, we know that
\begin{equation*}
\Tor^{\mathscr{U}_n(X)}_1\left(\mathscr{U}_n(f^{-1}Y_i), H^j\right)=0,
\end{equation*}
so that the short exact sequence of $(A\langle t\rangle, \mathscr{U}_n(X))$-bimodules from \cite[Lemma 4.11]{Bodecompl}, 
\begin{equation*}
\begin{xy}
\xymatrix{
0 \ar[r]& \mathscr{U}_n(X)\langle t\rangle_i \ar[r]^{u_i\cdot}& \mathscr{U}_n(X)\langle t\rangle_i \ar[r]& \mathscr{U}_n(f^{-1}Y_i)\ar[r]& 0}
\end{xy}
\end{equation*}
remains exact after applying the functor $-\otimes_{\mathscr{U}_n(X)} H^j$, producing a short exact sequence of left $A\langle t\rangle$-modules, which can be written as 
\begin{equation*}
0\to A\langle t\rangle \widehat{\otimes}_A H^j \to A\langle t \rangle \widehat{\otimes}_A H^j\to B_i\widehat{\otimes}_A H^j\to 0.
\end{equation*}
By Lemma \ref{sesUnVmod}, we also have a short exact sequence
\begin{equation*}
0\to A\langle t\rangle \widehat{\otimes}_A C^\bullet\to A\langle t\rangle \widehat{\otimes}_A C^\bullet\to B_i\widehat{\otimes}_A C^\bullet \to 0,
\end{equation*}
of left $A\langle t \rangle$-modules.\\
\\
This now induces a long exact sequence
\begin{equation*}
\dots\to \mathrm{H}^j(A\langle t\rangle\widehat{\otimes} C^\bullet)\to \mathrm{H}^j(A\langle t\rangle\widehat{\otimes} C^\bullet) \to \mathrm{H}^j(B_i\widehat{\otimes}C^\bullet)\to \dots
\end{equation*}
fitting into a commutative diagram
\begin{equation*}
\begin{xy}
\xymatrix{
A\langle t \rangle\widehat{\otimes} H^j\ar[r]^{u_i\cdot} \ar[d]^{\cong} & A\langle t\rangle\widehat{\otimes} H^j\ar[r] \ar[d]^{\cong} & B_i\widehat{\otimes} H^j\ar[r]^{\xi} \ar[d]^{\theta} & A\langle t\rangle \widehat{\otimes} H^{j+1}\ar[r]^{u_i\cdot} \ar[d]^{\cong} & A\langle t\rangle \widehat{\otimes} H^{j+1}\ar[d]^{\cong}\\
\mathrm{H}^j(A\langle t\rangle\widehat{\otimes} C^\bullet) \ar[r] & \mathrm{H}^j(A\langle t\rangle\widehat{\otimes} C^\bullet) \ar[r] & \mathrm{H}^j(B_i\widehat{\otimes} C^\bullet)\ar[r]& \mathrm{H}^{j+1}(A\langle t\rangle \widehat{\otimes} C^\bullet)\ar[r] & \mathrm{H}^{j+1}(A\langle t\rangle\widehat{\otimes} C^\bullet)
}
\end{xy}
\end{equation*}
where the vertical maps are isomorphisms as indicated by Lemma \ref{tversionloc}.\\
\\
The top row is exact (with $\xi$ being the zero map) by the exactness of the short exact sequences above, and the bottom row is exact by construction, so $\theta$ is an isomorphism by the 5-lemma, as required.
\end{proof}

\subsection*{Step B}
We now generalize the argument to arbitrary $\pi^n\mathcal{L}$-accessible rational subdomains.
\begin{proposition}
\label{stepblocalize}
Let $U\subseteq Y$ be a $\pi^n\mathcal{L}$-accessible rational subdomain of $Y$. Then the natural morphism
\begin{equation*}
\mathscr{U}_n(f^{-1}U)\otimes_{\mathscr{U}_n(X)} \check{\mathrm{H}}^j(\mathfrak{V}, \mathcal{M}_n)\to \mathrm{H}^j(\mathscr{U}_n(f^{-1}U)\widehat{\otimes}_{\mathscr{U}_n(X)} \check{C}^\bullet(\mathfrak{V}, \mathcal{M}_n))
\end{equation*}
is an isomorphism for each $j\geq 0$, and thus
\begin{equation*}
\wideparen{\mathscr{U}}(f^{-1}U)\wideparen{\otimes}_{\wideparen{\mathscr{U}}(X)} \mathrm{H}^j(X, \mathcal{M})\cong \mathrm{H}^j(f^{-1}U, \mathcal{M}).
\end{equation*}
\end{proposition}
\begin{proof}
Let $U$ be $\pi^n\mathcal{L}$-accessible in $r$ steps. Theorem \ref{stepalocalize} proves the case of $r=1$. We proceed inductively. Let $V\subseteq Y$ be a $\pi^n\mathcal{L}$-accessible rational subdomain in $r-1$ steps, containing $U$ and satisfying the properties in Definition \ref{rationalac}, so that $U=V(x)$ or $V(x^{-1})$ for a suitable $x\in \mathcal{O}_Y(V)$.\\
By induction hypothesis, we have
\begin{align*}
\mathscr{U}_n(f^{-1}V)\otimes_{\mathscr{U}_n(X)} \check{\mathrm{H}}^j(\mathfrak{V}, \mathcal{M}_n)&\cong \mathrm{H}^j(\mathscr{U}_n(f^{-1}V)\widehat{\otimes}_{\mathscr{U}_n(X)} \check{C}^\bullet(\mathfrak{V}, \mathcal{M}_n))\\
& \cong \mathrm{H}^j(f^{-1}V\cap \mathfrak{V}, \mathcal{M}_n).
\end{align*}
Moreover, the restriction 
\begin{equation*}
f|_{f^{-1}V}: f^{-1}V\to V
\end{equation*}
is an elementary proper morphism with trivial Stein factorization, and $\mathscr{L}|_{f^{-1}V}$ is a free Lie algebroid on $f^{-1}V$, i.e. all our assumption remain valid under restriction. But now $U$ is a rational subdomain of $V$ which is $\mathcal{C}\otimes_{\mathcal{A}} \pi^n\mathcal{L}$-accessible in one step, where $\mathcal{C}$ is a suitable affine formal model in $\mathcal{O}_Y(V)$. Thus Theorem \ref{stepalocalize} implies that
\begin{equation*}
\mathscr{U}_n(f^{-1}U)\otimes_{\mathscr{U}_n(f^{-1}V)} \check{\mathrm{H}}^j(f^{-1}V\cap \mathfrak{V}, \mathcal{M}_n)\cong \mathrm{H}^j(\mathscr{U}_n(f^{-1}U)\widehat{\otimes}_{\mathscr{U}_n(f^{-1}V)} \check{C}^\bullet(f^{-1}V\cap \mathfrak{V}, \mathcal{M}_n)).
\end{equation*}
Writing $H^j$ for $\check{\mathrm{H}}^j(\mathfrak{V}, \mathcal{M}_n)$, we therefore obtain
\begin{align*}
\mathscr{U}_n(f^{-1}U)\otimes_{\mathscr{U}_n(X)} H^j& \cong \mathscr{U}_n(f^{-1}U)\otimes_{\mathscr{U}_n(f^{-1}V)} \mathscr{U}_n(f^{-1}V) \otimes_{\mathscr{U}_n(X)} H^j\\
& \cong \mathscr{U}_n(f^{-1}U) \otimes_{\mathscr{U}_n(f^{-1}V)} \check{\mathrm{H}}^j(f^{-1}\cap \mathfrak{V}, \mathcal{M}_n),
\end{align*}
and thus
\begin{align*}
\mathscr{U}_n(f^{-1}U)\otimes_{\mathscr{U}_n(X)} H^j& \cong \mathrm{H}^j(\mathscr{U}_n(f^{-1}U)\widehat{\otimes}_{\mathscr{U}_n(f^{-1}V)} \check{C}^\bullet(f^{-1}V\cap \mathfrak{V}, \mathcal{M}_n))\\
 & \cong \mathrm{H}^j(\mathscr{U}_n(f^{-1}U)\widehat{\otimes}_{\mathscr{U}_n(f^{-1}V)} \mathscr{U}_n(f^{-1}V) \widehat{\otimes}_{\mathscr{U}_n(X)} \check{C}^\bullet(\mathfrak{V}, \mathcal{M}_n))\\
& \cong \mathrm{H}^j(\mathscr{U}_n(f^{-1}U)\widehat{\otimes}_{\mathscr{U}_n(X)} \check{C}^\bullet(\mathfrak{V}, \mathcal{M}_n)),
\end{align*}
as required.
\end{proof}

\subsection*{Step C}
\begin{theorem}
\label{stepclocalize}
Let $U\subseteq Y$ be an affinoid subdomain. Then the natural morphism
\begin{equation*}
\wideparen{\mathscr{U}}(f^{-1}U)\wideparen{\otimes}_{\wideparen{\mathscr{U}}(X)} \mathrm{H}^j(X, \mathcal{M})\to \mathrm{H}^j(f^{-1}U, \mathcal{M})
\end{equation*}
is an isomorphism for each $j\geq 0$.
\end{theorem}
\begin{proof}
We know from \cite[Proposition 7.6]{Ardakov1} that $U$ is $\pi^n\mathcal{L}$-accessible for sufficiently large $n$, so there exists a finite covering of $U$ by $\pi^n\mathcal{L}$-accessible rational subdomains $(W_i)$ of $Y$ for sufficiently large $n$. \\
By Corollary \ref{cohomcoad} and Proposition \ref{projlimiso}, $\mathrm{H}^j(X, \mathcal{M})$ is a coadmissible $\wideparen{\mathscr{U}}(X)$-module, so that
\begin{equation*}
\wideparen{\mathscr{U}}(f^{-1}U)\wideparen{\otimes}_{\wideparen{\mathscr{U}}(X)} \mathrm{H}^j(X, \mathcal{M})
\end{equation*}
is a coadmissible $\wideparen{\mathscr{U}}(f^{-1}U)=\wideparen{\mathscr{U}(f_*\mathscr{L})}(U)$-module.\\
\\
We have a natural morphism
\begin{equation*}
\Loc \left( \wideparen{\mathscr{U}}(f^{-1}U)\wideparen{\otimes}_{\wideparen{\mathscr{U}}(X)}\mathrm{H}^j(X, \mathcal{M})\right)\to \left(\mathrm{R}^jf_*\mathcal{M}\right)|_{U}
\end{equation*}
of sheaves of $\wideparen{\mathscr{U}(f_*\mathscr{L})}|_U$-modules, and by Proposition \ref{stepblocalize}, this becomes an isomorphism after taking sections over any $W_i$ or any finite intersection of $W_i$s.\\
Considering the corresponding \v{C}ech complex therefore forces the map between the global sections also to be an isomorphism, i.e.
\begin{equation*}
\wideparen{\mathscr{U}}(f^{-1}U)\wideparen{\otimes}_{\wideparen{\mathscr{U}}(X)} \mathrm{H}^j(X, \mathcal{M})\cong \mathrm{H}^j(f^{-1}U, \mathcal{M}).
\end{equation*} 
\end{proof}

This concludes the proof of Theorem \ref{KiehlDXfree}.

\section{Generalizations, examples, applications}
\subsection{Generalizations}
We can now state various generalizations of Theorem \ref{KiehlDXfree}, considering glueing and passing to coadmissible enlargements.
\begin{lemma}
\label{KiehlDXfreeaff}
Let $f:X\to Y$ be a proper morphism of rigid analytic $K$-varieties, and assume $Y$ is affinoid. Let $\mathscr{L}$ be a Lie algebroid on $X$ which is a free $\mathcal{O}_X$-module.\\
Then $\wideparen{\mathscr{U}(\mathscr{L})}$ is a global Fr\'echet--Stein sheaf on $X$, and $f_*\wideparen{\mathscr{U}(\mathscr{L})}$ is a global Fr\'echet--Stein sheaf on $Y$.\\
If $\mathcal{M}$ is a coadmissible $\wideparen{\mathscr{U}(\mathscr{L})}$-module, then $\mathrm{R}^jf_*\mathcal{M}$ is a coadmissible $f_*\wideparen{\mathscr{U}(\mathscr{L})}$-module for each $j\geq 0$.
\end{lemma}
\begin{proof}
Write $f=hg$ for the Stein factorization. Then the same argument as in Lemma \ref{pushforwalgebroid} shows that $g_*\mathscr{L}$ is a Lie algebroid, and by Lemma \ref{pushffsalg}, the morphism $\wideparen{\mathscr{U}(g_*\mathscr{L})}\to g_*\wideparen{\mathscr{U}(\mathscr{L})}$ becomes an isomorphism under restriction to each $h^{-1}Y_i$. Thus $g_*\wideparen{\mathscr{U}(\mathscr{L})}\cong \wideparen{\mathscr{U}(g_*\mathscr{L})}$ is a global Fr\'echet--Stein sheaf by Proposition \ref{recallglobal}, and hence so is $f_*\wideparen{\mathscr{U}(\mathscr{L})}$ by Lemma \ref{pushfaffinoidglobal}.\\
\\
By definition of properness, there exists an affinoid covering $(Y_i)$ of $Y$ such that for $X_i=f^{-1}Y_i$, the morphism $f|_{X_i}: X_i\to Y_i$ is elementary proper. Since $Y$ is affinoid, we can assume that the covering is finite, and the definition of properness now produces a finite covering $(U_{ij})$ of $X$ by affinoid subspaces (which is admissible by G-topology axioms). Choosing a Lie lattice in $g_*\mathscr{L}(Y)=\mathscr{L}(X)$, we can repeat the construction of $\mathscr{U}_n$ as in subsection 3.4 to verify that $\wideparen{\mathscr{U}(\mathscr{L})}$ is a global Fr\'echet--Stein sheaf on $X$.\\
\\
Now by Theorem \ref{KiehlDXfree} and Lemma \ref{pushfaffinoidglobal}, if $\mathcal{M}$ is a coadmissible $\wideparen{\mathscr{U}(\mathscr{L})}$-module, then $\mathrm{R}^jf_*\mathcal{M}|_{Y_i}$ is a coadmissible $f_*\wideparen{\mathscr{U}(\mathscr{L})}|_{Y_i}$-module for each $i$ and every $j\geq 0$, and thus $\mathrm{R}^jf_*\mathcal{M}$ is coadmissible.
\end{proof}
\begin{lemma}
\label{KiehlDXaffinoid}
Let $f: X\to Y$ be a proper morphism of rigid analytic $K$-varieties. Let $\mathscr{L}$ be a Lie algebroid on $X$, and suppose there exists an affinoid covering $(Y_i)$ of $Y$ such that the restriction of $\mathscr{L}$ to $X_i=f^{-1}Y_i$ is free for each $i$.\\
Then $f_*\wideparen{\mathscr{U}(\mathscr{L})}$ is a full Fr\'echet--Stein sheaf on $Y$.\\
If $\mathcal{M}$ is a coadmissible $\wideparen{\mathscr{U}(\mathscr{L})}$-module, then $\mathrm{R}^jf_*\mathcal{M}$ is a coadmissible $f_*\wideparen{\mathscr{U}(\mathscr{L})}$-module for each $j\geq 0$.
\end{lemma}
\begin{proof}
Without loss of generality, we can assume $Y$ to be affinoid, and it remains to show that $f_*\wideparen{\mathscr{U}(\mathscr{L})}$ is a global Fr\'echet--Stein sheaf in that case. Write $g: X\to Z$ for the first map in the Stein factorization of $f=hg$. As before, $g_*\mathscr{L}$ is a locally free $\mathcal{O}_Z$-module, and the induced anchor map from $\mathscr{L}$ makes it a Lie algebroid on $Z$. The morphism $\wideparen{\mathscr{U}(g_*\mathscr{L})}\to g_*\wideparen{\mathscr{U}(\mathscr{L})}$ becomes an isomorphism under restriction to each $h^{-1}Y_i$ by the above, so again $g_*\wideparen{\mathscr{U}(\mathscr{L})}\cong \wideparen{\mathscr{U}(g_*\mathscr{L})}$ is a global Fr\'echet--Stein sheaf on $Z$. Thus $f_*\wideparen{\mathscr{U}(\mathscr{L})}$ is a global Fr\'echet--Stein sheaf on $Y$ by Proposition \ref{recallglobal}.\\
For the coadmissibility result, it is again sufficient to restrict to each $Y_i$, reducing the claim to Lemma \ref{KiehlDXfreeaff}.
\end{proof}

This proves Theorem \ref{intromainthm}.(i) from the introduction.\\
\\
We can now state a more general Proper Mapping Theorem. For this, let $f: X\to Y$ be a proper morphism of rigid analytic $K$-varieties. Let $\mathscr{L}$ be a Lie algebroid on $X$, and suppose there exists an affinoid covering $(Y_i)$ of $Y$ such that the restriction of $\mathscr{L}$ to $X_i=f^{-1}Y_i$ is free for each $i$. Note in particular that $\wideparen{\mathscr{U}(\mathscr{L})}|_{X_i}$ is a global Fr\'echet--Stein sheaf for each $i$ by Lemma \ref{KiehlDXfreeaff}, and $f_*\wideparen{\mathscr{U}(\mathscr{L})}$ is a full Fr\'echet--Stein sheaf on $Y$ by Lemma \ref{KiehlDXaffinoid}.
\begin{proposition}
\label{KiehlDXgeneralcoad}
Let $\mathscr{U}$ be a sheaf of $K$-algebras on $X$ with a morphism of sheaves of algebras $\theta: \wideparen{\mathscr{U}(\mathscr{L})}\to \mathscr{U}$ such that $\mathscr{U}|_{X_i}$ is a coadmissible enlargement of the global Fr\'echet--Stein sheaf $\wideparen{\mathscr{U}(\mathscr{L})}|_{X_i}$ for each $i$ (in particular, this makes $\mathscr{U}$ a Fr\'echet--Stein sheaf on $X$).\\
Then $f_*\mathscr{U}|_{Y_i}$ is a coadmissible enlargement of $f_*\wideparen{\mathscr{U}(\mathscr{L})}|_{Y_i}$, and in particular a global Fr\'echet--Stein sheaf on $Y_i$ for each $i$. Thus $f_*\mathscr{U}$ is a Fr\'echet--Stein sheaf on $Y$.\\
If $\mathcal{M}$ is a coadmissible $\mathscr{U}$-module then $\mathrm{R}^jf_*\mathcal{M}$ is a coadmissible $f_*\mathscr{U}$-module for every $j\geq 0$.  
\end{proposition}

\begin{proof}
Note that $f_*\wideparen{\mathscr{U}(\mathscr{L})}|_{Y_i}$ is a global Fr\'echet--Stein sheaf by Lemma \ref{KiehlDXaffinoid}, and $f_*\mathscr{U}|_{Y_i}$ is a coadmissible $f_*\wideparen{\mathscr{U}(\mathscr{L})}|_{Y_i}$-module. As $f_*\mathscr{U}|_{Y_i}$ can easily be checked to have continuous multiplication by passing to affinoid coverings of $X_i$, this shows that $f_*\mathscr{U}|_{Y_i}$ is a coadmissible enlargement of $f_*\wideparen{\mathscr{U}(\mathscr{L})}|_{Y_i}$ by Proposition \ref{FSbasechange}. Thus $f_*\mathscr{U}|_{Y_i}$ is a global Fr\'echet--Stein sheaf.\\
If $\mathcal{M}$ is a coadmissible $\mathscr{U}$-module, it is in particular a coadmissible $\wideparen{\mathscr{U}(\mathscr{L})}$-module, by Proposition \ref{FSbasechange} applied to each restriction $\mathcal{M}|_{X_i}$. Applying Lemma \ref{KiehlDXaffinoid} shows that $\mathrm{R}^jf_*\mathcal{M}$ is a coadmissible $f_*\wideparen{\mathscr{U}(\mathscr{L})}$-module. So $\mathrm{R}^jf_*\mathcal{M}|_{Y_i}$ is coadmissible over the global Fr\'echet--Stein sheaf $f_*\wideparen{\mathscr{U}(\mathscr{L})}|_{Y_i}$, and thus coadmissible over $f_*\mathscr{U}|_{Y_i}$ by Proposition \ref{FSbasechange}. Therefore $\mathrm{R}^jf_*\mathcal{M}$ is a coadmissible $f_*\mathscr{U}$-module.
\end{proof}

\begin{lemma}
\label{KiehlDXtwoproper}
Let 
\begin{equation*}
\begin{xy}
\xymatrix{
X_1\ar[r]^{f_1}& X_2\ar[r]^{f_2}& Y
}
\end{xy}
\end{equation*}
be two proper morphisms, and assume that $Y$ is affinoid. Let $\mathscr{L}$ be a Lie algebroid on $X_1$ which is free as an $\mathcal{O}_{X_1}$-module. Note that by applying Lemma \ref{KiehlDXfreeaff} to $f_2f_1$, we know that $\wideparen{\mathscr{U}(\mathscr{L})}$ is a global Fr\'echet--Stein sheaf on $X_1$. Let $\mathscr{U}$ be a coadmissible enlargement of $\wideparen{\mathscr{U}(\mathscr{L})}$.\\
Then $f_{1*}\wideparen{\mathscr{U}(\mathscr{L})}$ is a global Fr\'echet--Stein sheaf on $X_2$, and $f_{1*}\mathscr{U}$ is a coadmissible enlargement of $f_{1*}\wideparen{\mathscr{U}(\mathscr{L})}$.\\
If $\mathcal{M}$ is a coadmissible $\mathscr{U}$-module, then $\mathrm{R}^jf_{1*}\mathcal{M}$ is a coadmissible $f_{1*}\mathscr{U}$-module for each $j\geq 0$.
\end{lemma}
\begin{proof}
Let $g: X_1\to Z$ be the first map in the Stein factorization of $f_1$. Then $g_*\mathscr{L}$ is a Lie algebroid on $Z$ which is free as an $\mathcal{O}_Z$-module, so by Lemma \ref{KiehlDXfreeaff} applied to the proper morphism $f_2h$, $\wideparen{\mathscr{U}(g_*\mathscr{L})}$ is a global Fr\'echet--Stein sheaf on $Z$. As before $g_*\wideparen{\mathscr{U}(\mathscr{L})}\cong \wideparen{\mathscr{U}(g_*\mathscr{L})}$, so Proposition \ref{recallglobal} proves that $f_{1*}\wideparen{\mathscr{U}(\mathscr{L})}\cong h_*\wideparen{\mathscr{U}(g_*\mathscr{L})}$ is a global Fr\'echet--Stein sheaf on $X_2$.\\
By Lemma \ref{KiehlDXaffinoid}, $f_{1*}\mathscr{U}$ is a coadmissible $f_{1*}\wideparen{\mathscr{U}(\mathscr{L})}$-module. As before, this makes $f_{1*}\mathscr{U}$ a coadmissible enlargement of $f_{1*}\wideparen{\mathscr{U}(\mathscr{L})}$.\\
The coadmissibility argument is now the same as in Proposition \ref{KiehlDXgeneralcoad}.
\end{proof}

As a corollary to Proposition \ref{KiehlDXgeneralcoad}, we can consider Lie algebroids $\mathscr{L}$ which are not themselves free, but admit an epimorphism $\mathscr{L}'\to \mathscr{L}$ for some free Lie algebroid $\mathscr{L}'$. The reason why we spell this out explicitly is given by the geometric interpretation later.

\begin{corollary}
\label{KiehlDXalgversion}
Let $f:X\to Y$ be a proper morphism of rigid analytic $K$-varieties, and let $\mathscr{L}$ be a Lie algebroid on $X$ such that there is an epimorphism $\mathscr{L}'\to \mathscr{L}$ for some free Lie algebroid $\mathscr{L}'$.\\
Then $f_*\wideparen{\mathscr{U}(\mathscr{L})}$ is a Fr\'echet--Stein sheaf, and if $\mathcal{M}$ is a coadmissible $\wideparen{\mathscr{U}(\mathscr{L})}$-module, then $\mathrm{R}^jf_*\mathcal{M}$ is a coadmissible $f_*\wideparen{\mathscr{U}(\mathscr{L})}$-module for each $j\geq 0$.
\end{corollary} 

This proves Theorem \ref{intromainthm}.(ii) from the introduction.\\
\\
We now consider one particular instance of this corollary.\\
Suppose that $f: X\to Y$ is elementary proper, and write $f=hg$ for the Stein factorization as usual. Let $(\rho, \mathscr{L})$ be a Lie algebroid on $X$ with the property that $g_*\mathscr{L}$ is free, i.e. $L:=\mathscr{L}(X)$ is a free $\mathcal{O}_X(X)$-module.
\begin{lemma}
The $\mathcal{O}_X$-module $g^*g_*\mathscr{L}$ given by $U\mapsto \mathcal{O}_X(U)\otimes_{\mathcal{O}_X(X)} \mathscr{L}(X)$ is a Lie algebroid on $X$ and is free as an $\mathcal{O}_X$-module.
\end{lemma}
\begin{proof}
As $L$ is a free $\mathcal{O}_X(X)$-module, finitely generated by Kiehl's Proper Mapping Theorem, $g^*g_*\mathscr{L}$ is a free coherent $\mathcal{O}_X$-module. Write $\iota: g^*g_*\mathscr{L}\to \mathscr{L}$ for the natural morphism.\\
The Lie bracket on $\mathscr{L}$ allows us to define a Lie bracket on $g^*g_*\mathscr{L}$, given by
\begin{equation*}
[a\otimes x, b\otimes y]:=ab\otimes [x, y]_L+\rho\iota(a\otimes x)(b)\otimes y-\rho\iota(b\otimes y)(a)\otimes x
\end{equation*}
for $a, b\in \mathcal{O}_X(U)$, $x, y\in L$, $U$ an admissible open subspace of $X$. Note that this is well-defined, as $\rho$ satisfies the anchor map property.\\
This turns $g^*g_*\mathscr{L}$ into a sheaf of $K$-Lie algebras such that $\iota$ is a morphism of sheaves of Lie algebras. The composition $\rho \iota$ is then also a morphism of sheaves of Lie algebras, satisfying the axiom of an anchor map by construction.
\end{proof}

Thus we can apply Corollary \ref{KiehlDXalgversion} as soon as the natural morphism $g^*g_*\mathscr{L}\to \mathscr{L}$ is an epimorphism. By definition of $g^*g_*\mathscr{L}$, this is equivalent to requiring $\mathscr{L}$ to be generated by global sections. 

\begin{corollary}
\label{KiehlDXglobalgenversion1}
Let $f:X\to Y$ be an elementary proper morphism of rigid analytic $K$-varieties, and let $\mathscr{L}$ be a Lie algebroid on $X$ such that $\mathscr{L}(X)$ is a free $\mathcal{O}_X(X)$-module and $\mathscr{L}$ is generated by global sections.\\
Then $f_*\wideparen{\mathscr{U}(\mathscr{L})}$ is a global Fr\'echet--Stein sheaf on $Y$. If $\mathcal{M}$ is a coadmissible $\wideparen{\mathscr{U}(\mathscr{L})}$-module, then $\mathrm{R}^jf_*\mathcal{M}$ is a coadmissible $f_*\wideparen{\mathscr{U}(\mathscr{L})}$-module for each $j\geq 0$.
\end{corollary}
\begin{corollary}
\label{KiehlDXproperglobal}
Let $X$ be a proper rigid analytic $K$-variety, and let $\mathscr{L}$ be a Lie algebroid on $X$ which is generated by global sections. Then $\wideparen{\mathscr{U}(\mathscr{L})}(X)$ is a Fr\'echet--Stein algebra, and if $\mathcal{M}$ is a coadmissible $\wideparen{\mathscr{U}(\mathscr{L})}$-module then $\mathrm{H}^j(X, \mathcal{M})$ is a coadmissible $\wideparen{\mathscr{U}(\mathscr{L})}(X)$-module for each $j\geq 0$. 
\end{corollary}
\begin{proof}
By Kiehl's Proper Mapping Theorem, $\mathscr{L}(X)$ is a finite-dimensional $K$-vector space, and $\mathscr{L}':=\mathcal{O}_X\otimes_K (\mathscr{L}(X))$ is a free Lie algebroid on $X$ by the same argument as above. Thus the result follows from Corollary \ref{KiehlDXalgversion}.
\end{proof}
Setting $\mathscr{L}=\mathcal{T}_X$ yields Corollary \ref{intropropervar} from the introduction.\\
More generally, Proposition \ref{KiehlDXgeneralcoad} gives directly the following.
\begin{corollary}
\label{KiehlDXglobalgenversion2}
Let $f: X\to Y$ be a proper morphism with Stein factorization $f=hg$, and let $\mathscr{L}$ be a Lie algebroid on $X$ such that the following holds:
\begin{enumerate}[(i)]
\item $g_*\mathscr{L}$ is locally free.
\item The natural morphism $g^*g_*\mathscr{L}\to \mathscr{L}$ is an epimorphism of sheaves on $X_{\rig}$.
\end{enumerate}
Then $f_*\wideparen{\mathscr{U}(\mathscr{L})}|_U$ is a global Fr\'echet--Stein sheaf for any admissible open affinoid subspace $U$ of $Y$ such that $\mathscr{L}(f^{-1}U)$ is a free $\mathcal{O}_X(f^{-1}U)$-module. In particular, $f_*\wideparen{\mathscr{U}(\mathscr{L})}$ is a Fr\'echet--Stein sheaf on $Y$.\\
If $\mathcal{M}$ is a coadmissible $\wideparen{\mathscr{U}(\mathscr{L})}$-module then $\mathrm{R}^jf_*\mathcal{M}$ is a coadmissible $f_*\wideparen{\mathscr{U}(\mathscr{L})}$-module for each $j\geq 0$.
\end{corollary}
We will discuss a number of examples after giving a more geometric motivation for the conditions we have imposed in our results.
\subsection{Geometric interpretation and examples}
Let $\mathcal{E}$ be a locally free $\mathcal{O}_X$-module of finite rank on a rigid analytic $K$-variety $X$. Note that we can associate to $\mathcal{E}$ a rigid analytic $K$-variety $V(\mathcal{E})$ with a projection morphism $\rho: V(\mathcal{E})\to X$, completely analogous to the construction of vector bundles in algebraic geometry. We sketch the construction below, as it is rarely discussed in this context and the reader might find the relation to our definition of $\wideparen{\mathscr{U}(\mathscr{L})}$ illuminating.\\
\\
Suppose that $X=\Sp A$ is affinoid, and $\mathcal{E}$ is free of rank $m$. Fix an affine formal model $\mathcal{A}$ of $A$ and free generators $e_1, \dots, e_m\in \mathcal{E}(X)$. We can now identify $S=\Sym_A \mathcal{E}(X)$ with the polynomial algebra $A[x_1, \dots, x_m]$. Via this identification, we can equip $S$ with an algebra norm corresponding to the gauge norm on $A[ x]$ with unit ball $\mathcal{A}[ \pi^nx]$, and we denote the corresponding Banach completion by $S_n$, which is naturally isomorphic to $A\langle \pi^nx\rangle$.\\
We denote by $V(\mathcal{E})$ the space obtained by glueing $\Sp S_n$, i.e. $V(\mathcal{E})=\varinjlim \Sp S_n$. By construction, $V(\mathcal{E})$ is isomorphic to $X\times \mathbb{A}_K^{m, \an}$, and the natural morphisms $A\to S_n$ give rise to a projection morphism $\rho: V(\mathcal{E})\to A$, corresponding to the natural projection onto the first factor.\\ 
A standard argument ensures that up to isomorphism, this construction is independent of the choices made. Moreover, we can glue this construction to obtain a vector bundle $V(\mathcal{E})$ for any locally free $\mathcal{O}_X$-module $\mathcal{E}$ of finite rank on a rigid analytic $K$-variety $X$, giving rise to a contravariant functor $V$.\\  
\\
We will now interpret the conditions imposed in our previous results as certain properness conditions on the level of vector bundles.\\
\\
Let $f:X \to Y=\Sp A$ be an elementary proper morphism such that $A=\mathcal{O}_X(X)$, and let $\mathscr{L}$ be a Lie algebroid on $X$ which is a free $\mathcal{O}_X$-module of rank $m$. In this case the rigid analytic vector bundles
\begin{equation*}
V(\mathscr{L})\cong X\times \mathbb{A}^{m, \an}, \ V(f_*\mathscr{L})\cong Y\times \mathbb{A}^{m, \an}
\end{equation*}
are trivial, and there is a natural morphism $V(\mathscr{L})\to V(f_*\mathscr{L})$, which is proper by \cite[Lemma 9.6.2/1]{BGR}.\\
Our Theorem \ref{KiehlDXfree} can thus be viewed as a noncommutative version of Kiehl's Theorem \ref{Kiehlthm} on trivial vector bundles.\\
\\
The next result extends this interpretation to the case of Corollary \ref{KiehlDXglobalgenversion2}.
\begin{proposition}
\label{vbinterpretation}
Let $f:X \to Y$ be a proper morphism of rigid analytic $K$-varieties with Stein factorization
\begin{equation*}
\begin{xy}
\xymatrix{
X\ar[r]^g& Z\ar[r]^h &Y,
}
\end{xy}
\end{equation*} 
and let $\mathscr{L}$ be a Lie algebroid on $X$ such that the following holds:
\begin{enumerate}[(i)]
\item $g_*\mathscr{L}$ is locally free.
\item The natural morphism $g^*g_*\mathscr{L}\to \mathscr{L}$ is an epimorphism of sheaves on $X_{\rig}$.
\end{enumerate}
Then there is a natural morphism
\begin{equation*}
V(\mathscr{L})\to V(g^*g_*\mathscr{L}),
\end{equation*}
which is a closed immersion, and and a proper morphism
\begin{equation*}
V(g^*g_*\mathscr{L})\to V(g_*\mathscr{L})
\end{equation*}
of rigid analytic $K$-varieties.\\
In particular, their composition $V(\mathscr{L})\to V(g_*\mathscr{L})$ is proper.
\end{proposition}

\begin{proof}
The natural map $\mu: \mathscr{L}':=g^*g_*\mathscr{L}\to \mathscr{L}$ induces a morphism of rigid analytic $K$-varieties $V(\mu): V(\mathscr{L})\to V(\mathscr{L}')$ by functoriality.\\
\\
We show that this is a closed immersion. Restricting to an admissible affinoid covering $(U_i)$ of $X$ on which both $\mathscr{L}'$ and $\mathscr{L}$ are free, the morphism $\theta_i: \Sym \mathscr{L}'(U_i)\to \Sym \mathscr{L}(U_i)$ is a surjection for each $i$ by assumption. \\
Choosing a residue norm on $\mathcal{O}_X(U_i)$ with unit ball $\mathcal{B}_i$ and a free generating set $e_1, \dots, e_m$ of $\mathscr{L}'(U_i)$, endow $S=\Sym \mathscr{L}'(U_i)$ with the norm with unit ball the $\mathcal{B}_i$-subalgebra generated by the $e_j$, and endow $\Sym \mathscr{L}(U_i)$ with the corresponding quotient norm via $\theta_i$. In particular, $\theta_i$ is strict with respect to these choices of norm by construction.\\
The completion of $\Sym \mathscr{L}'(U_i)$ is the affinoid algebra $S_0$ constructed above, and by strictness this surjects onto the completion of $\Sym \mathscr{L}(U_i)$, which is again affinoid, as it is topologically of finite type over $K$. \\
Replacing $e_j$ by $\pi^ne_j$ for varying $n$, the affinoid spaces $\Sp S_n$ form an admissible covering of $V(\mathscr{L}'|_{U_i})$ by affinoid subspaces, and the surjections between affinoid algebras exhibit $V(\mu)$ as a closed immersion. In particular, $V(\mu)$ is a proper morphism by \cite[Proposition 9.6.2/5]{BGR}.\\
\\
Choosing an admissible covering $(Z_i)$ of $Z$ such that $g_*\mathscr{L}|_{Z_i}$ is free of rank $m$ on each $i$, $g^*g_*\mathscr{L}|_{g^{-1}Z_i}$ is also free of rank $m$, again inducing a proper morphism
\begin{equation*}
\begin{xy}
\xymatrix{
g^{-1}Z_i \times (\mathbb{A}^m)^{\an} \ar[r] \ar[d]^{\cong} & Z_i\times (\mathbb{A}^m)^{\an}\ar[d]^{\cong}\\
V(\mathscr{L}'|_{g^{-1}Z_i})\ar[r] & V(g_*\mathscr{L}|_{Z_i})
}
\end{xy}
\end{equation*} 
These glue to give a proper morphism $V(\mathscr{L}')\to V(g_*\mathscr{L})$, and the result follows from the fact that the composition of proper morphisms is proper (see \cite[Corollary 3.2]{Luet}).
\end{proof}

Thus our assumptions can be interpreted as requiring a vector bundle $V(g_*\mathscr{L})$ on $Z$ together with a proper morphism $V(\mathscr{L})\to V(g_*\mathscr{L})$. \\
\\
Our next goal will be to spell out a number of naturally occuring cases in which Proposition \ref{KiehlDXgeneralcoad} applies. We reserve the main application, our discussion of analytic partial flag varieties, for the next subsection.

\subsubsection*{Example 1: Closed immersions}
Let $Y=\Sp A$ be an affinoid $K$-variety and let $\iota: X\to Y$ be a closed immersion of affinoid varieties, i.e. if $X=\Sp B$, then the corresponding morphism of affinoid algebras $A\to B$ is a surjection. This map is proper by \cite[Proposition 9.6.2/5]{BGR} with trivial Stein factorization in the sense that $g=\text{id}_X$ is the identity on $X$ and $h=\iota$ in our usual notation.\\
In particular, if $\mathscr{L}$ is a Lie algebroid on $X$, then $g_*\mathscr{L}=\mathscr{L}$, $g^*g_*\mathscr{L}=\mathscr{L}$, and all conditions in Corollary \ref{KiehlDXglobalgenversion2} are trivially satisfied. \\
Since all conditions in Corollary \ref{KiehlDXglobalgenversion2} are local, it follows that the same holds true for arbitrary closed immersions $\iota$.\\
\\
This is of course not really surprising, as $\iota$ is an affinoid morphism, so we could deduce everything in this case simply from Lemma \ref{pushfaffinoidglobal} (as is tacitly done in \cite{Ardakov2}).
\subsubsection*{Example 2: Projections}
Let $X=\mathbb{P}^{n, \an}$ be the analytification of projective $n$-space over $K$, and consider the projection to a point
\begin{equation*}
f: \mathbb{P}^{n, \an}\to \Sp K.
\end{equation*}
This is trivially a projective morphism and hence proper, and the Stein factorization in this case is $g=f$, $h=\text{id}_{\Sp K}$. \\
If $\mathscr{L}$ is a Lie algebroid on $X$, then $g_*\mathscr{L}=\mathscr{L}(X)$ is a (finite-dimensional) $K$-vector space, and
\begin{equation*}
g^*g_*\mathscr{L}=\mathcal{O}_X\otimes_K \mathscr{L}(X).
\end{equation*}
Thus our assumptions are satisfied if and only if $\mathcal{O}_X\otimes \mathscr{L}(X)\to \mathscr{L}$ is an epimorphism, i.e. if and only if $\mathscr{L}$ is generated by global sections. This is for example the case when $\mathscr{L}=\mathcal{T}_X$ is the tangent sheaf of $X$.

\subsubsection*{Example 3: Direct products}
Let $Y$ be a smooth rigid analytic $K$-variety and consider the projection $f: \mathbb{P}^{n, \an}\times Y\to Y$, which is again proper. As in the previous example, the Stein factorization is trivial. Now $X=\mathbb{P}^{n, \an}\times Y$ is smooth, so the tangent sheaf $\mathcal{T}_X$ is a Lie algebroid on $X$. By definition of smoothness, $Y$ admits an admissible covering by affinoid subspaces $(Y_i)$ such that $\mathcal{T}_{Y_i}$ is free, and we write $X_i=\mathbb{P}^{n, \an}\times Y_i$. Write $p_1: X\to \mathbb{P}^{n, \an}$ for the projection onto the first factor. Since
\begin{equation*}
\mathcal{T}_X\cong\mathcal{O}_X\otimes_{p_1^{-1}\mathcal{O}_{\mathbb{P}^{n, \an}}} p_1^{-1}\mathcal{T}_{\mathbb{P}^{n, \an}} \oplus \mathcal{O}_X\otimes_{f^{-1}\mathcal{O}_{Y}} f^{-1}\mathcal{T}_{Y}
\end{equation*}
as in the algebraic case, $\mathcal{T}_X(X_i)$ is a free module over $\mathcal{O}_Y(Y_i)=\mathcal{O}_X(X_i)$, and $\mathcal{T}_{X_i}$ is again generated by global sections. \\
Thus $f_*\wideparen{\mathcal{D}}_X|_{Y_i}$ is a global Fr\'echet--Stein sheaf on $Y_i$, $f_*\wideparen{\mathcal{D}}_X$ is a Fr\'echet--Stein sheaf on $Y$, and $\mathrm{R}^jf_*\mathcal{M}$ is a coadmissible $f_*\wideparen{\mathcal{D}}_X$-module for each $j\geq 0$, where $\mathcal{M}$ is any coadmissible $\wideparen{\mathcal{D}}_X$-module. \\
Moreover, note that $f^*\mathcal{T}_Y=\mathcal{O}_X\otimes_{f^{-1}\mathcal{O}_Y} f^{-1}\mathcal{T}_Y$ is a Lie algebroid on $X$ (via the natural embedding into $\mathcal{T}_X$) with the property that $f^*\mathcal{T}_Y|_{X_i}=\mathcal{O}_{X_i}\otimes_{\mathcal{O}_Y(Y_i)} \mathcal{T}_Y(Y_i)$ is a free $\mathcal{O}$-module. As $\mathcal{O}_X(X_i)=\mathcal{O}_Y(Y_i)$, we have $f_*f^*\mathcal{T}_Y\cong \mathcal{T}_Y$, and thus $f_*\wideparen{\mathscr{U}(f^*\mathcal{T}_Y)}\cong \wideparen{\mathcal{D}}_Y$ by Lemma \ref{pushffsalg}. Hence we can formulate the following theorem as a direct consequence of Lemma \ref{KiehlDXaffinoid}.
\begin{theorem}
\label{derhamcoad}
Let $f: X=\mathbb{P}_K^{n, \an}\times Y\to Y$ be the projection for some smooth rigid analytic $K$-variety $Y$. Let $\mathcal{M}$ be a coadmissible $\wideparen{\mathscr{U}(f^*\mathcal{T}_Y)}$-module on $X$. Then $\mathrm{R}^jf_*\mathcal{M}$ is a coadmissible $\wideparen{\mathcal{D}}_Y$-module for each $j\geq 0$.
\end{theorem}

This allows us to briefly discuss $\wideparen{\mathcal{D}}$-module pushforwards, i.e. a functor which sends $\wideparen{\mathcal{D}}_X$-modules to $\wideparen{\mathcal{D}}_Y$-modules rather than $f_*\wideparen{\mathcal{D}}_X$-modules.\\
Let $f: X\to Y$ be an arbitrary projective morphism of smooth rigid analytic $K$-varieties, which can be factored as 
\begin{equation*}
\begin{xy}
\xymatrix{
X\ar[r]^-\iota & \mathbb{P}^{n, \an}\times Y\ar[r]^-g & Y,}
\end{xy}
\end{equation*}
where $\iota$ is a closed immersion and $g$ is the natural projection. Any coadmissible $\wideparen{\mathcal{D}}_X$-module $\mathcal{M}$ gives rise to a coadmissible $\wideparen{\mathcal{D}}_{\mathbb{P}^{n, \an}\times Y}$-module $\iota_+\mathcal{M}$ by \cite{Ardakov2}, and we expect a well-behaved $\wideparen{\mathcal{D}}$-module pushforward to be obtained by applying $\mathrm{R}g_*$ to a suitable relative de Rham complex $DR_g(\iota_+\mathcal{M})$, in strict analogy to \cite[Proposition 1.5.28]{Hotta}. The above results should then make it straightforward to verify that coadmissibility is preserved under this pushforward functor, once the correct (derived) categorical framework has been formulated rigorously.

\subsection{Application: Analytic partial flag varieties}
Let $\mathbf{G}$ be a split reductive affine algebraic group scheme over $K$, and let $G=\mathbf{G}(K)$ with Lie algebra $\mathfrak{g}$. Let $\mathbf{B}\leq \mathbf{G}$ be a Borel subgroup scheme, $\mathbf{P}\leq \mathbf{G}$ a parabolic subgroup scheme and let $\mathbf{X}=\mathbf{G}/\mathbf{P}$ be the partial flag variety. In this section, we will be concerned with coadmissible $\wideparen{\mathcal{D}}$-modules on the analytification $X=\mathbf{X}^{\an}$.\\
By \cite[II.1.8]{Jantzen}, $\mathbf{G}/\mathbf{P}$ is projective, and thus $X$ is proper over $\Sp K$ by \cite[Satz 2.16]{Kopf}. \\
More generally, if $\mathbf{P}_1 \leq \mathbf{P}_2$ are two parabolics, $X_i=(\mathbf{G}/\mathbf{P}_i)^{\an}$, then the natural projection morphism $X_1\to X_2$ is proper by \cite[Proposition 9.6.2/4]{BGR} and \cite[Satz 2.16]{Kopf}.\\  
\\
Let $\mathbf{R}\leq \mathbf{G}$ be the unipotent radical of $\mathbf{P}$ and $\mathbf{L}$ its Levi factor. Write $\mathfrak{l}$ for the Lie algebra of $L=\mathbf{L}(K)$. Following \cite{Backelin-Krem}, the natural morphism $\xi: \mathbf{G}/\mathbf{R}\to \mathbf{G}/\mathbf{P}$ turns $\mathbf{G}/\mathbf{R}$ into an $\mathbf{L}$-torsor in the sense of \cite[4.1]{Backelin-Krem}, where $\mathbf{L}$ acts on $\mathbf{G}/\mathbf{R}$ by right translations.\\
\\
Define the enhanced tangent sheaf $\widetilde{\mathcal{T}}_{\mathbf{G}/\mathbf{P}}:=(\xi_* \mathcal{T}_{\mathbf{G}/\mathbf{R}})^{\mathbf{L}}$, a Lie algebroid on $\mathbf{X}$ (see \cite[Definition 4.2]{Backelin-Krem}, \cite[4.4]{Ardakovannals}).\\
Applying the analytification functor, we obtain the Lie algebroid $\widetilde{\mathcal{T}}_X$. Since the natural morphism $\mathcal{O}_{\mathbf{X}}\otimes_K \mathfrak{g}\to \widetilde{\mathcal{T}}_{\mathbf{G}/\mathbf{P}}$ is an epimorphism by the same argument as in \cite[Proposition 4.8.(a)]{Ardakovannals}, it follows from \cite[Theorems 6.3/12 and 13]{Bosch} that $\widetilde{\mathcal{T}}_X$ is generated by global sections.\\
\\
We now set 
\begin{equation*}
\wideparen{\widetilde{\mathcal{D}}}_X:=\wideparen{\mathscr{U}\left(\widetilde{\mathcal{T}}_X\right)},
\end{equation*}
a coadmissible enlargement of $\wideparen{\mathscr{U}(\mathcal{O}_X\otimes \mathfrak{g})}$. Applying Corollary \ref{KiehlDXproperglobal}, we obtain the following.
\begin{corollary}
\label{enhancedDXKiehl}
The global sections $\wideparen{\widetilde{\mathcal{D}}}_X(X)$ form a Fr\'echet--Stein algebra, and if $\mathcal{M}$ is a coadmissible $\wideparen{\widetilde{\mathcal{D}}}_X$-module on $X$, then $\mathrm{R}^j\Gamma(X, \mathcal{M})$ is coadmissible over both $\wideparen{\widetilde{\mathcal{D}}}_X(X)$ and $\wideparen{U(\mathfrak{g})}$ for each $j\geq 0$.
\end{corollary}

Write $\mathfrak{h}$ for a Cartan subalgebra of $\mathfrak{g}$,  and let $\lambda \in \mathfrak{h}^*$. The centre of the enveloping algebra $U(\mathfrak{g})$ will be denoted by $Z(\mathfrak{g})$.\\
The triangular decomposition $\mathfrak{g}=\mathfrak{n}^{-}\oplus \mathfrak{h} \oplus \mathfrak{n}$ induces the Harish-Chandra morphism $\theta: Z(\mathfrak{g})\to U(\mathfrak{g})\to U(\mathfrak{h})=\Sym \mathfrak{h}$, which allows us to view $\lambda$ as a character for $Z(\mathfrak{g})$. We let $K_{\lambda}$ denote the corresponding one-dimensional $Z(\mathfrak{g})$-representation, and set
\begin{equation*}
U_{\lambda}:= U(\mathfrak{g})\otimes_{Z(\mathfrak{g})} K_{\lambda}.
\end{equation*}
We denote the kernel of the surjection $U(\mathfrak{g})\to U_{\lambda}$ by $\mathfrak{m}_{\lambda}$.\\
Now choose an $(R, R)$-Lie lattice $\mathfrak{g}_R$ inside $\mathfrak{g}$, and write $U_n=U(\pi^n\mathfrak{g}_R)$. This induces a norm on $U(\mathfrak{g})$, and we let $Z(\mathfrak{g})$ be equipped with the corresponding subspace norm. We define
\begin{equation*}
\wideparen{U}_{\lambda}:= \varprojlim \left(\widehat{U_n}_K \widehat{\otimes}_{Z(\mathfrak{g})} K_{\lambda}\right),
\end{equation*}
where the norm on $K_{\lambda}$ is given by identification with $K$.

\begin{lemma}
The $K$-algebra $\wideparen{U}_{\lambda}$ is naturally isomorphic to the quotient of $\wideparen{U(\mathfrak{g})}$ by the closure of $\mathfrak{m}_{\lambda}$. \\
In particular, $\wideparen{U}_{\lambda}$ is a Fr\'echet--Stein algebra. 
\end{lemma}
\begin{proof}
This is Lemma \ref{fgfrcompletion} applied to the short exact sequence 
\begin{equation*}
0\to \mathfrak{m}_{\lambda}\to U(\mathfrak{g})\to U_{\lambda}\to 0,
\end{equation*}
together with \cite[Proposition 3.7]{Schneider03}.
\end{proof}

As in section 6.1, note that $\pi^n\mathfrak{g}_R$ determines compatible norms on $U(\widetilde{\mathcal{T}}_X)(U)$ for any admissible open affinoid subspace $U\subset X$, giving rise to completions $\widetilde{\mathcal{D}}_n:=\mathscr{U}_n(\widetilde{\mathcal{T}}_X)$ such that $\wideparen{\widetilde{\mathcal{D}}}_X=\varprojlim \widetilde{\mathcal{D}}_n$ is a global Fr\'echet--Stein sheaf on $X$.\\ 
\\
Identifying $Z(\mathfrak{l})$ with $\mathbf{L}$-invariant differential operators on $\mathbf{G}/\mathbf{R}$ (see \cite[4.1]{Backelin-Krem}), we obtain a natural morphism $Z(\mathfrak{l})\to \widetilde{\mathcal{D}}_n$ with central image for each $n$, and we define the sheaf of twisted differential operators on $X$ by 
\begin{equation*}
\wideparen{\mathcal{D}}_X^{\lambda}= \varprojlim (\widetilde{\mathcal{D}}_n\widehat{\otimes}_{Z(\mathfrak{l})} K_{\lambda}).
\end{equation*}
Again, Lemma \ref{fgfrcompletion} shows that the natural morphism $\wideparen{\widetilde{\mathcal{D}}}_X\to \wideparen{\mathcal{D}}_X^{\lambda}$ is an epimorphism which turns $\wideparen{\mathcal{D}}_X^{\lambda}$ into a coadmissible enlargement of $\wideparen{\widetilde{\mathcal{D}}}_X$, and hence a coadmissible enlargement of $\wideparen{\mathscr{U}(\mathcal{O}_X\otimes \mathfrak{g})}$.\\
\\
Now let $\mathbf{P}_1\leq \mathbf{P}_2$ be two parabolic subgroups, and consider the proper morphism $f: X_1\to X_2$, where $X_i=(\mathbf{G}/\mathbf{P}_i)^{\an}$.

\begin{corollary}
The pushforward $f_*\wideparen{\mathscr{U}(\mathcal{O}_{X_1}\otimes \mathfrak{g})}$ is a global Fr\'echet--Stein sheaf on $X_2$, and $f_*\wideparen{\mathcal{D}}_{X_1}^{\lambda}$ is a coadmissible enlargement. In particular, $f_*\wideparen{\mathcal{D}}_{X_1}^{\lambda}$ is a global Fr\'echet--Stein sheaf on $X_2$.\\
If $\mathcal{M}$ is a coadmissible $\wideparen{\mathcal{D}}_{X_1}^{\lambda}$-module, then $\mathrm{R}^jf_*\mathcal{M}$ is coadmissible over $f_*\wideparen{\mathcal{D}}_{X_1}^{\lambda}$, and a fortiori coadmissible over $f_*\wideparen{\mathscr{U}(\mathcal{O}_{X_1}\otimes \mathfrak{g})}$ for each $j\geq 0$.
\end{corollary}
\begin{proof}
This is the content of Lemma \ref{KiehlDXtwoproper}.
\end{proof}

Note that in the extreme case $\mathbf{P}_2=\mathbf{G}$, we obtain the following generalization of the first statement in \cite[Theorem 6.4.7]{Ardakovequiv}.
\begin{corollary}
Let $\mathbf{P}\leq \mathbf{G}$ be a parabolic subgroup, and let $X=(\mathbf{G}/\mathbf{P})^{\an}$. Then the global sections $\wideparen{\mathcal{D}}_X^{\lambda}(X)$ form a Fr\'echet--Stein algebra, and if $\mathcal{M}$ is a coadmissible $\wideparen{\mathcal{D}}_X^{\lambda}$-module, then $\mathrm{R}^j\Gamma(X, \mathcal{M})$ is coadmissible over both $\wideparen{\mathcal{D}}_X^{\lambda}(X)$ and over $\wideparen{U(\mathfrak{g})}$ for each $j\geq 0$.
\end{corollary}
As the $\wideparen{U(\mathfrak{g})}$-action factors through $\wideparen{U}_{\lambda}$, this makes $\mathrm{R}^j\Gamma(X, \mathcal{M})$ a coadmissible $\wideparen{U}_{\lambda}$-module for each $j\geq 0$ by \cite[Lemma 3.8]{Schneider03}.\\
We thus obtain Corollary \ref{intropartialfv}.

\end{document}